\documentclass[10pt]{amsart}
\usepackage{amssymb,amsmath,amscd,amsfonts,amsthm,mathrsfs}
\usepackage{enumitem}
\usepackage{verbatim}
\usepackage{a4wide}
\usepackage[latin1]{inputenc}

\newcommand{\N}{\mathbb{N}}
\newcommand{\Z}{\mathbb{Z}}

\newcommand{\R}{\mathbb{R}}
\newcommand{\C}{\mathbb{C}}

\newtheorem{theorem}{Theorem}[section]
\newtheorem{proposition}[theorem]{Proposition}
\newtheorem{lemma}[theorem]{Lemma}
\theoremstyle{remark}\newtheorem{remark}[theorem]{Remark}
\theoremstyle{definition}\newtheorem{definition}[theorem]{Definition}

\numberwithin{equation}{section}

\newcommand{\Ph}{{\mathcal P}}
\newcommand{\U}{{\mathcal U}}
\newcommand{\resto}{{\mathcal R}}
\newcommand{\Tr}{{\mathcal T}}
\def\norma#1{\left\| #1\right\|}
\def \p{\partial}

\def \rmi{{\rm i}}

\def \im{{\rm i}}
\def\uno{{\kern+.3em {\rm 1} \kern -.22em {\rm l}}}

\usepackage{multicol}

\makeindex

\title{On stability of standing waves of nonlinear Dirac equations}

\begin{document}

\author{Nabile Boussaid   and  Scipio Cuccagna}

\begin{abstract} We consider the stability problem for standing waves of
nonlinear Dirac models. Under a suitable definition of linear stability, and
under some restriction on the spectrum, we prove at the same time orbital and
asymptotic stability. We are not able to get the full result proved in
\cite{Cuccagna1} for the nonlinear Schr\"odinger equation, because of the strong
indefiniteness of the energy.
\end{abstract}

\maketitle

\section{Introduction}
In this paper we study the stability   of standing waves of a class of nonlinear
Dirac equations (NLDE). We assume that these standing waves are smooth,  have
exponential decay to 0 at infinity and that they are  smoothly dependent on a
parameter. We then partially characterize, under a number of further technical
hypotheses, their stability and their instability. We succeed partially in
transposing to NLDE  results proved for the  nonlinear Schr\"odinger equations
(NLS)  in \cite{Cuccagna1}  and in previous references. We recall   that
\cite{CazenaveLions,shatah,shatahstrauss,Weinstein,Weinstein2,
GrillakisShatahStrauss,GrillakisShatahStrauss2} contain a quite satisfactory
characterization of the orbital stability of standing waves of the NLS.
They do not apply to the Dirac equation, due to
the strong indefiniteness of the energy. In this paper we initiate
a theory of stability in the case of the NLDE, using ideas coming
from the theory  of asymptotic stability which are less sensitive to
indefiniteness of the energy. This idea is explored also in
\cite{PelinovskyStefanov} in a very special situation.

\subsection{The nonlinear Dirac equation}
\label{subsec:NLDE} We consider for $m>0$  a NLDE
\begin{equation}\label{Eq:NLDE}
\left\{\begin{matrix}
\rmi u_t -D_m u +g(u\overline{u} )\beta u=0\,\\
u(0,x)=u_0(x)
\end{matrix}\right. (t,x) \in \R \times \R^3
\end{equation}
where   $D_m=-\rmi \sum _{j=1}^3\alpha_j\partial_{x_j}+m\beta$,
with for $j=1,2,3$
\begin{equation*}\index{$\alpha_j$}\index{$\beta$}\index{$\sigma_j$}
\alpha _j=  \begin{pmatrix}  0 &
\sigma _j  \\
\sigma_j & 0
\end{pmatrix} \, , \quad \beta =  \begin{pmatrix}  I _{\C^2} &
0 \\
0 & -I _{\C^2}
 \end{pmatrix}
\, , \quad
\sigma _1=\begin{pmatrix}  0 &
1  \\
1 & 0
 \end{pmatrix} \, ,
\sigma _2= \begin{pmatrix}  0 &
\rmi   \\
-\rmi  & 0
 \end{pmatrix}  \, ,
\sigma _3=\begin{pmatrix}  1 &
0  \\
0 & -1
 \end{pmatrix}.
\end{equation*}
The unknown $u$ is   $\C^4$-valued. Given two vectors of $\C^4$,
$uv:=u\cdot v$ is the inner product in $\C^4$,
${v}^{\ast}$ is the complex conjugate, $u\cdot  {v}^{\ast}$ is the
hermitian product in $\C^4$, which we write as $uv^\ast =u\cdot
{v}^{\ast}$. We set
$ \overline{u}:=\beta {u}^{\ast}$,
so that $u\overline{u} =u\cdot \beta  {u}^{\ast}$\index{$u\overline{u} $}.
We have
\begin{equation*}\label{Eq:Commutation} \aligned & \alpha _j \alpha _\ell +
\alpha _\ell \alpha _j=2\delta
_{j\ell} I _{\C^4}\, , \quad   \alpha _j \beta + \beta \alpha
_j=0\, , \quad \beta ^2=I _{\C^4}.\endaligned
\end{equation*}
Thus the operator $D_m$ is self-adjoint on
$L^2(\R ^  3, \C ^4),$ with domain $H^1(\R ^3, \C ^4)$ and we have $
D^2_m=-\Delta+m^2$. The   spectrum is  $ \sigma   (D_m)=(-\infty,
-m]\cup [m,+ \infty),$ see \cite[Theorem 1.1]{Thaller}.

\subsection{State of the art}
The equation  in \S  \ref{subsec:NLDE} arises in   Dirac models used
to model either extended particles with self-interaction or
particles in space-time with geometrical
 structure. In the latter case, physicists have shown that a relativistic
 theory sometimes imposes a fourth order nonlinear potential
 (i.e., a cubic nonlinearity) such as the square of a quadratic form
on $\C^4$; see  \cite{Ranada} and the references therein.
The associated stationary equation is called the Soler model,
\cite{Soler}, as it was proposed by Soler to model the elementary
fermions.

In our study, we   assume the existence of stationary solutions as well as a number of properties like  the smooth dependence on a parameter, the smoothness and the fact that they are rapidly decaying. These are not well established properties.
 Stationary solutions were   actively studied in the last thirty years.
 References \cite{CazenaveVazquez,Merle,BalabaneCazenaveDouadyMerle,BalabaneCazenaveVazquez}  used a dynamical systems approach.
For the use of the variational structure of the
stationary equation, see
\cite{EstebanSere}. For an approach yielding stationary solutions
of the NLDE from solutions of the NLS, see   \cite{Ounaies,Guan}.

Turning to the question of stability,
\cite{StraussVazquez} discusses the  Soler model  within the framework of
\cite{shatahstrauss}, without attempting a proof.
Some partial results involving small standing waves obtained by bifurcation from linear ones, with $D_m$
replaced by $H:=D_m+V$  with $V$ a nice potential, are   in  \cite{Boussaid,Boussaid2}.
  \cite{Boussaid2}   shows that if a resonance
condition holds, there is a stable manifold outside
  which any initial condition leads to instability. If
the resonance condition is not fulfilled, the stability problem is
left open. The results we present here answer this question and can be used to clarify \cite{Boussaid}.
  \cite{KomechKomech} proves the existence of global
attractors in  a model involving a  Dirac equation coupled to
an harmonic oscillator.
The stability problem for the 1 dimensional NLDE is
discussed under very restrictive hypotheses   in \cite{PelinovskyStefanov} which reproduces for the 1 D NLDE an analogue of the result in \cite{SofferWeinstein2}.

\subsection{Hypotheses}
We   assume  the following hypotheses {\ref{Assumption:H1}}--{\ref{Assumption:H12}}.
\begin{enumerate}[label={\bf(H:\arabic{*})}, ref={\bf(H:\arabic{*})}]
\item\label{Assumption:H1} $g(0)=0$, $g\in C^\infty(\R,\R)$.

\item\label{Assumption:H2} There
exists an open interval $\mathcal{O}\subseteq (m/3,m)$\index{$\mathcal{O}$} such
that $
    D_m u-\omega u-g(u\overline{u} )\beta u=0
$ admits a  $C^\infty$  family of solutions $\omega \in
\mathcal{O}\to \phi _ {\omega } \in H^{k,\tau }(\R ^3) $ for any
$(k,\tau )$, see \eqref{weighted} for a definition. In spherical
coordinates $x_1= \rho \cos (\vartheta ) \sin (\varphi )$, $x_2= \rho
\sin (\vartheta ) \sin (\varphi )$, $x_3= \rho \cos  (\varphi )$,
    these standing waves are of the form
$$\index{$\phi _\omega$}\phi _\omega
(x)=\left [ \begin{matrix} a(\rho ) \left [ \begin{matrix} 1\\ 0  \end{matrix}
\right ]\\
\rmi b(\rho ) \left [ \begin{matrix}   \cos \varphi \\ e^{\rmi
\vartheta }\sin \varphi
\end{matrix} \right ]
\end{matrix} \right ] $$
with $a(\rho )$ and $b(\rho )$ real valued and satisfying the
following properties:
$$\aligned &
 a, b \in C^\infty([0,\infty), \R)
  \, ,  \quad
\forall \rho \geq 0,\quad  a^2(\rho )-b^2(\rho )\geq 0,
 \\& a ^{(j)} \text{ and } \,  b^{(j)} \, \text{decay exponentially at infinity for all $j$.}
    \endaligned
$$
Notice that  {  $\phi _\omega(-x)=\beta \phi_\omega(x)$ and $\phi
_\omega(-x_1,-x_2,x_3)=S_3\phi_\omega(x_1, x_2,x_3)$ with
$S_3:=\begin{pmatrix}\sigma_3&0\\0&\sigma_3\end{pmatrix}$.}

\item\label{Assumption:H3} Let $q(\omega )= \| \phi  _{\omega} \|
_{L^2}^2.$\index{$q(\omega )$} We
assume  $  q'(\omega ) \neq 0$ for all $\omega \in \mathcal{O}$.

\item\label{Assumption:H4} For any   $x\in \R ^3 $ we consider in
\eqref{Eq:NLDE} initial data s.t.  { $u_0(-x) =\beta u_0(
x )$ and $u_0(-x_1,-x_2,x_3)=S_3 u_0(x_1, x_2,x_3)$}.

\item\label{Assumption:H5} Let $\mathcal{H}_\omega$ be the linearized operator
around $e^{\rmi
t\omega}\phi_\omega$, see Sect. \ref{sec:spectrum}. We assume that $\mathcal{H}_\omega$
satisfies the definition of linear stability in Definition \ref{def:LinStab}.

\item\label{Assumption:H6}
Consider    {  $\mathbf{X}:=\{ (\Upsilon_1,\Upsilon_2) \in L^2
(\R^3,(\C^4)^2):  (\Upsilon_1( -x ),\Upsilon_2(-x_1,-x_2,x_3)\equiv (\beta
\Upsilon_1 (x ),-\beta \Upsilon_1 (x ),$  $  (\Upsilon_1(
-x_1,-x_2,x_3),\Upsilon_2(-x_1,-x_2,x_3)\equiv (S_3
\Upsilon_1 (x ),-S_3\Upsilon_1 (x ))\}$}, see  Sect.
\ref{sec:spectrum} and under
Lemma \ref{lem:SymmLin1}. $\mathbf{X}$
  is invariant for the action of
$\mathcal{H}_\omega $.  Consider the
 restriction of $\mathcal{H}_\omega $  in $\mathbf{X}$. Then
    $\mathcal{H}_\omega $
    has $2n$
   nonzero eigenvalues, counted with
    multiplicity, all contained in $(\omega  -m,m-
\omega )$. The positive eigenvalues can be listed as
$0 <\lambda _1(\omega )\le ...
\le \lambda _n(\omega )< m- \omega  $,
where we repeat each eigenvalue according to the multiplicity. For
each $\lambda _j(\omega )$, also $ -\lambda _j(\omega )$ is an
eigenvalue (this symmetry follows from\eqref{Eq:SymmLin3}).   There
are no other eigenvalues except for 0.

\item\label{Assumption:H7} The points
and $\pm (m- \omega ) $ and $\pm (m+ \omega ) $ are not resonances
for $\mathcal{H}_\omega $, see \eqref{Eq:resonance}-- \eqref{Eq:resonancebis} below.

\item\label{Assumption:H8} Suppose that $\lambda \in \R$ with $|\lambda | >m-\omega $ is a resonance for $\mathcal{H}_\omega $, that is one of the   following two equations admits a nontrivial solution:
\begin{eqnarray}\label{Eq:resonance} &
(1+R^{+}_{\mathcal{H}_{\omega ,0}}(\lambda )V_{\omega} ) u=0,
 \quad  u \in L^ {2,-\tau}(\R ^3, \C ^8) \text{ for some  $\tau>1/2$
  ;} \\& \label{Eq:resonancebis}
(1+R^{-}_{\mathcal{H}_{\omega ,0}}(\lambda )V_{\omega} ) u=0,
 \quad  u \in L^ {2,-\tau}(\R ^3, \C ^8) \text{ for some  $\tau>1/2$
   .}
\end{eqnarray}
 Then if $u$ satisfies  either  \eqref{Eq:resonance} or \eqref{Eq:resonancebis} we have $u \in L^ {2 }(\R ^3, \C ^8) $ and $\lambda$ is an eigenvalue of $\mathcal{H}_{\omega  }$.

\item\label{Assumption:H9} There are natural numbers $N_j$\index{$N_j$} defined
by the
property   $0<N_j\lambda _j(\omega )<
m-\omega < (N_j+1)\lambda _j(\omega )$.

\item\label{Assumption:H10} There is no multi index $\mu \in \mathbb{Z}^{k}$
with $|\mu|:=|\mu_1|+...+|\mu_k|\leq 2N_1+3$ such that $\mu \cdot
\lambda =m\pm  \omega$.

\item\label{Assumption:H11} If $\lambda _{j_1}<...<\lambda _{j_k}$ are $k$
distinct
  $\lambda$'s, and $\mu\in \Z^k$ satisfies
  $|\mu| \leq 2N_1+3$, then we have
$$
\mu _1\lambda _{j_1}+\dots +\mu _k\lambda _{j_k}=0 \iff \mu=0\ .
$$

\item\label{Assumption:H12} The nonlinear Fermi golden rule \eqref{eq:FGR} is
true.

\end{enumerate}
The space of functions satisfying {\ref{Assumption:H4}} is invariant by
\eqref{Eq:NLDE}. Except for the smoothness with respect to the parameter
$\omega$,  for some non-linearities {\ref{Assumption:H2}}  is a consequence of
\cite{EstebanSere}.   Continuous dependence on
$\omega$ for some examples is proved in \cite{Guan}.

\begin{remark} \label{rem:embei0}
  $2\omega$ is always an eigenvalue of $ \mathcal{H}_\omega $ in $
L^2(\R^3,\C^8)$,   \cite{Comech}.   So
for $3\omega  >m$ we have  $2\omega \in ( m-\omega, m+\omega) $ is an embedded eigenvalue.
We can avoid  it thanks to the symmetry {\ref{Assumption:H4}} since the
  eigenvectors do not belong to $\textbf{X}$, see Lemma \ref{lem:spectrum} below
and subsequent comments.  Reducing to the space  $\mathbf{X}$  reduces the number of parameters,   simplifying the problem.  The   parameters eliminated involve translation and orientation of the solutions.
For  work on
moving ground states of the NLS see
 \cite{Cuccagna6}.

\end{remark}

\begin{remark}
\label{rem:resonances}By {\ref{Assumption:H6}}--{\ref{Assumption:H8}}  there are no  resonances
   for the restriction of $\mathcal{H}_\omega $ in  $\mathbf{X}$.
{\ref{Assumption:H8}} is proved in the case of the NLS assumption
in
\cite{CuccagnaPelinovskyVougalter}.  In the case of the Dirac
system we are not able to prove it  except for
resonances contained in $(-\omega +m, \omega -m)$ or for large
energies. This is yet a consequence of the strong indefiniteness of
the energy of the Dirac system. We expect that
{\ref{Assumption:H8}} can be eliminated.

\end{remark}
\subsection{Main results}
The main result in this article is the following one.
\begin{theorem}\label{th:AsStab}
Suppose that $ \mathcal{O}\subset (m/3,m)$ and fix {  $k_0\ge 4$} , $k_0\in \Z$.
Pick $\omega_1\in \mathcal{O}$ and let $\phi_{\omega_1}(x)$ be a
standing wave of
 \eqref{Eq:NLDE}.
 Let $u(t,x)$ be a solution to \eqref{Eq:NLDE}. Assume
\ref{Assumption:H1}--\ref{Assumption:H12}.
Then,    there exist an $\epsilon_0>0$ and a $C>0$ such that for
any $\epsilon \in (0,\epsilon _0)$ and  for any $u_0$ with $\inf _{\gamma \in \R}
\|u_0-e^{\rmi \gamma  }\phi_ {\omega _1} \|_{H^{k_0}}<\epsilon  ,$
there exist $\omega_+\in \mathcal{O}$, $\theta\in C^1(
\mathbb{R};  \mathbb{R})$ and $h_+ \in H^{k_0}$ with $\| h_+\|
_{H^{k_0}}+|\omega_+-\omega_1|\le C \epsilon $ such that

$$
\lim_{t\to +\infty}\|u(t,\cdot)-e^{\rmi
\theta(t)}\phi_{\omega_{+}}-e^{-\rmi tD_m }h_+\|_{H^{k_0}}=0 .
$$
\end{theorem}

\begin{remark}
\label{rem:unnatural}   The constraint  $3\omega >m$ allows to exploit the nonlinear Fermi
Golden Rule (FGR) like for the NLS in
\cite{Cuccagna1} by circumventing  the strong indefiniteness of the Dirac system.
 We expect that
that   the hypothesis $3\omega >m$ can be eliminated. Specifically, it is used to guarantee that appropriate
multiples of the eigenvalues belong to portions of the spectrum where
there is no superposition   of the continuous spectrum of distinct coordinates.
This fact  and our results continue  to hold
if $3\omega <m$ and  $ (2N_j +1)
\omega > m$
for all $j=1,...,n$, see Remark \ref{rem:H3bis}.
\end{remark}

\begin{remark}
\label{rem:enindef}
Energy indefiniteness affects our methods because it   results in superposition   of the continuous spectrum of distinct coordinates.
There are two points
where our methods  are affected. The first is discussed in Remark
\ref{rem:unnatural}. The second point
is when we take {\ref{Assumption:H8}} as an hypothesis, see Remark \ref{rem:resonances}.
\end{remark}

\begin{remark}\label{rem:examples}
We do not know of examples of $g$ and
$\omega$ satisfying   our spectral assumptions. The situation is not very different
from the case of the NLS where the spectrum is unknown except in few cases. Rigorous analysis of examples is certainly a difficult open problem.
Like for the NLS, see \cite{ChangGustafsonNakanishiTsai}, one can consider   numerical analysis. For some example in 1--D see \cite{BerkolaikoComech,Chugunova}. For NLDE ,  by a non
relativistic bifurcation
argument, see \cite{Ounaies,Guan,EstebanLewinSere}, it is possible to extend what is known for the NLS equation to the
NLDE for $\omega$ close to $m$.
\end{remark}

\begin{remark}
\label{rem:cook up}  A partial justification of our hypotheses can be given using
bifurcation theory from  linear problem, see
\cite{Boussaid,Boussaid2,PelinovskyStefanov,SofferWeinstein,SofferWeinstein2}.
It is easy to prove the existence of  ``small solitons''
for which {\ref{Assumption:H1}}--{\ref{Assumption:H11}}   hold. In this
context the symmetry  {$\phi _\omega(-x)=\beta \phi_\omega(x)$
and} in
{\ref{Assumption:H4}} are unnecessary. In particular {\ref{Assumption:H6}}
holds
replacing $\mathbf{X}$ with  $L^2(\R^3,\C^8)$.  In the set up of small
solitons, {\ref{Assumption:H6}} and  {\ref{Assumption:H8}} are always true while
 {\ref{Assumption:H7}} and {\ref{Assumption:H9}}--{\ref{Assumption:H11}} hold
generically.  In the context of  small solitons it is easy to prove existence of
examples with just one eigenvalue $\lambda (\omega ) $ with $N=1$ for which
 \ref{Assumption:H12}   holds, in fact is generic,  thanks to   the easy form  the FGR
takes,  see formula (1.5) \cite{tsaiyau} for the NLS.
\end{remark}

\begin{remark}
\label{rem:hypotheses} Under {\ref{Assumption:H1}}--{\ref{Assumption:H11}},
we prove that, in an appropriate coordinate system, some key coefficients of
the    discrete modes equations are non negative.    If \ref{Assumption:H12}
holds, then these coefficients are positive and our proof tells us that
the continuous modes disperse and the discrete ones decay to 0. We
expect   \ref{Assumption:H12} to hold generically. Our proof extends with
minor modifications to the case of small solitons discussed in Remark \ref{rem:cook up}, where even the case of   just one eigenvalue $\lambda (\omega ) $ with $N=1$ (in fact even   the case with no eigenvalues) was an open problem.
\end{remark}

\begin{remark}
\label{rem:translation} One can envisage
extending   Theorem \ref{th:AsStab} to moving { {red} and
rotating} solitons. This would require
dropping \ref{Assumption:H4}. Then, since $3\omega >m$, one would face the
embedded eigenvalue $2\omega$. Problems arising from the possible failure of
the dispersive estimates in Sect. \ref{sec:dispersion} might  be solvable,
considering that \cite{CuccagnaPelinovskyVougalter} proves smoothing estimates
in the presence of embedded eigenvalues. However, looking  at the nonlinear  FGR
(which considers multiples of the eigenvalues),  we also have the problem that $4\omega > m+\omega$. So $4\omega$ belongs to a portion of the spectrum where
there is superposition of continuous spectrum of distinct components and the hypothesis $3\omega >m$ is of no help  to avoid this.
\end{remark}

Consider $\xi \in \ker (\mathcal{H} _{\omega } -\lambda _j(\omega
))$.
One of the requirements for linear stability in Definition
\ref{def:LinStab} is that if $\xi \neq 0$ then $\langle \xi ,
\Sigma _3 \xi ^* \rangle >0$. As it might seem artificial, we prove what
follows.
\begin{theorem}\label{th:orbital instability}
Suppose that $ \mathcal{O}\subset (m/3,m)$.
Pick $\omega \in \mathcal{O}$ and let $\phi_{\omega }(x)$ be a standing wave of
 \eqref{Eq:NLDE}.  Replace {\ref{Assumption:H5}} with the following assumption:
 \begin{enumerate}[label={\bf(H:\arabic{*}')},
ref={\bf(H:\arabic{*}')}]\setcounter{enumi}{4}
\item\label{Assumption:H5'}      We assume that $\mathcal{H}_\omega$
satisfies  all the conditions of  Definition \ref{def:LinStab} except for
condition (4) which we restate as follows. That is, we assume  that for any
eigenvalue $\lambda  >0$ the quadratic form $\xi \to \langle\xi    ,\Sigma_3\xi ^* \rangle  $ is non degenerate in $\ker ( \mathcal{H}_\omega  - \lambda    )   $. We assume that there exists at least one  eigenvalue $\lambda  >0$ such that
the quadratic form is non positive in
 $\ker ( \mathcal{H}_\omega  - \lambda    )   $.
 \end{enumerate}
 Assume
{\ref{Assumption:H1}}--{\ref{Assumption:H4}},
{\ref{Assumption:H5'}} and
{\ref{Assumption:H6}}--{\ref{Assumption:H12}}.
 Then $\phi_{\omega }(x)$ is orbitally unstable.
\end{theorem}

We will follow the argument developed in
\cite{Cuccagna1} for the NLS. The NLDE is harder than the
NLS. For example,   the regularity of $\phi _\omega $ in $\omega$ for  NLDE  is
unknown.
The classical methods to prove orbital
stability in
\cite{CazenaveLions,Weinstein2,GrillakisShatahStrauss,
GrillakisShatahStrauss2},  based as they are on the positivity of
certain functionals,  do not apply to NLDE because of the strong
indefiniteness of the energy. We already mentioned some initial
results for the Dirac equation   in
\cite{Boussaid,Boussaid2,PelinovskyStefanov}. Like in these
articles,
  we exploit the  dispersive properties of the linearizations,
  adapting  the methods
  used to prove asymptotic stability for the NLS
  initiated in
\cite{SofferWeinstein,SofferWeinstein2,BuslaevPerelman,BuslaevPerelman2} and
developed by a substantial number of authors, see the references in
\cite{Cuccagna1}.
 One of the difficult issues
for the NLS, was, and still is,  to prove that  the energy of the discrete modes associated
to the eigenvalues in {\ref{Assumption:H6}} leaks   either in the radiation part or in the standing wave.
The solution to this problem was initiated in \cite{BuslaevPerelman2}, where the
eigenvalues
are close to the continuous spectrum,  and solved in quite   general form in
\cite{Cuccagna1}, see also \cite{BambusiCuccagna,Cuccagna4}.
We recall that there is leaking because, in appropriate coordinates, the
nonlinear interaction between discrete and continuous modes yields
some dissipative coefficients in the equations of the discrete modes, in a way
similar to the classical Fermi Golden Rule (FGR). This phenomenon was first
established in special cases for the NLS in  \cite{BuslaevPerelman2}. The coefficients were identified  generally  in \cite{cuccagnamizumachi}, which built on
\cite{zhousigal}. Their dissipative nature was established in \cite{Cuccagna1}.
We refer to
\cite{Cuccagna1} for a
discussion of the fact that it is essential to exploit the hamiltonian
structure of the equation.  For work \cite{Cuccagna6}
extending the result in \cite{Cuccagna1}
to moving ground states    see Remark \ref{rem:embei0}.

In this article we follow  the same framework of
\cite{Cuccagna1} obtaining similar results. In particular the key coefficients
in the discrete modes equations are shown to be quadratic forms, see Lemma \ref{lemma:FGR81}. By the energy indefiniteness, see Remark \ref{rem:enindef},  the sign of these quadratic forms is unclear. We can overcome this uncertainty
if we assume $3\omega>m$, since in this case   there is no superposition of continuous spectrum of distinct components and the quadratic forms are easily
proved to be non negative.

We need to develop some of the
linear
theory of dispersion, which in the case of the NLS had been developed in the
course of a decade, see \cite{Cuccagna5,CuccagnaPelinovskyVougalter}. Key to
dispersion theory is the proof of
smoothing estimates for Schr\"odinger operators with magnetic
potentials in \cite{ErdoganGoldbergSchlag}.
There are two points       in the article where the strong indefiniteness
of the energy interferes   with our method and they are discussed in Remark \ref{rem:enindef}.
  We expects these difficulties to be technical and solvable.   Notice that in in \cite{Boussaid,Boussaid2,PelinovskyStefanov}
 these difficulties do not arise because   smallness of  solitons
  yields  absence of
resonances   for free and
 the FGR  is not addressed because of their restrictive hypotheses.

The instability result in Theorem \ref{th:orbital instability}
arises from our desire to justify  Assumption \ref{Assumption:H5}  in
our definition
of linear stability, see Definition \ref{def:LinStab}.  The proof
of Theorem \ref{th:orbital instability} is similar to
\cite{Cuccagna2}. That is, we show that orbital stability implies
asymptotic stability, and we then show that this is incompatible
with {\ref{Assumption:H5'}}. All the proofs are conditional on
{\ref{Assumption:H12}}, that is that a certain non negative
quantity is actually positive. Presumably this is true generically.

\subsection{Notation and preliminaries} We consider spaces
\begin{equation}\label{weighted}\index{$H^{k,s}$}
H^{k,s}  (\R^3,\C^4)=\left\{ f\in {\mathcal S}'(\R^3),\,
\|\langle x\rangle^s\langle \nabla \rangle^k f\|_2
<\infty\right\}
\end{equation}
for $s,k\in\R$  with  norm
$
\|f\|_{ H^{k,s} }= \|\langle x\rangle^s\langle
\nabla\rangle^k f\|_2.
$ Sometimes we will write $H^{k,s} _x$ to emphasize the independent
variable $x$.
If~$k=0$, we write~$L^{2,s} $ instead of~$H^{0,s}$.

For~$k\in\R$ and~$1\leq p,q\leq\infty$, the Besov
space~$B^k_{p,q}(\R^3,\C^d)$\index{$B^k_{p,q}$} is the space of all   tempered
distributions $f\in {\mathcal S}'(\R^3,\C^d)$  such that
\[
   \|f\|_{B^k_{p,q}}=  (\sum_{j\in\N}2^{jkq} \|\varphi_j *
f\|_p^q )^{\frac{1}{q}}<+\infty
\]
with~$\widehat{\varphi} \in {\mathcal C}^\infty_0(\R^n\setminus
\left\{0\right\})$ such that~$\sum_{j\in
\Z}\widehat{\varphi}(2^{-j}\xi)=1$ for all~$\xi
\in\R^3\setminus\left\{0\right\}$,
$\widehat{\varphi}_j(\xi)=\widehat{\varphi}(2^{-j} \xi)$ for all
$j\in\N^*$ and for all~$\xi \in\R^3$, and
$\widehat{\varphi_0}=1-\sum_{j\in \N^*}\widehat{\varphi}_j$. It is
endowed with the norm $   \|f\|_{B^k_{p,q}}$.

For $A$ a closed operator on a Hilbert space $X$  we will  set
$R_A(z):=(A-z)^{-1}$   for any $z$ in the resolvent
set of $A$.

\subsection{Structure of the article}

The paper is organized as follows.
In Sections  \ref{Sec:setup}--\ref{sec:modulation}, we
  study of the linearization of \eqref{Eq:NLDE} at the stationary solution,
  we give some information on the spectrum and on symmetries of the linearization,  we define the notion of linear
  stability and we introduce an appropriate
coordinate system related to the spectral decomposition of the linearized operator.
In Sect. \ref{sec:Dispersive} and in the Appendix
we discuss estimates  on such  operators.
In Sect. \ref{Sec:Hamiltonian}
 we discuss we reframe the system in a hamiltonian form. In Sect. \ref{section:Darboux} we look for canonical coordinates. In Sect. \ref{sec:Canonical} we reformulate the system in these coordinates. In Sect. \ref{section:Birkhoff} we apply the method of Birkhoff normal forms. The proofs
 of the analogous parts in \cite{Cuccagna1} work almost unaltered.  Having
 chosen  an appropriate coordinate system, in Sect \ref{sec:dispersion}
 we begin to prove nonlinear dispersion, in particular estimating the
 continuous modes. We finish with the closing up of the estimates
 in Sect. \ref{sec:FGR} where we prove the Fermi Golden Rule. Specifically  we
 prove that appropriate coefficients are quadratic forms and that  for $\omega >m/3$
 they are non negative. Finally, under hypothesis  {\ref{Assumption:H12}},
 which presumably holds generically, we close up the inequalities and we conclude the
 proof of asymptotic stability, Theorem \ref{th:AsStab}. We also prove
Theorems \ref{th:orbital instability} using similar ideas. In the Appendix we
proves smoothing estimates and scattering estimates.

\section{Set up and symmetries}\label{Sec:setup}

\subsection{Set up}

Since our ambient space is $H^{k_0} (\R ^3, \C ^4)$ with $k_0\ge 4$ and so in
particular $k_0>3/2$, under {\ref{Assumption:H1}} the functional $u\to
g(u\overline{u}) \beta  u$ is locally Lipschitz and   \eqref{Eq:NLDE} is locally
well posed, see pp. 293--294 volume III  \cite{Taylor}. Consider the solution
$u(t,x)$ of \eqref{Eq:NLDE}. Then by
{\ref{Assumption:H4}} we have  $u(t,-x )=\beta u(t,x )$ and
$u(t,-x_1,-x_2,x_3)=S_3 u(t,x )$. We write
the ansatz
\begin{equation}\label{Eq:ansatz}
u(t,x) = e^{\rmi  \vartheta (t)} (\phi_{\omega (t)} (x)+
r(t,x)) .
\end{equation}
Inserting \eqref{Eq:ansatz} in    \eqref{Eq:NLDE}  we get from the definition of
$\phi_\omega$
\begin{equation} \label{Eq:syst1}
\aligned &
  \rmi r_t  =
 D_m r -\omega (t) r-
g ( { \phi} _{\omega (t)}  \overline{\phi} _{\omega (t)} )\beta r -g
^\prime (  { \phi} _{\omega (t)}  \overline{\phi} _{\omega (t)} )
(r\overline{\phi} _{\omega (t)} ) \beta \phi _{\omega (t)} \\&-
 g ^\prime (  { \phi} _{\omega (t)}  \overline{\phi} _{\omega (t)} )
({\phi} _{\omega (t)}  \overline{r} )  \beta\phi _{\omega (t)}
 + (\dot \vartheta (t)+\omega (t)) (\phi _{\omega (t)}+r) -\rmi \dot \omega
(t)
\partial _\omega \phi   _{\omega (t)}
 + n(r ),
\endaligned
\end{equation}
where $n(r)=O(r^2)$ is defined by
\begin{equation*}
\begin{aligned}
n(r ):&= g ( ({ \phi}
_{\omega (t)}+r) \overline{\phi_{\omega (t)}+r}
)\beta(\phi_{\omega(t)}+ r) -g ( { \phi} _{\omega (t)}
\overline{\phi_{\omega (t)}} )\beta\phi_{\omega(t)}\\ & -g ^\prime (  {
\phi} _{\omega (t)} \overline{\phi} _{\omega (t)} )
(r\overline{\phi} _{\omega (t)} ) \beta \phi _{\omega (t)} -
 g ^\prime (  { \phi} _{\omega (t)}  \overline{\phi} _{\omega (t)} )
({\phi} _{\omega (t)}  \overline{r} )  \beta\phi _{\omega
(t)}.\end{aligned}
\end{equation*}

We denote by $C:\C^4\to\C^4$\index{$C$, charge conjugation}   the charge
conjugation operator
$
u^c:=Cu:=\rmi \beta \alpha_2 u^\ast .$ We have
$\alpha_j C = C \alpha_j $ and $ \beta C =-C
\beta$ for all $j \in \{1,2,3\}$,  \cite[Sect. 1.4.6]{Thaller}. Since it is anti-linear, for any $u\in\C^4$, $C(u^\ast)=(Cu)^\ast$.

We  state without proof the following simple lemma.
\begin{lemma}\label{lem:C}
For any vector $v\in \C ^4$ we have $C^2v=v$.
  Moreover   we have:
   \begin{equation*} \label{Eq:C1}
\aligned & C (\rmi v)= -\rmi v^c  , \quad
v\overline{v}=-Cv\overline{Cv}, \quad     ,  \quad C (\beta v)= -\beta v^c , \quad C (D_m w)= -D_m w^c.
\endaligned
\end{equation*}
For   $u_0$ satisfying
 {\ref{Assumption:H4}} we
have  { $u^c_0(-x )=-\beta u^c_0( x )$ and
$u^c_0(-x_1,-x_2,x_3)=-S_3 u^c_0( x )$.}
\end{lemma}
Applying $-C$ to \eqref{Eq:syst1}, we obtain
\begin{equation*}
\label{Eq:syst1Con} \aligned &
  \rmi r_t^c  =
  D_m r^c +\omega (t) r^c-
g ( { \phi}_{ \omega (t)}   \overline{\phi  } _{ \omega (t)} )\beta
r^c +g ^\prime (  { \phi }  _{\omega (t)} \overline{\phi
 } _{\omega (t)} ) (r^c\overline{\phi ^{c}}_{ \omega (t)} )
\beta \phi _{ \omega (t)}^{c} \\& +
 g^\prime (  { \phi}  _{\omega (t)}  \overline{\phi  } _{\omega (t)} )
({\phi} _{ \omega (t) }^{c}  \overline{r^c} ) \beta \phi _{ \omega
(t)}^{c}
 - (\dot \vartheta (t)+\omega (t)) ( \phi _{ \omega (t)}^{c}+r^c) -\rmi \dot
\omega
(t)
\partial _{ \omega} \phi   _{ \omega (t)}^{c}
-C
  n(r  ).\endaligned
\end{equation*}
 We set
\begin{equation}\label{Eq:linearization}\index{$U$}\index{$R$}\index{
$\Phi_\omega$}\index{$\mathcal{H} _{\omega }$} \index{$\mathcal{H} _{\omega,0
}$} \index{$V _{\omega }$} \aligned &  U=
\begin{pmatrix} u
  \\
  u^c
 \end{pmatrix} \quad ,\quad
  R=
\begin{pmatrix} r
  \\
  r^c
 \end{pmatrix} \quad , \quad \Phi_\omega
=
\begin{pmatrix} \phi _{\omega } \\
\phi_{ \omega }^c
 \end{pmatrix} \quad , \quad N(R)
=  \begin{pmatrix} n(r) \\ -C n(r ) \end{pmatrix}\quad ,
\\&   \mathcal{H} _{\omega }=\mathcal{H} _{\omega, 0}+V_{\omega}\quad ,\quad
 \mathcal{H}_{\omega ,0}=  \begin{pmatrix}D_m - \omega&0\\
0&D_m+\omega\end{pmatrix}\quad , \\&
V_\omega= g(\phi_{\omega
}\overline{\phi}_{\omega }) \beta + g^\prime (\phi_{\omega
}\overline{\phi}_{\omega
 })
\begin{pmatrix}
-(  \beta  {\phi}_{\omega  }^* \quad )\beta\phi _{\omega  }& (\beta
(\phi_{ \omega  }
^c)^*\quad  )\beta\phi _{\omega  }\\
 -(\beta \phi_{ \omega
 }^* {\quad  })\beta \phi_{ \omega  }
^c&
 (  \beta (\phi_{ \omega  }
^c)^*\quad )\beta \phi_{ \omega  } ^c
\end{pmatrix}\quad
\endaligned \end{equation}
where the first $\beta $  in the last line is meant in the sense of
\eqref{eq:symmR} below and where   $(\phi\quad)$
stands  for
 the map $r\mapsto \phi r$. Then  we have:
\begin{equation} \label{Eq:syst2}
\rmi \dot{R}=\mathcal{H}_\omega R
+ (\dot \vartheta (t)+\omega (t))(\Sigma_3\Phi_\omega +\Sigma_3R)
 -\rmi\dot\omega\p_\omega\Phi_\omega+N(R),
\end{equation}
$$\index{$\Sigma_j$} \text{where }
\Sigma _1=\begin{pmatrix}  0 &
I_{\C^4}  \\
I_{\C^4}  & 0
 \end{pmatrix} \, ,
\Sigma _2=\begin{pmatrix}  0 &
\rmi I_{\C^4}  \\
-\rmi I_{\C^4}  & 0
 \end{pmatrix} \, ,
\Sigma _3=\begin{pmatrix}  I_{\C^4}  &
0  \\
0 & -I_{\C^4}
 \end{pmatrix} .
$$

Notice that by {\ref{Assumption:H4}} and Lemma \ref{lem:C} we have
for $\Upsilon(x)\in\{\Phi _{\omega}(x), R(t,x)\}$
\begin{align}\label{eq:symmR}
&\Upsilon ( -x )=\beta \Sigma _3 \Upsilon ( x ) \text{ where }  \beta
=\begin{pmatrix} \beta &
0\\
 0 &
\beta
\end{pmatrix}\\
\mbox{ and }\label{eq:symmR2}
&{ {\Upsilon ( -x_1,-x_2,x_3)=S_3 \Sigma _3 \Upsilon ( x ) \text{ where }
S_3
=\begin{pmatrix} S_3 &
0\\
 0 &
S_3
\end{pmatrix} }}.
\end{align}

\subsection{Symmetries}
\label{sec:symm}

We consider now the bilinear map
\begin{equation} \label{eq:Inner Product}
\left \langle \begin{pmatrix}r _1\\ r_2 \end{pmatrix} ,
\begin{pmatrix}s _1 ^*\\ s_2 ^* \end{pmatrix} \right \rangle =
\int _{\R ^3} ( r _1 \cdot  s _1^* +r _2 \cdot s _2 ^* ) dx.
\end{equation}
By $\mathcal{H}_\omega^*$ we denote the adjoint of $
\mathcal{H}_\omega $ with respect to  this inner product. We have:

\begin{lemma}
  \label{lem:SymmLin} We have
\begin{eqnarray}
\label{Eq:SymmLin1}
&\mathcal{H}_\omega^*=\Sigma_3\mathcal{H}_\omega\Sigma_3\ ,\\
&\label{Eq:SymmLin3}
\mathcal{H}_\omega=-C\Sigma_1\mathcal{H}_{\omega}C\Sigma_1
\text{  where }
C=\begin{pmatrix}
C&0\\
0& C
\end{pmatrix} \  , \\
&\label{Eq:SymmLin31} V_\omega (-x )=\beta \Sigma_3V_\omega ( x )\beta
\Sigma_3 \text{  with $ \beta$ in the sense of \eqref{eq:symmR}} \\
&\label{Eq:SymmLin311} { {V_\omega (-x_1,-x_2,x_3)=S_3 \Sigma_3V_\omega ( x )S_3
\Sigma_3 \text{  with $ S_3$ in the sense of \eqref{eq:symmR2}} }}\ .
\end{eqnarray}
\end{lemma}
{\it Proof.}  First of all, \eqref{Eq:SymmLin1}--\eqref{Eq:SymmLin3}
hold  with $\mathcal{H}_\omega$ replaced by $\mathcal{H}_{\omega
,0}$. It remains to check them with $\mathcal{H}_\omega$ replaced
by $V_\omega$. We have  $ V_\omega^*= \Sigma_3V_\omega\Sigma_3 $ by
\begin{equation} \label{Eq:SymmLin5}\Sigma_3\begin{pmatrix}
-(  \beta  {\phi}_{\omega  }^* \quad )\beta\phi _{\omega  }& (\beta
(\phi_{ \omega  }
^c)^*\quad  )\beta\phi _{\omega  }\\
 -(\beta \phi_{ \omega
 }^* {\quad  })\beta \phi_{ \omega  }
^c&
 (  \beta (\phi_{ \omega  }
^c)^*\quad )\beta \phi_{ \omega  } ^c
\end{pmatrix} \Sigma_3 =
\begin{pmatrix}
-(  \beta  {\phi}_{\omega  }^* \quad )\beta\phi _{\omega  }&
-(\beta (\phi_{ \omega  }
^c)^*\quad  )\beta\phi _{\omega  }\\
  (\beta \phi_{ \omega
 }^* {\quad  })\beta \phi_{ \omega  }
^c&
 (  \beta (\phi_{ \omega  }
^c)^*\quad )\beta \phi_{ \omega  } ^c
\end{pmatrix}
\end{equation}
and from the fact that the matrix in rhs\eqref{Eq:SymmLin5} is the
adjoint of the matrix in  lhs\eqref{Eq:SymmLin5}.
\eqref{Eq:SymmLin3} holds  with $\mathcal{H}_\omega$ replaced by
$\mathcal{H}_{\omega ,0}$ by Lemma \ref{lem:C}. We have

\begin{equation} \label{Eq:SymmLin6}\begin{aligned}&
C \Sigma_1
\begin{pmatrix}
-(  \beta  {\phi}_{\omega  }^* \quad )\beta\phi _{\omega  }& (\beta
(\phi_{ \omega  }
^c)^*\quad  )\beta\phi _{\omega  }\\
 -(\beta \phi_{ \omega
 }^* {\quad  })\beta \phi_{ \omega  }
^c&
 (  \beta (\phi_{ \omega  }
^c)^*\quad )\beta \phi_{ \omega  } ^c
\end{pmatrix}  =\\& -C
\begin{pmatrix}
- (\beta (\phi_{ \omega  } ^c)^*\quad  )\beta\phi _{\omega  } ^c&
  (  \beta  {\phi}_{\omega  }^* \quad )\beta\phi _{\omega  }^c \\ - (\beta
(\phi_{ \omega  } ^c)^*\quad  )\beta\phi _{\omega  } & (  \beta
{\phi}_{\omega  }^* \quad )\beta\phi _{\omega  }
\end{pmatrix}  \Sigma_1 =\\& -
\begin{pmatrix}
 (\beta (\phi_{ \omega  } ^c)^*\quad  )^* \beta\phi _{\omega  }  &
  -(  \beta  {\phi}_{\omega  }^* \quad )^* \beta\phi _{\omega  }
   \\   (\beta
(\phi_{ \omega  } ^c)^*\quad  )^* \beta\phi _{\omega  }^c & -(
\beta {\phi}_{\omega  }^* \quad )^*\beta\phi _{\omega  } ^c
\end{pmatrix}  \Sigma_1 .
\end{aligned}
\end{equation}
We have  for $v\in \C ^4$ \begin{equation*}
\label{Eq:SymmLin7}\begin{aligned}& (\beta (\phi _\omega ^c)^*v )^*
 = \beta (\rmi \beta \alpha _2
{\phi}_{ \omega  }^* ) v ^*  = -\beta {\phi}_{ \omega }^* C(v ),
\\&
  (\beta \phi_{ \omega
 }^* v)^* =\beta \phi_{ \omega
 } v ^*= -\beta ( \rmi \beta \alpha _2  \phi_{ \omega
 } ) (\rmi   \beta \alpha _2 v^*)=- \beta ({\phi}_{ \omega  }^{c})^*
 C(v) .\end{aligned}\nonumber
\end{equation*}
Then
\begin{equation*}
\label{Eq:SymmLin8}\begin{aligned}& \text{rhs\eqref{Eq:SymmLin6}} =-
\begin{pmatrix}
-(  \beta  {\phi}_{\omega  }^* \quad )\beta\phi _{\omega  }& (\beta
(\phi_{ \omega  }
^c)^*\quad  )\beta\phi _{\omega  }\\
 -(\beta \phi_{ \omega
 }^* {\quad  })\beta \phi_{ \omega  }
^c&
 (  \beta (\phi_{ \omega  }
^c)^*\quad )\beta \phi_{ \omega  } ^c
\end{pmatrix} C \Sigma _{1}
.\end{aligned}
\end{equation*}
This yields \eqref{Eq:SymmLin3}. The proof of \eqref{Eq:SymmLin31}
goes as follows. Using $\phi (-x)= \beta \phi ( x)$ and $\phi ^c
(-x)= -\beta \phi ^c( x)$, where we omit the subindex $\omega$, we
have
\begin{equation}\label{Eq:SymmLin9}
 \begin{aligned}& V(-x)\beta \Sigma _{3} =
 g(\phi (x)\overline{\phi} (x)) \Sigma _{3}
+g ^\prime (  { \phi }  _{\omega (t)} \overline{\phi
 } _{\omega (t)} )
\begin{pmatrix}
-(     {\phi} ^* (x)\quad ) \phi (x) &
-(  (\phi ^c  (x))^*\quad  ) \phi (x) \\
  (  \phi ^*(x) {\quad  })\phi ^c  (x)&
 (  (\phi ^c  (x))^*\quad )\phi ^c  (x).
\end{pmatrix} \beta \Sigma _{3} \\& =
g(\phi (x)\overline{\phi} (x)) \Sigma _{3} +
g ^\prime (  { \phi }  _{\omega (t)} \overline{\phi
 } _{\omega (t)} )
\begin{pmatrix}
-(   \beta  {\phi} ^* (x)\quad ) \phi (x) &
 ( \beta (\phi ^c  (x))^*\quad  ) \phi (x) \\
  (  \beta \phi ^*(x) {\quad  })\phi ^c  (x)&
- ( \beta (\phi ^c  (x))^*\quad )\phi ^c  (x).
\end{pmatrix}.
\end{aligned}
\end{equation}
Similarly
\begin{equation*}\label{Eq:SymmLin92}
 \begin{aligned}& \beta \Sigma _{3}V( x)
 =g(\phi (x)\overline{\phi} (x)) \Sigma _{3} +g ^\prime (  { \phi }  _{\omega
(t)} \overline{\phi
 } _{\omega (t)} )
 \beta \Sigma _{3}
 \begin{pmatrix}
-(   \beta  {\phi} ^* (x)\quad )\beta \phi (x) &
 (  \beta(\phi ^c  (x))^*\quad  )\beta \phi (x) \\
  -(  \beta\phi ^*(x) {\quad  })\beta \phi ^c  (x)&
 ( \beta (\phi ^c  (x))^*\quad )\beta \phi ^c  (x).
\end{pmatrix}  \\& =
\text{second line of \eqref{Eq:SymmLin9}}.
\end{aligned}
\end{equation*}
The last two formulas yield \eqref{Eq:SymmLin31}.   Identity
\eqref{Eq:SymmLin311} is proved similarly.\qed

\begin{lemma}
 \label{lem:SymmLin1} For $\mathbf{A}$  the operator in $L^2(\R
^3,\C ^8)$ defined by $(\mathbf{A}X)(x ):=\beta \Sigma_3X(-x )$, then
$\mathbf{A}^2=Id$, $\mathbf{A}$ is selfadjoint and  $\sigma (\mathbf{A}) =\{
1,-1 \}$.
   We have $[ \mathbf{A}, \mathcal{H}_\omega ]= [ \mathbf{A}, \mathcal{H}_{
\omega 0} ] =0$.

  For $\mathbf{B}$  the operator in $L^2(\R ^3,\C ^8)$
defined by $(\mathbf{B}X)(x ):=S_3\Sigma_3X(-x_1,-x_2,x_3)$, then
$\mathbf{B}^2=Id$,
$\mathbf{B}$ is selfadjoint and  $\sigma (\mathbf{B}) =\{ 1,-1 \}$.
   We have $[ \mathbf{B}, \mathcal{H}_\omega ]= [ \mathbf{B}, \mathcal{H}_{
\omega 0} ] =0$.

Moreover, $[\textbf{A},\textbf{B}]=0$. 
 \end{lemma}
\proof  The first sentence is elementary. The second follows by $[ \mathbf{A},
A ]=0$ for $A= D_m, \Sigma _3  $ (straightforward)  and $A=V_\omega $, from
\eqref{Eq:SymmLin31}. The   statements for $\textbf{B}$ are obtained similarly. $[\textbf{A},\textbf{B}]=0$
is elementary.

\subsection{Energy and  charge}
\label{sec:Energy}

We have the following    elementary result.
\begin{lemma}
  \label{lem:systU}   Let $U^T=  (u,C{u})  $. Set for $G(0)=0$ and $G'(s)=g (s)
$
 \begin{equation*}
\label{eq:energyfunctional}\index{$E$}\index{$E_K$}\index{$E_P$}\index{$Q$}
\begin{aligned}&
 E(U)=E_K(U)+E_P(U)\, , \,
E_K(U)= \int _{\R ^3}
  (D_m u)  u^* dx  \, , \,
E_P(U)= -
 \int _{\R ^3}G( u \overline{u}) dx, \\& Q(U)=  \int _{\R ^3}u
 {u}^* dx .\end{aligned}
\end{equation*}
Then $E(U) $ and  $Q(U) $  are invariants of motion for
\eqref{Eq:NLDE}  and we have
\begin{equation}\label{eq:invariants}\begin{aligned} & E(U)=
\frac{1}{2}\langle  \im \beta \alpha _2  \Sigma _3\Sigma _1D_m U,
U\rangle  -
 \int _{\R ^3}G\left ( \frac{1}{2}U \cdot \im \alpha _2\Sigma _3 \Sigma
_1U\right ) dx   \\& Q(U)=   \frac{1}{2}\langle U, \im \beta \alpha _2  \Sigma
_1 U\rangle, \end{aligned}
\end{equation}  where for
$\langle\cdot \,,\cdot \rangle$ see \eqref{eq:Inner Product}.
  $U$ satisfies system \begin{equation}\label{eq:NLSvectorial} \im
\dot U =
 \im \beta \alpha _2 \Sigma _3 \Sigma _1 \nabla E (U).
\end{equation}
\end{lemma}
\begin{proof} For any symmetric
operator $A$ acting on $L^2(\R^3,\C^4)$  with the domain invariant
by $C$ and anticommuting to $C$ and any $u\in
D(A)$,
\[\begin{aligned}
u \cdot (A  u)^*&=\frac{u \cdot (A  u)^\ast+ u^\ast \cdot (A  u)}{2}
=\frac{u \cdot \im  \beta \alpha _2 C A  u+\im  \beta \alpha _2 Cu \cdot A
u}{2}\\
&= \frac{-u \cdot \im  \beta \alpha _2  A   u^c+\im  \beta \alpha _2 u^c \cdot A
u}{2}
= \frac{-u \cdot \im  \beta \alpha _2  A   u^c+u^c \im  \beta \alpha _2  \cdot A
u}{2}\\
&=\frac{ \im }{2} U \cdot \beta \alpha _2 \Sigma _3 \Sigma _1 A U
\text{, where we write $A$ for
$
A=
\begin{pmatrix}
A & 0\\
 0 & A
\end{pmatrix}.$}
\end{aligned}\]
 If $A$   commutes with $C$, then a similar calculation shows
$
\langle u, (A  u)^* \rangle = \frac{ \im }{2} \langle U , \beta \alpha _2 \Sigma
_1 A U \rangle
.$
These identities  for $A=D_m$, $A=\beta$ or $A=I$ prove  the lemma.
\end{proof}

\section{Spectrum and linear stability}
\label{sec:spectrum}

From now  on we restrict attention to    { $\mathbf{X}   =\{ \Upsilon
\in L^2 (\R^3,\C ^8): \Upsilon ( -x )\equiv \beta \Sigma _3 \Upsilon (
x ),\,\Upsilon ( -x_1,-x_2,x_3)\equiv S_3 \Sigma _3 \Upsilon ( x )\}:=\ker
(\mathbf{A} -Id)\cap \ker
(\mathbf{B} -Id)\subset L^2(\R ^3,\C ^8)$}. It is invariant by
$\mathcal{H} _{\omega ,0}$ and  $\mathcal{H} _{\omega  }$, see Lemma
\ref{lem:SymmLin1}.  We consider the spectrum
\[
\sigma(\mathcal{H}_\omega)= \left\{\lambda \in \C, \, \mathcal{H}_\omega-\lambda
Id : \mathbf{X} \cap H^1(\R^3,\C^8)\mapsto \mathbf{X} \mbox{ is not invertible}\right\}
\]
We summarize  what we know about the spectrum.
\begin{lemma}\label{lem:spectrum}
\begin{itemize}
\item[(1)] For the    essential
 spectrum we have,
 $\sigma_{\rm
ess}(  \mathcal{H} _\omega )=(-\infty,  \omega  -m]\cup
 [m- \omega  ,+
 \infty)$.

\item[(2)] For each $z\in \sigma _p (  \mathcal{H} _\omega )$ the corresponding
generalized eigenspace $N_g(  \mathcal{H}_\omega -z) $\index{$N_g$,
generalized kernel} has finite
dimension.

\item[(3)] If $z\in  \sigma   (  \mathcal{H} _\omega )$ then also
$-z\in  \sigma   (  \mathcal{H} _\omega )$.

\item[(4)] For the generalized kernel we have  $N_g(\mathcal{H} ^\ast
_\omega )\supseteq \{ \Phi _\omega , \Sigma_3\partial _\omega \Phi
_\omega \} .$

\item[(5)] $\partial _\omega \| \phi _\omega  \| _{2}^2\neq 0$
implies  that there are no $v$ such that $\mathcal{H}_\omega v=
\p_\omega\Phi_\omega$.
\item[(6)] We have $ \mathcal{H} _\omega Y =-2\omega Y$   and $ \mathcal{H} _\omega C\Sigma _1Y = 2\omega C\Sigma _1Y$   for
$
  Y:= \begin{pmatrix}
 \alpha _1\alpha _2\alpha _3 \beta\phi_\omega\\
0\end{pmatrix}.$ We have  $\Upsilon ( -x)\equiv -\beta \Sigma _3 \Upsilon ( x)$
for $\Upsilon =Y,   C\Sigma _1Y$.

\end{itemize}
\end{lemma}

\begin{proof} We have that (1) and (2) are consequences of the above discussion.
If
$z\in \sigma_{\rm
ess}(  \mathcal{H} _\omega )$ then (3) is a consequence
of (1). If $z$ is an eigenvalue, then (3) is a consequence of
\eqref{Eq:SymmLin3}. (4)  is a consequence of   $N_g(\mathcal{H}
_\omega )\supseteq \{ \Sigma_3 \Phi _\omega ,  \partial _\omega \Phi
_\omega \}$ which can be seen as follows. By the gauge invariance of
the nonlinearity, $
G((e^{\rmi\theta}u)\overline{(e^{\rmi\theta}u)})=G(u\overline{u}), $
where $G$ is a primitive of $g$, we have
\begin{equation*}
 \mathcal{H}_\omega\begin{pmatrix}
\rmi\phi_\omega\\
C\rmi\phi_\omega
\end{pmatrix} =0
\text{ or } \mathcal{H}_\omega\Sigma_3\Phi_\omega=0.
\end{equation*}
Then differentiating  with
respect to $\omega$ \eqref{Eq:NLDE} and taking its image by $C$, we obtain $
\mathcal{H}_\omega\p_\omega\Phi_\omega=-\Sigma_3\Phi_\omega .$
(5) follows by the following argument, if we assume existence of
$v$ s.t. $\mathcal{H}_\omega v= \p_\omega\Phi_\omega$,
\begin{equation*} \label{eq:GenKer1}\begin{aligned} & 0=\langle  v,  (
\mathcal{H}_\omega ^*\Phi_\omega
)^*  \rangle =\langle  \p_\omega\Phi_\omega ,\Phi_\omega  ^*\rangle
= \langle  \p_\omega\phi_\omega ,\phi_\omega  ^*\rangle + \langle
\p_\omega \rmi \beta \alpha _2\phi_\omega ^* ,\rmi \beta \alpha
_2\phi_\omega   \rangle \\& = \langle \p_\omega\phi_\omega
,\phi_\omega  ^*\rangle + \langle \p_\omega
 \phi_\omega ^* , \phi_\omega
    \rangle =
\partial _\omega \| \phi _\omega  \| _{2}^2 \neq 0.
\end{aligned}\end{equation*}
(6) is obtained by a direct computation.
\end{proof}

\begin{remark}
From \eqref{Eq:SymmLin1}, if $z\in\sigma(\mathcal{H}_\omega)$ then
$\bar{z}\in \sigma(\mathcal{H}_\omega)$. So if $z\in \sigma
(\mathcal{H}_\omega)$ then $\{z,-z, \overline{z},-\overline{z}\}\subseteq \sigma
(\mathcal{H}_\omega )$.
\end{remark}

\begin{remark} \label{rem:embei} The observation that
 $   2\omega $ is an eigenvalue of $\mathcal{H}_\omega$ in $L^2(\R ^3 , \C^2)$
 is due to \cite{Comech}. For $3\omega >m$  the eigenvalue $   2\omega $ is
embedded in the continuous spectrum. The fact that  the  vectors  in
 Claim (6) Lemma \ref{lem:spectrum}   do not satisfy   the symmetry
\eqref{eq:symmR} and   are not in $\mathbf{X}$, shows that the existence of this
eigenvalue does not interfere with our proof. Obviously the symmetry
{\ref{Assumption:H4}} is crucial.
\end{remark}

We have   the beginning of $
\mathcal{H}_{\omega }$ invariant Jordan
block decomposition $\mathbf{X}=
N_g(\mathcal{H} _{\omega })\oplus
 N_g^\perp (\mathcal{H} ^\ast _{\omega }).
$
Linear stability means to us what follows, see \cite{Cuccagna2}.
\begin{definition}[Linear Stability]\label{def:LinStab}
A standing wave $e^{\rmi t\omega}\phi_\omega$ is
linearly stable when the following   hold:
\begin{itemize}
\item[(1)] $\sigma   ( \mathcal{H} _\omega )\subset \R$;
\item[(2)] $N_g(\mathcal{H}   )= \{ \Sigma _3\Phi _\omega ,
 \partial _\omega \Phi _\omega \} $;
\item[(3)] for any eigenvalue $z\neq 0$  of $\mathcal{H} _\omega$
we have $N_g( \mathcal{H} _\omega -z)=
      \ker ( \mathcal{H} _\omega -z);$
\item[(4)]  for any  positive eigenvalue $\lambda >0$ and for any
$\xi \in \ker ( \mathcal{H}_\omega  - \lambda  )   $, we have
$\langle\xi  ,\Sigma_3\xi ^* \rangle >0.$
\end{itemize}
\end{definition}

As a consequence of {\ref{Assumption:H5}}, the Jordan decomposition
can be continued as follows:
\begin{equation} \label{Eq:SpecDec}\index{$L^2_c$} \begin{aligned} & \mathbf{X}=
N_g(\mathcal{H}
 _{\omega }) \oplus \big (
\oplus _{j,\pm  }\ker (\mathcal{H} _{\omega }\mp \lambda _j(\omega
)) \big ) \oplus L_c^2(\mathcal{H}_{\omega })\text{ with } \mathbf{X}_c (\mathcal{H}
_\omega )=\left \{  \mathbf{X}_d(\mathcal{H}^*
_\omega )  \right \} ^{\perp}\cap \mathbf{X} ,  \\&    \text{ where for
$K=\mathcal{H} ^\ast _\omega , \mathcal{H}  _\omega$ we set }  \mathbf{X}_d(K)
:= N_g(K) \oplus \oplus
_{j,\pm  }\ker (K \mp \lambda _j(\omega
)) .   \end{aligned}
 \end{equation}
 Let $(\xi _j(\omega ,x ))_j$ be a basis of $\oplus _{j=1
}^{n}\ker (\mathcal{H}_\omega - \lambda _j(\omega ))$
  so that each
vector  is smooth in both
 variables, with $|\partial ^{\alpha }_{\omega x}
 \xi _j(\omega , x)| < c_\alpha  e^{-a_{\alpha}|x|}$ for some
$c_\alpha >0$ and  $a_\alpha >0$. This can be proved by the Combes-Thomas method  \cite{Hislop}  using {\ref{Assumption:H2}}.
We normalize $\xi _j(\omega ,x )$ so that    $ \varepsilon _j=\langle\xi _j
,\Sigma_3\xi _j^* \rangle \in \{ 1, -1  \}$ and   $  \langle\xi _j  ,\Sigma_3\xi
_i^* \rangle =0$ for $j\neq i$.   In Theorem \ref{th:AsStab} for all
j we have $ \varepsilon _j=1$ while for Theorem \ref{th:orbital instability} we
have $ \varepsilon _j=-1$ for at least one $j$.

From the calculations of this section, we have built a dual
basis. Hence, given any vector $X$, we have
\begin{equation}\label{eq:decVectorfield}\begin{aligned} &
 X =  \frac{\langle   X ,
   \left ( e^{ \im \Sigma _3\vartheta }\Sigma _3\partial _\omega \Phi
   \right )^*\rangle }
  {q'(\omega )}  e^{  \im \Sigma _3\vartheta } \Sigma _3
       \Phi   +
  \frac{\langle X ,\left ( e^{ \im \Sigma _3\vartheta }
    \Phi \right ) ^*\rangle }
  {q'(\omega )}
  e^{  \im \Sigma _3\vartheta }  \partial _\omega \Phi +
  \\&
\sum  _{j=1}^{n} \varepsilon _j \langle X,  \left ( e^{ \im \Sigma
_3\vartheta }\Sigma _3\xi _j\right )^* \rangle
   e^{  \im \Sigma _3\vartheta }   \xi _j  +
   \sum _{j=1}^{n} \varepsilon _j \langle X, \left ( e^{ \im \Sigma _3\vartheta
}\Sigma _1 \Sigma _3 C\xi _j\right )^* \rangle
     e^{  \im \Sigma _3\vartheta }
    \Sigma _1C\xi _j       +      e^{   \im \Sigma _3\vartheta }
   P_c(\mathcal{H}_\omega  )e^{   -\im \Sigma _3\vartheta } X,
 \end{aligned}
\end{equation}
with $P_c(\mathcal{H}_\omega)$\index{$P_c(\mathcal{H}_\omega)$}
   the projector onto $\mathbf{X}_c(\mathcal{H}_\omega)$ with respect to
decomposition \eqref{Eq:SpecDec}. More generally, for $X\in L^2(\R ^3, \C ^8)
=\mathbf{X} \oplus \mathbf{X}^\perp $,  see {  the simultaneous
spectral decomposition of $\mathbf{A}$ and $\mathbf{B}$} in Lemma
\ref{lem:SymmLin1},
we denote by $P_c(\mathcal{H}_\omega)X$
the vector obtained first projecting in $\mathbf{X}$ and then in   $\mathbf{X}_c(\mathcal{H}_\omega)$.
By duality, we have the following lemma.
\begin{lemma}
\label{lem:dual spec dec} Suppose that for a given $\omega \in \mathcal{O}$ the conditions of
Definition \ref{def:LinStab} are satisfied. Then

\begin{equation} \label{Eq:dualSpecDec} \begin{aligned} & \mathbf{X}= N_g(\mathcal{H}
 _{\omega }^*) \oplus \big (
\oplus _{j,\pm  }\ker (\mathcal{H} _{\omega }^*\mp \lambda _j(\omega
)) \big ) \oplus \mathbf{X}_c(\mathcal{H}_{\omega }^*)\text{ with } \mathbf{X}_c(\mathcal{H}
_\omega ^*):=\left \{ \mathbf{X}_d(\mathcal{H}_{\omega } ) \right \} ^{\perp}.   \end{aligned}
 \end{equation}
 Any 1 form $\alpha =\langle \alpha ^{\sharp}, \quad \rangle $ can be decomposed
as follows:
 \begin{equation}\label{eq:dec1form}\begin{aligned} &
\alpha ^{\sharp}  =
\frac{\langle \alpha ^{\sharp} ,
    e^{ \im \Sigma _3\vartheta }\partial _\omega \Phi \rangle }
  {q'(\omega )}
  \left ( e^{  \im \Sigma _3\vartheta }    \Phi \right ) ^* +
  \frac{\langle \alpha ^{\sharp} , e^{ \im \Sigma _3\vartheta }
  \Sigma _3 \Phi \rangle }
  {q'(\omega )}\left (
  e^{  \im \Sigma _3\vartheta }\Sigma _3 \partial _\omega \Phi \right )^*
  \\&    + \sum  _{j=1}^{n} \varepsilon _j \langle \alpha ^{\sharp} , e^{  \im
\Sigma _3\vartheta }    \xi _j \rangle
  \left ( e^{  \im \Sigma _3\vartheta }\Sigma _3   \xi _j\right ) ^* -
   \sum _{j=1}^{n} \varepsilon _j \langle \alpha ^{\sharp} , e^{  \im \Sigma
_3\vartheta }
   \Sigma _1C\xi _j \rangle  \left ( e^{  \im \Sigma _3\vartheta }
   \Sigma _3\Sigma _1C\xi _j \right ) ^*  \\&  +
   e^{-  \im \Sigma _3\vartheta }
   \left (P_c(\mathcal{H}_\omega ^* ) e^{ -\im \Sigma _3\vartheta }(\alpha
^{\sharp} )^* \right ) ^* .
 \end{aligned}
\end{equation}
\end{lemma}

\section{Modulation and coordinates}
\label{sec:modulation}

\subsection{Modulation}

Consider the $U$ in \eqref{Eq:linearization}. Then, in the notation of \eqref{Eq:linearization},
\eqref{Eq:ansatz} can be written as
\begin{equation}\label{Eq:ansatzU} U = e^{\rmi \Sigma _3 \vartheta  }
(\Phi _{\omega}+ R).
\end{equation}
Consider the following two functions
\begin{equation} \mathcal{F}(U,\omega , \vartheta ):=\langle
e^{-{\rm i} \Sigma _3\vartheta }U-\Phi _\omega  , \Phi _\omega ^*\rangle
\, , \quad \mathcal{G}(U,\omega , \vartheta ):=\langle e^{-{\rm i}
\Sigma _3\vartheta }U  ,\Sigma _3\partial _\omega  \Phi
_\omega ^*\rangle . \nonumber
\end{equation}
Notice that $R \in N_g^\perp (\mathcal{H} ^\ast _{\omega  }) $ if
and only if $\mathcal{F}(U,\omega , \vartheta )=
\mathcal{G}(U,\omega , \vartheta )=0$. By {\ref{Assumption:H2}} the
map $\omega \in \mathcal{O}\to \phi _ {\omega } \in H^1(\R ^3) $
is $C^\infty$. Then  $\mathcal{F}$ and $\mathcal{G}$ are $C^\infty$
functions with partial derivatives
\begin{equation} \label{Eq:derAnsatz}
\begin{aligned} & \mathcal{F}_{\vartheta}(U,\omega , \vartheta ) =-{\rm i}
\langle
\Sigma _3e^{-{\rm i} \Sigma _3\vartheta }U, \Phi _\omega ^* \rangle
 \, , \\&
\mathcal{F}_{\omega}(U,\omega , \vartheta ) = -2 q'(\omega)+\langle
e^{-{\rm i}\Sigma _3\vartheta }U, \partial _\omega \Phi _\omega ^*
\rangle \ ,
 \\&  \mathcal{F}_{U}(U,\omega , \vartheta )=e^{-{\rm i} \Sigma _3\vartheta
} \Phi _\omega  ^*  \,  , \quad   \mathcal{G}_{U}(U,\omega , \vartheta
)=e^{-{\rm i} \Sigma _3\vartheta } \Sigma _3\partial _\omega \Phi
_\omega  ^* \, ,\\& \mathcal{G}_{\vartheta} (U,\omega , \vartheta )=-{\rm i}
\langle
 e^{-{\rm i} \Sigma _3\vartheta }U,\partial _\omega  \Phi _\omega   ^* \rangle
 \,  , \quad   \mathcal{G}_{\omega} (U,\omega , \vartheta )=\langle e^{-{\rm i} \Sigma
_3\vartheta }U ,\Sigma _3\partial _\omega ^2  \Phi _\omega ^* \rangle
  .
\end{aligned}
\end{equation}
We have $\mathcal{F}(e^{ {\rm i} \Sigma _3\vartheta } \Phi _\omega
,\omega , \vartheta ) =\mathcal{G}(e^{ {\rm i} \Sigma _3\vartheta }
\Phi _\omega ,\omega , \vartheta ) =0$. For $U=e^{ {\rm i} \Sigma
_3\vartheta } \Phi _\omega $ in \eqref{Eq:derAnsatz} we get

\begin{equation*} \label{Eq:derAnsatz1}
\begin{aligned} & \mathcal{F}_{\vartheta}(e^{ {\rm i} \Sigma
_3\vartheta } \Phi _\omega ,\omega , \vartheta ) =0 \quad , \quad
\mathcal{F}_{\omega}(U,\omega , \vartheta ) = -  q'(\omega)  \quad ,
   \\& \mathcal{G}_{\vartheta} (e^{ {\rm i} \Sigma
_3\vartheta } \Phi _\omega ,\omega , \vartheta )=-{\rm i} q'(\omega)
  \quad , \quad \mathcal{G}_{\omega} (e^{ {\rm i} \Sigma
_3\vartheta } \Phi _\omega ,\omega , \vartheta )=0 \quad
  .
\end{aligned}
\end{equation*}
Then by the implicit function theorem and {\ref{Assumption:H3}}
there is a unique choice of functions $\theta =\theta (U)$, $\omega
=\omega (U)$ which are $C^\infty$ and yield to the following lemma.

\begin{lemma}[Modulation] \label{lem:modulation} For any $\omega _1\in
\mathcal{O}$ there exist $\varepsilon >0 $ and $C>0$ such that for
any $ u\in H^1 (\R ^3)$ with $\| u-e^{\rmi  \vartheta _1}\phi
_{\omega _1}\| <\epsilon <\varepsilon $, there exists a unique
choice of $(\vartheta , \omega , r)$ such that  $|\omega  -\omega
_1|+|\vartheta-\vartheta _1|<C\epsilon $ for a fixed $C$,  $R \in
N_g^\perp (\mathcal{H} ^\ast _{\omega  }) $ and \eqref{Eq:ansatzU} hold.
\end{lemma}

Consider   the two $C^\infty$
 functions $\vartheta,\omega : U\in B_{H^1}(e^{i\Sigma
 \vartheta_0}\Phi_{\omega_1},\varepsilon)\to \R$. Inserting  \eqref{Eq:ansatzU}
in  \eqref{Eq:derAnsatz} we get
  \begin{equation}
\begin{aligned} & \mathcal{F}_{\vartheta} = -\im \langle \Sigma _3R, \Phi
_\omega ^*
\rangle   \, ; \quad   \mathcal{F}_{\omega} =  - q'(\omega)+\langle
R,
\partial _\omega \Phi _\omega  ^*  \rangle
 \, ; \\&   \mathcal{F}_{U}=e^{-\im \Sigma _3\vartheta
} \Phi _\omega  ^*   \, ; \quad   \mathcal{G}_{U}=e^{-\im
\Sigma _3\vartheta } \Sigma _3\partial _\omega \Phi _\omega  ^* \,
;\\& \mathcal{G}_{\vartheta} =-\im  ( q'(\omega)+   \langle  R,
\partial _\omega \Phi _\omega
 ^* \rangle   )  \, ; \quad
  \mathcal{G}_{\omega} =
  \langle R ,
  \Sigma _3\partial _\omega ^2  \Phi _\omega ^* \rangle \, .
\end{aligned}\nonumber
\end{equation}
 Then, if we set
\begin{equation}\label{eq:matrixA} \mathcal{A}:=\begin{pmatrix} -
q'(\omega)+\langle  R, \partial _\omega \Phi _\omega  ^*  \rangle &
-\im \langle \Sigma _3R, \Phi _\omega  ^* \rangle \\ \langle R ,\Sigma
_3\partial _\omega ^2  \Phi _\omega ^* \rangle  & -\im  ( q'(\omega)+
\langle  R, \partial _\omega \Phi _\omega  ^* \rangle   ) \end{pmatrix}
\end{equation}
we have the following equality
\begin{equation}\label{eq:ApplmatrixA}
\mathcal{A}
\begin{pmatrix} \nabla \omega  \\ \nabla \vartheta \end{pmatrix}
=\begin{pmatrix} -e^{-\im \Sigma _3\vartheta } \Phi _\omega  ^*  \\
-e^{-\im \Sigma _3\vartheta } \Sigma _3\partial _\omega\Phi _\omega ^*
\end{pmatrix} ,
\end{equation}
where   given a vector field $X$ and a scalar valued function $F$, we have $XF=\langle \nabla F  ,X\rangle =dF(X),$ with $dF$ the exterior differential and $\nabla F$ the gradient.

By the above discussion we obtain  the following lemma.
\begin{lemma}  \label{lem:grad omega theta} We have the following formulas:
\begin{equation*}\label{eq:GradModulation}
\begin{aligned} &
\nabla \omega =\frac{(q'(\omega)+ \langle  R, \partial _\omega \Phi
_\omega  ^* \rangle   ) \left ( e^{  \im \Sigma _3\vartheta }    \Phi \right )
^*
-\langle \Sigma _3R, \Phi _\omega  ^* \rangle \left (
  e^{  \im \Sigma _3\vartheta }\Sigma _3 \partial _\omega \Phi \right
)^*}{(q'(\omega))^2 -\langle  R,
\partial _\omega \Phi _\omega  ^* \rangle ^2 + \langle \Sigma _3R, \Phi
_\omega  ^* \rangle \langle R ,\Sigma _3\partial _\omega ^2  \Phi
_\omega ^* \rangle }
\\ & \nabla \vartheta =\frac{\langle R ,\Sigma _3\partial _\omega ^2  \Phi
_\omega ^* \rangle  \left ( e^{  \im \Sigma _3\vartheta }    \Phi \right ) ^*
+(q'(\omega)- \langle  R, \partial _\omega \Phi _\omega ^*  \rangle
)\left (
  e^{  \im \Sigma _3\vartheta }\Sigma _3 \partial _\omega \Phi \right )^*} {\im
\left [ q'(\omega))^2 -\langle R,
\partial _\omega \Phi _\omega  ^* \rangle ^2 + \langle \Sigma _3R, \Phi
_\omega ^*  \rangle \langle R ,\Sigma _3\partial _\omega ^2  \Phi
_\omega ^* \rangle\right ]} \, .
\end{aligned}
\end{equation*}
\end{lemma}

\subsection{Coordinates}

 For $\omega\in\mathcal{O}$  we consider decomposition \eqref{Eq:SpecDec}.
    By
$P_c(\mathcal{H}_{\omega})$ (resp.
$P_d(\mathcal{H}_{\omega})$), or simply by $P_c( {\omega})$\index{$P_c(\omega)$}
(resp.
$P_d( {\omega})$), we denote the projection on
$\mathbf{X} _c(\mathcal{H}_\omega)$  (resp. $\mathbf{X} _d(\mathcal{H}_\omega)$). The space
$\mathbf{X} _c(\mathcal{H}_{\omega})$ ``depends continuously'' on $\omega$, as
$P_c(\omega)=1-P_d(\omega)$ depends smoothly on $\omega$.

  By Lemma \ref{lem:modulation} we specify the
ansatz \eqref{Eq:ansatzU} imposing    $\omega \in \mathcal{O}$, $\vartheta \in \R$ and $R\in N^{\perp}_g
(\mathcal{H}_\omega ^*)$.
Fix $\omega _0$, where $q(\omega _0)=\| u_0\| ^2_{L^2}$.
For $\omega$ close to $\omega_0$ the map
$P_c(\mathcal{H}_{\omega})$ is an isomorphism from
$\mathbf{X} _c(\mathcal{H}_{\omega _0})$ to $\mathbf{X} _c(\mathcal{H}_{\omega })$.
In particular  we write ($\overline{z
}_j$ is the complex conjugate  of the scalar $z_j$)
\begin{align}
  \label{eq:decomp2}
& N_g^\perp(\mathcal{H}_{\omega }^*)\backepsilon R  =\sum
_{j=1}^{n}z_j \xi _j(\omega )+ \sum _{j=1}^{n}\overline{z
}_j\Sigma_1C\xi _j (\omega ) +P_c(\mathcal{H}_{\omega  } )f \,
,\quad f \in \mathbf{X} _c(\mathcal{H}_{\omega _0}).
\end{align}
Setting $z\cdot \xi =\sum _{j=1}^{n}z_j \xi _j$ and
$\overline{z}\cdot \Sigma_1C\xi   =\sum _{j=1}^{n}\overline{z}_j
\Sigma_1C\xi _j $, we write
\begin{equation}\label{eq:coordinate}U= e^{\im \Sigma _3\vartheta}
 \left( \Phi
_\omega +z \cdot \xi  (\omega )+  \overline{z }\cdot \Sigma_1C\xi
  (\omega )+P_c(\mathcal{H}_{\omega}) f\right )    \end{equation}
 $\omega \in \mathcal{O}$  close to $\omega_0$, $(z,f)\in \mathbb{ C}^n \times
\mathbf{X} _c(\mathcal{H}_{\omega _0})$ close to 0, are our coordinates.
In the sequel, we set
\begin{equation*} \label{eq:partialR}\index{$\partial _\omega R$} \partial
_\omega R:=
\sum _{j=1}^{n}z_j \partial _\omega \xi _j(\omega )+ \sum
_{j=1}^{n}\overline{z }_j\Sigma_1C\partial _\omega \xi _j (\omega
)+\partial _\omega P_c(\mathcal{H}_{\omega } )f.
\end{equation*}
Then  we have the vector fields
\begin{equation}\label{eq:vectorfields}\index{$\frac \partial {\partial
{\omega}}$} \index{$\frac \partial {\partial
{\vartheta}}$}\index{$\frac \partial {\partial
{z_j}}$} \index{$\frac \partial {\partial
{\bar z_j}}$}\begin{aligned} &
\frac \partial {\partial  {\omega}}    =
   e^{ \im \Sigma _3\vartheta } \partial _\omega ( \Phi +R)
\, ,\, \frac \partial {\partial  {\vartheta}}  =\im
   e^{ \im \Sigma _3\vartheta } \Sigma _3 ( \Phi +R)
,\\& \frac \partial {\partial  {z_j}}   =
   e^{ \im \Sigma _3\vartheta }  \xi _j   \, ,\,
   \frac \partial {\partial  {\overline{z}_j}}    =
   e^{ \im \Sigma _3\vartheta }\Sigma _1 C\xi _j  .\end{aligned}
\end{equation}
 In particular, given a scalar function $F$,   we have
\begin{equation*}\label{eq:derivativeZ} \begin{aligned} &
\partial _{\omega}F =\langle \nabla F  ,
   e^{ \im \Sigma _3\vartheta } \partial _\omega ( \Phi +R) \rangle
\, ,\, \partial _{\vartheta }F  =\im \langle \nabla F  ,
   e^{ \im \Sigma _3\vartheta } \Sigma _3 ( \Phi +R) \rangle
,\\&
\partial _{z_j}F =
\langle \nabla F  ,
   e^{ \im \Sigma _3\vartheta }  \xi _j \rangle \, ,\,
   \partial _{\overline{z}_j}F = \langle \nabla F  ,
   e^{ \im \Sigma _3\vartheta }\Sigma _1 C\xi _j\rangle .\end{aligned}
\end{equation*}

\begin{lemma} \label{lem:gradient zf} We have the following
formulas:

\begin{equation*} \label{ZOmegaTheta} \begin{aligned} &
 \varepsilon _j\nabla z_j =  - \langle \Sigma _3 \xi _j^*, \partial _\omega R
\rangle
  \nabla \omega - \im \langle \Sigma _3 \xi _j^*, \Sigma _3 R \rangle
  \nabla \vartheta +   e^{-\im \Sigma _3\vartheta }
  \Sigma
_3\xi _j^*\\   & \varepsilon _j \nabla \overline{z}_j =  - \langle \Sigma
_1\Sigma
_3 (C\xi _j )^*,
\partial _\omega R \rangle \nabla \omega-
   \im \langle \Sigma _1\Sigma _3 (C\xi _j )^*, \Sigma _3 R \rangle
   \nabla \vartheta  +    e^{-\im \Sigma _3\vartheta }
  \Sigma _1\Sigma _3 (C\xi _j )^*\\   &
   f'(U)= (P_c(\omega )P_c(\omega _0))^{-1} P_c(\omega )\left [
 -   \partial _\omega R \, d\omega -\im
  \Sigma _3 R  \, d\vartheta +  e^{-\im
\Sigma _3\vartheta }\uno \right ] ,\end{aligned}
\end{equation*} with $(P_c(\omega
)P_c(\omega _0))^{-1}: \mathbf{X} _c(\mathcal{H}_{\omega
 })\to \mathbf{X} _c(\mathcal{H}_{\omega _0})$ the inverse of
$P_c(\omega )P_c(\omega _0):\mathbf{X} _c(\mathcal{H}_{\omega _0})\to
\mathbf{X} _c(\mathcal{H}_{\omega })$ and $ \varepsilon _j=
\langle\xi_j,\Sigma_3\xi_j\rangle$.
\end{lemma}

\begin{proof} The proof is similar to the proof of
\cite[Lemmas 4.1--4.2 ]{Cuccagna1}.
 Let us see for example the proof of the first formula.
Equalities   $ \frac{\partial z_j}{\partial z_\ell }  =\delta _{j\ell}$,
$ \frac{\partial z_j}{\partial \overline{z}_\ell }  = \frac{\partial
z_j}{\partial \omega }= \frac{\partial z_j}{\partial \vartheta } =0$
 and $\nabla _f z_j=0$  are equivalent to
\begin{equation}\label{eq:indentitiesGradZ} \begin{aligned} & \langle \nabla
z_j,
e^{ \im \Sigma _3\vartheta } \xi _\ell \rangle =\delta _{j\ell},
\langle \nabla z_j, e^{ \im \Sigma _3\vartheta } \Sigma _1C\xi _\ell
\rangle \equiv 0 =\langle \nabla z_j, e^{ \im \Sigma _3\vartheta }
\Sigma _3(\Phi +R) \rangle \\& \langle \nabla z_j, e^{ \im \Sigma
_3\vartheta } \partial _\omega (\Phi +R) \rangle =0\equiv \langle
\nabla z_j, e^{ \im \Sigma _3\vartheta } P _c( \omega ) P _c( \omega
_0)g \rangle \, \forall g\in \mathbf{X} _c(\mathcal{H}_{\omega _0}).
\end{aligned}
\end{equation}
Notice that the last identity implies $ P _c( \mathcal{H}_{\omega
_0}^{*} ) P _c( \mathcal{H}_{\omega  }^{*} )e^{ \im \Sigma
_3\vartheta }\nabla z_j=0$ which in turn implies $   P _c(
\mathcal{H}_{\omega  }^{*} )e^{ \im \Sigma _3\vartheta }\nabla
z_j=0$. Then, applying \eqref{eq:decVectorfield} and using the product
row column, we get for some pair of numbers $(a,b)$
\begin{equation} \begin{aligned} & \nabla z_j=a
e^{- \im \Sigma _3\vartheta }  \Phi ^*
+b e^{- \im \Sigma _3\vartheta }\Sigma _3\partial _\omega \Phi ^*+\varepsilon _j
e^{-\im \Sigma _3\vartheta } \Sigma _3\xi _j^*
\\& =  (a,b)  \begin{pmatrix} e^{- \im \Sigma _3\vartheta }  \Phi ^*
\\ e^{- \im \Sigma _3\vartheta }\Sigma _3\partial _\omega \Phi ^* \end{pmatrix}
+  \varepsilon _j
e^{-\im \Sigma _3\vartheta } \Sigma _3\xi _j^* =-(a,b)  \mathcal{A}
\begin{pmatrix} \nabla \omega
\\ \nabla \vartheta \end{pmatrix} +  \varepsilon _j
e^{-\im \Sigma _3\vartheta } \Sigma _3\xi _j^*,\end{aligned}\nonumber
\end{equation}
where in the last line we used \eqref{eq:ApplmatrixA}. Equating the two extreme
sides and  applying to the formula $\langle \cdot , \frac{\partial}{\partial
\omega}\rangle $ and $\langle \cdot , \frac{\partial}{\partial \vartheta}\rangle
$, by $\langle \nabla z_j , \frac{\partial}{\partial \omega}\rangle =\langle
\nabla z_j , \frac{\partial}{\partial \vartheta}\rangle =\langle \nabla
\vartheta , \frac{\partial}{\partial \omega}\rangle =\langle \nabla \omega ,
\frac{\partial}{\partial \vartheta}\rangle =0 $, by
$ \langle \nabla \vartheta , \frac{\partial}{\partial \vartheta}\rangle =\langle
\nabla \omega , \frac{\partial}{\partial \omega}\rangle =1 $ and by
  \eqref{eq:vectorfields} and \eqref{eq:indentitiesGradZ},  we get
\begin{equation}  \mathcal{A}^* \begin{pmatrix} a  \\ b \end{pmatrix}
=\varepsilon _j
\begin{pmatrix} \langle \Sigma _3 \xi _j^*, \partial _\omega R \rangle
  \\ \im \langle \Sigma _3 \xi _j^*, \Sigma _3 R \rangle \end{pmatrix}.\nonumber
\end{equation}
This implies
\begin{equation} \begin{aligned} &
\nabla z_j =  -\varepsilon _j(\langle \Sigma _3 \xi _j^*, \partial _\omega R
\rangle
  , \im \langle \Sigma _3 \xi _j^*, \Sigma _3 R \rangle )\begin{pmatrix} \nabla
\omega
\\ \nabla \vartheta \end{pmatrix}+  \varepsilon _j
e^{-\im \Sigma _3\vartheta } \Sigma _3\xi _j^*.\end{aligned}\nonumber
\end{equation}
\end{proof}

\section{Smoothing and dispersive estimates}
\label{sec:Dispersive}
We collect   the statements on linear theory needed
later to prove the nonlinear estimates.

\begin{lemma}
  \label{lem:absflat} The following facts are true.
 \begin{itemize}
\item[(i)] For any $\tau \geq 1 $   there exists $C$
independent of $\omega$ s.t.

 \begin{eqnarray}& \label{eq:absfree1}
    \|  R _{D_m}(z)  \psi  \| _{L^{2,-\tau }}
     \leq C \|\psi \| _{L^{2, \tau }}
     \text{ for all $z\not \in \R$}\\& \label{eq:absfree2}
      \|  R _{\mathcal{H} _{\omega ,0}}(z)
      \psi  \| _{L^{2,-\tau }} \leq C \|\psi \| _{L^{2, \tau }}
      \text{ for all $z\not \in \R$.}\end{eqnarray}
      \item[(ii)]
  For any   $\tau >1$   the following limits
   \begin{equation} \label{eq:absfree5}\index{$R_{D_m}^+$}
   R_{D_m}^+
    (\lambda )=\lim _{\varepsilon \searrow 0} R_{D_m}
    (\lambda \pm \rmi \varepsilon  )
    \text{ and } R_{\mathcal{H}_{\omega ,0}}^+
    (\lambda )= \lim _{\varepsilon \searrow 0}
R_{\mathcal{H}_{\omega ,0}}
    (\lambda \pm \rmi \varepsilon  )
   \end{equation}
   exist  in $B(H^{1, \tau   }_x, L^{2, -\tau   }_x)$ and the
   convergence is uniform for $\lambda $ in compact sets.
   \end{itemize}
\end{lemma}
\begin{proof} Estimate \eqref{eq:absfree1} implies \eqref{eq:absfree2}.
Then (i) is the content of \cite[Theorem 2.1]{IftimoviciMantoiu}
while (ii) is contained in \cite[Theorem 1.6]{GeorgescuMantoiu}.
\end{proof}

\begin{lemma}
\label{lem:flat R} We have $R  _{\mathcal{H}_{\omega ,0}
}(x,y,\lambda )=  R  _{\mathcal{H}_{\omega ,0}
}(x-y,\lambda )=\begin{pmatrix} R  _{D_m }(x-y,\lambda +\omega ) & 0 \\
0 & R  _{D_m }(x-y,\lambda -\omega )
\end{pmatrix}  $ for $\lambda \not \in \sigma
(\mathcal{H}_{\omega ,0})$    with
\begin{equation} \label{eq:flat R1} \begin{aligned}&
 R  _{D_m }(x ,\Lambda  )= \begin{pmatrix}
 (\Lambda +m) I_2 &   \rmi \sqrt{m^2-\Lambda ^2}\sigma \cdot \widehat{x} \\
 \rmi \sqrt{m^2-\Lambda ^2}\sigma \cdot \widehat{x} &   (\Lambda
-m) I_2
\end{pmatrix} \frac{e^{-\sqrt{m^2-\Lambda ^2} |x|}}{4\pi |x| } +
\rmi   \frac{\alpha  \cdot \widehat{x}} {4\pi |x|
^2}e^{-\sqrt{m^2-\Lambda ^2} |x|}
\end{aligned}
\end{equation}
where  $\widehat{x}=x/|x|$ and  where for $\zeta =e^{\rmi
\vartheta} r$ with $r\ge 0$ and $\vartheta \in (-\pi , \pi )$ we
set $\sqrt{\zeta} =e^{\rmi \vartheta/2} \sqrt{r}$.
\end{lemma}
\begin{proof} This is \cite[Identity  (1.263) section 1.E]{Thaller}.
\end{proof}

\begin{remark}
\label{rem:sq root}   $ R  _{D_m }^{+}(x ,\Lambda  )$ for $\Lambda
>m  $   (resp. $\Lambda <-m  $) is obtained substituting
$\sqrt{m^2-\Lambda
^2}$  in \eqref{eq:flat R1}  with  $\displaystyle -\rmi \sqrt{
\Lambda ^2 -m^2}= \lim _{\varepsilon \searrow 0}\sqrt{m^2-(\Lambda
+\rmi \varepsilon )^2} $ (resp. $\displaystyle  \rmi \sqrt{ \Lambda
^2 -m^2}= \lim _{\varepsilon \searrow 0}\sqrt{m^2-(\Lambda +\rmi
\varepsilon )^2} $).
\end{remark}

\begin{theorem}
\label{Thm:Smoothness flat} For any $\tau > 1 $ and $k\in \R$
$\exists $ $C$ s.t.
\begin{equation*}\label{eq:Smoothness flat} \begin{aligned} &%
 \|  e^{-\rmi t D_m}\psi \|_{L_t^ 2(\R,H^{k, -\tau})} \leq
C \|\psi \|_{H^{k }},
 \\&
  \| \int_{\R}e^{\rmi t D_m}  F(t)\;dt \|_{H^k} \leq
C \|F \|_{L_t^2(\R,H^{k,  \tau})},
 \\&
 \|\int_{t'<t}  e^{-\rmi(t-t')
D_m}   F(t')\;dt'  \|_{L_t^2(\R,H^{k, -\tau})} \leq C \| F
\|_{L_t^2(\R,H^{k,  \tau})}.
  \end{aligned}
\end{equation*}
The same estimates with the same constants hold when we replace
$D_m$ with $\mathcal{H}_{\omega ,0} $.
\end{theorem}
\begin{proof}
This is \cite[Theorem 1.1]{Boussaid2} in the free case. But can be easily
deduced from Lemma \ref{lem:absflat} using tools in \cite[Section
XIII.7]{ReedSimon4}.
\end{proof}
The   following theorem is a special case of Theorem 1.1 \cite{Boussaid}.
\begin{theorem}
\label{Thm:wdec} For any $\tau > 5/2 $ and $k\in \R$
$\exists $ $C$ s.t.
$
 \|  e^{-\rmi t D_m}\psi \|_{ H^{k, -\tau} (\R ^3)} \leq
C  \langle t \rangle ^{-\frac{3}{2}}   \|\psi \|_{H^{k,\tau }}
  .$
The same estimates with the same constants hold when we replace
$D_m$ with $\mathcal{H}_{\omega ,0} $.
\end{theorem}

\begin{theorem}
\label{Thm:Strichartz flat} For any $2\leq p,q \leq \infty$,
$\theta\in[0,1]$, with
$(1-\frac{2}{q})(1\pm\frac{\theta}{2})=\frac{2}{p}$ and
$(p,\theta)\neq(2,0)$, and for any reals $k$, $k'$ with $k'-k\geq
\alpha(q)$, where $\alpha(q)=(1+\frac{\theta}{2})(1-\frac{2}{q})$,
there exists a positive constant $C$ such that
\begin{equation*} \label{eq:Strichartz flat} \begin{aligned} &
\left\|e^{-\rmi t
D_m}\psi\right\|_{L_t^{p}(\R,B^k_{q,2}(\R^3,\C^4))} \leq
 C\left\|\psi\right\|_{H^{k'}(\R^3,\C^4)},
\\&
  \left\|\int e^{\rmi t D_m} F(t)\,dt\right\|_{H^k} \leq
 C\left\|F\right\|_{L_t^{p'}(\R,B^{k'}_{q',2}(\R^3,\C^4))},
\\&
  \left\| \int_{t'<t} e^{-\rmi(t-t')D_m} F(t')\,dt'
\right\|_{L_t^{p}(\R,B^{ k}_{q,2}(\R^3,\C^4))} \leq  C
 \left\|F\right\|_{L_t^{a'}(\R,B^{h}
 _{b',2}(\R^3,\C^4))},
\end{aligned}
\end{equation*}
for any  $(a, b)$ chosen like $(p,q)$, and
$h-k\geq\alpha(q)+\alpha(b)$. Exactly the same estimates hold with
$D_m$ replaced by $\mathcal{H}_{\omega ,0}$.
\end{theorem}
\begin{proof}
For    $D_m$ see \cite{Boussaid2}, see also \cite{Brenner} for the Klein-Gordon
case. For $ \mathcal{H}_{\omega ,0}$
the statement is an immediate consequence of the case
$D_m$.\end{proof}

\begin{lemma}
\label{lem:surrogate}  Consider pairs $(p,q)$ as in Theorem
\ref{Thm:Strichartz flat} with $p>2$, $k\in \R$ arbitrary and
$k'-k\ge \alpha (q)$. Then for any $\tau >1$ there is a constant
$C_0=C_0(\tau , k,p,q)$ such that
\begin{equation*}
\label{eq:surrogate} \left\|  \int _{0} ^t e^{\im D_m(t'-t)} F(t')
dt'\right \| _{L^p_tB^{k} _{q,2}} \le C_0    \|
 F\|_{L_t^2H ^{   k', \tau } }.
\end{equation*}
  The same estimates hold with $D_m$ replaced by $\mathcal{H}
_{\omega ,0}$
\end{lemma}
 \begin{proof}  For $F(t,x
)\in C^\infty _0(  \mathbb{R}\times  \mathbb{R}^3)$ set
\begin{eqnarray}T F(t)=\int _0^{+\infty}
e^{\im (t'-t)D_m}  F(t') dt' \, , \quad f= \int _0^{+\infty} e^{\im
t' D_m}  F(t') dt' .\nonumber
\end{eqnarray} Theorem \ref{Thm:Strichartz flat} implies
$\left\|  T F\right \| _{L^p_tB^{k} _{q,2}} \le \| f \| _{H^{k'}}$
for $k'-k=\alpha (q)$. By Theorem \ref{Thm:Smoothness flat}  we have  $\| f \|
_{H^{k'}} \le C  \|
 F\|_{L_t^2H ^{   k', \tau } }.$ Since $p>2$,
by a well known lemma due to Christ and Kiselev \cite{ChristKiselev}, see
 Lemma 3.1 \cite{SmithSogge},
  the statement of Lemma \ref{lem:surrogate} follows.
\end{proof}

\begin{lemma}
\label{lem:w-est}
Let $\tau _1 >1$,
$\mathcal{K}$   a compact subset of $\mathcal{O}$ and  $I$   a
compact subset of $\sigma _e(\mathcal{H}_{\omega } ) \backslash \{
\pm (m\pm \omega )  \}$.   Assume   {\ref{Assumption:H1}}  and {\ref{Assumption:H6}}--{\ref{Assumption:H8}}.
Then there exists a $C>0$, such that
\begin{equation*}\label{eq:w-est}
\|  e^{- \rmi  t \mathcal{H}_{\omega, 0}
}R_{\mathcal{H}_\omega}^{+}(\lambda ) P_c(\omega) \psi  _{0}\|_{L^{2,
-\tau _1}(\R^3)} \le C\langle  t \rangle ^{-\frac 32} \| \psi _{0}
\|_{L^{2,  \tau _1+1}(\R^3)}
\end{equation*}
for every $t\ge 0$, $\lambda \in I$, $\omega\in\mathcal{K}$ and
$\psi  _0\in \mathcal{S}(\R^3;\C^2)$.
\end{lemma}
 \begin{proof} We expand $ R_{\mathcal{H}_\omega}^{+}(\lambda )=
R_{\mathcal{H}_{\omega, 0}}^{+}(\lambda )   -  R_{\mathcal{H}_{\omega
0}}^{+}(\lambda ) V_\omega R_{\mathcal{H}_\omega}^{+}(\lambda )$. We have from
 \cite[Theorem 2 ]{BerthierGeorgescu}
 \begin{equation*}\label{eq:w-estflat}
\|  e^{- \rmi  t \mathcal{H}_{\omega, 0}
}R_{\mathcal{H}_{\omega, 0}}^{+}(\lambda )   \psi  _{0}\|_{L^{2,
-\tau _1}(\R^3)} \le C\langle  t \rangle ^{-\frac 32} \|
R_{\mathcal{H}_{\omega, 0}}^{+}(\lambda )   \psi _{0}
\|_{L^{2,  \tau _1}(\R^3)} \leq  C_1\langle  t \rangle ^{-\frac 32} \|
\psi _{0}
\|_{L^{2,  \tau _1+1}(\R^3)},
\end{equation*}
with $C_1$ locally bounded in $\lambda$ and
$\tau_1$.  Hence, by exponential
decay of
$\phi_\omega$ and by \eqref{eq:resolv} below,
\begin{equation*} \begin{aligned} & \|  e^{-
\rmi t \mathcal{H}_{\omega, 0}} R_{\mathcal{H}_{\omega, 0}}^{+}(\lambda )
V_\omega
R_{\mathcal{H}_\omega}^{+}(\lambda )  P_c(\omega)   \psi  _{0}\|_{L^{2,
\tau _1} }\\
& \le  C_1\langle  t \rangle ^{-\frac 32}
\left\|V_\omega\right\|_{B(L^{2,-\tau_1},L^{2,\tau_1+1})}
\left\|R_{\mathcal{H}_\omega}^{+}(\lambda )  P_c(\omega)
\right\|_{B(L^{2,\tau_1},L^{2,-\tau_1})}\|  \psi
_{0}
\|_{L^{2,  \tau_1} } \le C ' \langle  t \rangle ^{-\frac 32}.
\end{aligned} \nonumber
\end{equation*}
\end{proof}

\begin{lemma}  \label{lem:smooth11}
Assume  the  hypotheses   of Lemma \ref{lem:w-est}.
 Then for any $\tau >1$, for any $k\in \Z$ with $k\ge 0$,  for a constant  $ C_2=C _2(\tau , \omega ,k)$ semicontinuous in $\omega$, for any $T>0$ and  for any   $\forall$ $g(t,x)\in {S}(\R^4)$, we have

\begin{equation*}\label{eq:smooth111}
 \left\|  \int_0^t e^{-\im (t-s)\mathcal{H}_{\omega
}}P_c(\mathcal{H} _\omega )g(s,\cdot)ds\right\|_{L_{  t}^2 ([0,T],H_x^{k, -\tau})} \le
C\|  g\|_{L_{  t}^2 ([0,T],H_x^{k,   \tau})}.
\end{equation*}

\end{lemma}
\proof   It is not restrictive to focus only on $T=\infty$ and $k=0$. By  Plancherel  inequality      we have

\begin{equation*} \label{eq:smooth112}\begin{aligned} &
\| \int _{0}^t e^{-\im (t-s)\mathcal{H}_{\omega }}P_c(\mathcal{H}_\omega )g(s,\cdot)ds\|_{
L_{t }^2L_{ x}^{2,-\tau }} \le    \| R_{\mathcal{H}_\omega }^+(\lambda
)P_c (\mathcal{H}_\omega )
  \widehat{ \chi }_{[0,+\infty )}\ast _\lambda
   \widehat{ g}(\lambda,x)\|_{L_{\lambda }^2L_{ x}^{2,-\tau } }   \\& \le
   \left\| \,
\|   R_{\mathcal{H}_\omega }^+ (\lambda )P_c  (\mathcal{H}_\omega ) \| _{B(
L^{2,\tau}_x, L^{2,- \tau}_x)} \|
   \widehat{ \chi }_{[0,+\infty )}
   \ast _{\lambda } \widehat{g} (\lambda,x) \|_{L_{ x}^{2, \tau }}\, \right\|_{L^2_\lambda}
\\& \le
  \|  R_{\mathcal{H}_\omega }^+ (\lambda
)P_c(\mathcal{H}_\omega )   \| _{L^\infty _\lambda ( \mathbb{R} ,B( L^{2,\tau}_x,
L^{2,-\tau}_x))}\| g\|_{L_{t }^2L_{ x}^{2, \tau } }   .
\end{aligned}
\end{equation*}
We are done if we can prove
\begin{equation} \label{eq:resolv} \begin{aligned} & \|  R_{\mathcal{H}_\omega }^+ (\lambda
)P_c(\mathcal{H}_\omega )   \| _{L^\infty _\lambda ( \mathbb{R} ,B( L^{2,\tau}_x,
L^{2,-\tau}_x))}\le C_2.
\end{aligned}
\end{equation}
By \eqref{eq:smooth11} and   Lemma \ref{lem:smooth1}  we have
\begin{equation} \label{eq:smooth113} \begin{aligned} &\|  R_{\mathcal{H}_\omega }^+ (\lambda
)P_c(\mathcal{H}_\omega )   \| _{ B( L^{2,\tau}_x,
L^{2,-\tau}_x) }\le  \| ( 1+ AR_{\mathcal{H}_{\omega ,0}} ^{+}(\lambda  )B^*   )^{-1}\| _{ B( \mathbf{X},
\mathbf{X}) } \|
  R_{\mathcal{H}_{\omega ,0}}^+ (\lambda
)\| _{ B( L^{2 ,\tau }_x,
L^{2,-\tau}_x) }.
\end{aligned}
\end{equation}
To prove \eqref{eq:smooth113} it
is enough to consider $\lambda \in (\R \backslash
    [-m+\omega +\delta _0,m-\omega -\delta _0])$ as in \eqref{eq:smooth2}.
    Then we can exploit  inequality \eqref{eq:smooth2} to bound uniformly in $\lambda$ the
    first factor in the rhs of \eqref{eq:smooth113}. The proof that $\|
  R_{\mathcal{H}_{\omega ,0}}^+ (\lambda
)\| _{ B( L^{2,\tau}_x,
L^{2,-\tau}_x) }\le C$ for a fixed $C$ is a consequence of
$\| \lambda
  R_{-\Delta}^+ (\lambda ^2
)\| _{ B( L^{2,\tau}_x,
L^{2,-\tau}_x) }\le C'$ and $\|  \nabla
  R_{-\Delta}^+ (\lambda ^2
)\| _{ B( L^{2,\tau}_x,
L^{2,-\tau}_x) }\le C'$ by \eqref{eq:smooth23}--\eqref{eq:smooth24}.  The last two inequalities are proved in \cite{Agmon}

\qed

\section{Hamiltonian structure}\label{Sec:Hamiltonian}

The discussion in   Sections \ref{Sec:Hamiltonian}--\ref{section:Darboux}  is almost the same of \cite{Cuccagna1}, rewritten in the context of the Dirac systems.

\subsection{Symplectic structure}
\label{section:symplectic}

We recall that in view of Theorem \ref{th:orbital instability} we
 set $ \varepsilon _j=\langle\xi _j  ,\Sigma_3\xi _j^* \rangle $ where $
\varepsilon _j\in \{ 1, -1  \}$. Notice that in Theorem \ref{th:AsStab} and
 in \cite{Cuccagna1},   we have $ \varepsilon _j\equiv 1$.
Our ambient space is    $  \mathbf{X}$.
We focus only on the subspace formed by the  points which satisfy $\Sigma
_1U=C{U}$. In view of \eqref{eq:NLSvectorial}, the natural symplectic structure
  is $
  \Omega (X,Y):=\langle X,  \im \beta \alpha _2 \Sigma _1\Sigma _3Y \rangle .
$

The Hamiltonian vector field $X_G$ of a scalar function $G$   is
defined by the equation $\Omega (X_G,Y)=-\im \langle \nabla G ,  Y \rangle$ for
any vector   $Y$ and
is $X_G=   \beta  \alpha _2 \Sigma _3\Sigma _1 \nabla G  $.

We call Poisson bracket of a pair  of scalar valued  functions $F$
and $G$  the  scalar valued   function
\begin{equation}\label{eq:PoissonBracket}\index{$\{ F,G \}$}
  \{ F,G \} = \langle \nabla F , X_ G   \rangle
= \im \Omega ( X_F,   X_G )=\im \Omega ( \nabla F,   \nabla G ).
\end{equation}
This can be extended to vector valued function using $1$-forms or equivalently
defining the extension the following way.
\begin{definition}\label{def:PoissonFunct}
Given a  function  $\mathcal{G}(U)$
 with
values in $\mathbf{X} _c(\mathcal{H} _{\omega _0}) $, a symplectic form
$\Omega$ and a scalar function $F(U)$, we define
$  \{ \mathcal{G}, F\}  =
\mathcal{G}'(U)X_F(U)$,
with  $X_F$ the Hamiltonian vector field associated to $F$. We set
  $ \{ F,\mathcal{G} \} :=-\{ \mathcal{G}, F\}.$
\end{definition}

\begin{lemma}  \label{lem:InvarianceQ}
Let $Q$ be the function defined by
 \eqref{eq:invariants}and let  $X_Q$ its Hamiltonian vectorfield of $Q$.
Then $ X_Q=-\frac{\partial}{\partial \vartheta} .
$
We have the following formulas :
\begin{equation}\label{Qomega}\begin{aligned} &
& \{ Q,\omega  \} =0  \, , \quad \{ Q,\vartheta  \} =1 \, , \quad
\{ Q,z_j  \} =\{ Q,\overline{z}_j  \}=0 \, , \quad \{ Q,f \}=0 .
\end{aligned}
\end{equation}
\end{lemma}
\proof
\eqref{Qomega} follows from
$ X_Q=-\frac{\partial}{\partial \vartheta} .
$   The latter follows from \eqref{eq:vectorfields}:
\begin{equation*} \begin{aligned} &
& X_Q  = \beta \alpha _2\Sigma _3 \Sigma _1 \nabla Q =\beta \alpha _2
\Sigma _3 \Sigma _1 \im  \beta \alpha _2  \Sigma _1 U = -\im \Sigma _3 U = -\frac{\partial}{\partial \vartheta} .
\end{aligned}
\end{equation*}
\qed

\subsection{Hamiltonian reformulation of the system}
\label{section:Hamiltonian reformulation}

For any scalar function $F$, the time derivative of $F(U(t))$ is $\langle
\nabla F (U), \dot{U} \rangle$ and thus if $U$ satisfies
\eqref{eq:NLSvectorial} it is $\{F,E\}$. A similar identity holds for vector
valued function and thus as in \cite{Cuccagna1} we write our system as
\begin{equation} \label{eq:SystPoiss} \begin{aligned} &
  \dot \omega  = \{ \omega , E \}   \, , \quad    \dot f= \{f, E
\}  \, , \quad   \dot z_j  = \{ z_j , E \} \, , \quad    \dot
\vartheta = \{ \vartheta , E  \}. \end{aligned}
\end{equation}
For $u_0$ the initial datum in
\eqref{Eq:NLDE}, we introduce a new Hamiltonian for which the stationary
solution $\Phi_{\omega_0}$, with $q(\omega_0)=\| u_0\| _{L^2_x}^{2}$, is a
critical point :
\begin{equation}    \label{eq:K} \index{$K$, Hamiltonian}  \begin{aligned} &
K(U)=E(U)+\omega (U)   Q(U)-\omega (U)\| u_0\| _{L^2_x}^{2}.
\end{aligned}
\end{equation}
By Lemma \ref{lem:InvarianceQ} and since $Q(U)$ is an invariant of the motion,
see Lemma \ref{lem:systU}, the solution of the initial value
problem in \eqref{Eq:NLDE} solves also
\begin{equation} \label{eq:SystPoissK} \begin{aligned} &
  \dot \omega  = \{ \omega , K \}   \, , \quad    \dot f= \{f, K
\} \, , \quad  \dot z_j  = \{ z_j , K \} \, , \quad   \dot
\vartheta  -\omega  = \{ \vartheta , K  \}. \end{aligned}
\end{equation}
By $ \frac{\partial}{\partial \vartheta} K  =0$ and  \eqref{Qomega}
the right hand
sides in the equations \eqref{eq:SystPoissK} do not depend on
$\vartheta$. Hence, if we look at the new system
\begin{equation} \label{eq:SystK} \begin{aligned} &
  \dot \omega  = \{ \omega , K \}   \, , \quad    \dot f= \{f, K
\}  \, , \quad    \dot z_j  = \{ z_j , K \} \, , \quad    \dot
\vartheta  = \{ \vartheta , K  \}, \end{aligned}
\end{equation}
the evolution of the crucial variables $(\omega , z, f)$ in
\eqref{eq:SystPoiss} and \eqref{eq:SystK} is the same. Therefore,
to prove Theorem \ref{th:AsStab} it is sufficient to consider
system \eqref{eq:SystK}.
\section{Application of the   Darboux Theorem}
\label{section:Darboux}

We  will show that a resonance  phenomenon is responsible for energy
leaking from discrete to continuous spectrum. This   will be seen in  appropriate coordinates  system, obtained by means of  Birkhoff normal forms.  Since  the coordinates \eqref{eq:coordinate}   are not canonical   for the symplectic
form $\Omega $, it is natural to  apply Darboux theorem, moving to a different set of coordinates.  It is   key  that   our nonlinear Dirac equation   remain semilinear. Hence we follow the argument of \cite[Section
7]{Cuccagna1},  which takes care of this, and
to which we refer for more details.

\medskip

\paragraph{\it Strategy of the proof} For
$q=q(\omega )=\left\|\phi_\omega\right\|^2_{L^2}$, we introduce the
$2$-form
\begin{equation} \label{eq:Omega0}\index{$\Omega _0$} \Omega _0=\im d\vartheta
\wedge
dq +  \varepsilon _j dz_j\wedge d\overline{z}_j+\langle f' (U )\cdot ,\im \beta
\alpha _2 \Sigma _3 \Sigma _1 f' (U )\cdot \rangle ,
\end{equation}
summing on repeated indexes, with $f (U)$ the function in Lemma
\ref{lem:gradient zf},  $f '(U)$
its Frech\'et derivative and the last term in \eqref{eq:Omega0} acting on pairs
$(X,Y)$ like $\langle f' (U )X  , \im \beta \alpha
_2 \Sigma _3 \Sigma _1f' (U )Y\rangle $.

The proof of the Darboux Theorem goes as follows. First consider
\begin{equation} \label{eq:Omegat} \Omega _\tau  =(1-\tau )\Omega _0+\tau
\Omega =\Omega _0 +\tau\widetilde{\Omega} \text{ with
$\widetilde{\Omega}
 :=\Omega -\Omega _0$.}  \end{equation}
In Lemma \ref{lem:OmegaOmega0}, we
check that $\Omega _0(U) =\Omega  (U)$ at $U=e^{\im \Sigma
_3\vartheta}
  \Phi _{\omega_{0}}$.  Then  $\Omega_\tau$ is non degenerate  near $ e^{\im \Sigma
_3\vartheta}
  \Phi _{\omega_{0}}$.
One considers a 1-- form $\gamma (\tau,U)$ such that
 $ d \gamma (\tau,U) = \widetilde{\Omega} $ with $\gamma (U) =0$
 at $U=e^{\im
\Sigma _3\vartheta}
  \Phi _{\omega_{0}}$ (external differentiation will
always be on the $U$
variable only) and the   vector field
   $\mathcal{Y}^\tau $ such that
 $i_{\mathcal{Y}^\tau }\Omega _\tau =-  \gamma $. The
 flow $\mathfrak{F}_\tau $ generated by $\mathcal{Y}^\tau $,   close
 the points $e^{\im \Sigma _3\vartheta}\Phi _{\omega_{0}}$ is defined up to
time 1, and is such that $ \mathfrak{F}_1^*\Omega =\Omega _0$  by
\begin{equation} \label{eq:dartheorem}
\begin{aligned} &\frac{d}{d\tau }
\left ( \mathfrak{F}_\tau ^*\Omega _\tau \right )
  =\mathfrak{F}_\tau ^*\left ( L_{\mathcal{Y}^\tau } \Omega _\tau \right )
  +  \mathfrak{F}_\tau ^*\frac{d}{d\tau}\Omega _\tau  =\\&
 = \mathfrak{F}_\tau^*d\left (  i_{\mathcal{Y}^\tau } \Omega _\tau \right )  +
\mathfrak{F}_\tau ^*  \widetilde{\Omega }  =
  \mathfrak{F}_\tau ^*\left ( -  d \gamma +  \widetilde{\Omega}\right ) =0.
\end{aligned}
\end{equation}
This procedure can be carried out abstractly. But here we need to be careful, choosing  $\gamma$ appropriately, because we want the   new Hamiltonian $\widetilde{K}= K\circ \mathfrak{F}_1$ to be $\vartheta $ invariant
and  yield a semilinear Dirac equation.

In the sequel of this section  all the work is finalized to the correct choice
if
$\gamma$. In Lemma \ref{lem:1forms} we compute explicitly a
differential form $\alpha$ and we make the preliminary choice
$\gamma = \alpha$. This is not yet the right choice. By the
computations in Lemma \ref{lem:linearAlgebra}  we find the obstruction to the
fact that
$\widetilde{K}$ is of the desired type. Lemmas
\ref{lem:HamThetaOmega}--\ref{lem:flow Htheta} are necessary to find
an appropriate solution   $F$ of a differential equation in Lemma
\ref{lem:correction alpha}. Then $\gamma = \alpha + \im  dF$ is the
right choice of $\gamma$. In Lemma \ref{lem:flow1} we collect a
number of useful estimates for $\mathfrak{F}_1$.
  Lemma
\ref{lem:flow2}   contains information necessary for the reformulation of our
system
\eqref{eq:SystK1}--\eqref{eq:SystK2}.

\bigskip

\paragraph{\it Preliminary remarks}

Note that for $U$ in a sufficiently small neigborhood of $\Phi_\omega$, that is
$R$ small, from \eqref{eq:decVectorfield} the vector fields defined in
\eqref{eq:vectorfields} can be completed into a basis of $T _U L^2$ (tangent
space at $U$).
For any vector $Y\in T _U L^2$, we have
\begin{equation*} \label{eq:Y}  \begin{aligned} &
Y=Y _{\vartheta}\frac{\partial}{\partial \vartheta}+Y
_{\omega}\frac{\partial}{\partial \omega} +\sum Y
_{j}\frac{\partial}{\partial z_j}+\sum Y
_{\overline{j}}\frac{\partial}{\partial \overline{z}_j} +e^{\im
\Sigma _3 \vartheta} P_c (\omega )Y _{f }
\end{aligned}\end{equation*}
and defining the dual basis we set
\begin{equation*} \label{eq:Y1}  \begin{aligned} & Y_{\vartheta}
=d\vartheta (Y)\, , \quad  Y_{\omega} =d\omega (Y)\, , \quad Y_{j}
=dz_j (Y)\\& \quad Y_{\overline{j}} =d\overline{z}_j (Y) \, , \quad
\quad Y_{f } =f' (U)  Y .
\end{aligned}\end{equation*}
So similarly, a  differential 1-form $\gamma $ decomposes as
\begin{equation*} \label{eq:gamma}  \begin{aligned} &
\gamma=\gamma ^{\vartheta}d \vartheta +\gamma ^{\omega}d \omega
+\sum \gamma ^{j}d z_j +\sum \gamma ^{\overline{j}}d\overline{z}_j
+\langle \gamma ^{f }, f'\cdot  \rangle  ,
\end{aligned}\end{equation*}
where $\langle \gamma ^{f }, f'\cdot  \rangle $ acts on a vector
$Y$ as $\langle \gamma ^{f }, f'Y \rangle $, with here $\gamma ^f
\in L^2_c(\mathcal{H}_{\omega _0}^*)$; $\gamma ^{\vartheta}$,
$\gamma ^{\omega}$, $\gamma ^{j}$ and  $\gamma ^{\overline{j}}$ are
in $\C$.

Notice that we are reversing the standard notation on
super and subscripts for forms and vector fields.

In the sequel, given a differential $1$-form $\gamma $ and a point $U$, we will
denote by $\gamma _U$ the value of $\gamma $ at $U$.

Given a function $\chi $, denote its hamiltonian vector field with
respect to $\Omega_\tau $ by $X^\tau _\chi$ :  $i _{X^\tau _\chi}\Omega_\tau
=-\im \,
d\chi$. By \eqref{eq:Omega0}  we have $ X_{q(\omega )}^{0}=
-\frac{\partial}{\partial \vartheta}.
$

\paragraph{\it The proof}
We have the following preliminary observation ensuring that $\Omega_\tau$ is a
non degenerate $2$-form in a neighborhood of $e^{\im \Sigma _3\vartheta}
\Phi_{\omega_{0}}$.
\begin{lemma}
  \label{lem:OmegaOmega0} At $U=e^{\im \Sigma _3\vartheta}
  \Phi _{\omega_{0}}$, for any $\vartheta $, we have $\Omega _0(U)=\Omega
(U)$.
\end{lemma}
\proof See also \cite[ Lemma 7.1]{Cuccagna1}.
 Using \eqref{eq:decVectorfield} we get, summing on repeated
indexes,
\begin{equation*}\label{eq:OmegaComponent1}
\begin{aligned} & \Omega (X,Y)=\langle
 X ,\im \beta
\alpha _2 \Sigma _3 \Sigma _1Y  \rangle =
   \\&
\frac{1}{q'}  \langle \cdot   ,  e^{ - \im \Sigma _3\vartheta
}\Sigma _3\partial _\omega \Phi ^*\rangle \wedge
   \langle  \cdot   ,e^{ - \im \Sigma _3\vartheta }
     \Phi ^*\rangle (X,Y) +
    \\&   +   \varepsilon _j
 \langle  \cdot   ,   e^{ - \im \Sigma _3\vartheta }\Sigma _3  \xi _j^*
 \rangle\wedge  \langle \cdot  , e^{ - \im \Sigma _3\vartheta }\Sigma _1
 \Sigma _3 (C\xi _j)^*
 \rangle (X,Y)  \\& +
   \langle  P_c(\mathcal{H}_\omega  )
   e^{  -\im \Sigma _3\vartheta } X ,
   \im \beta
\alpha _2 \Sigma _3 \Sigma _1
   P_c(\mathcal{H}_\omega  ) e^{ - \im \Sigma _3\vartheta }
   Y  \rangle
 . \end{aligned}
\end{equation*}
Set
\begin{equation}\label{eq:a1}
\begin{aligned} & a_1:= -\im q'+\frac{\det \mathcal{A}}{q'} +
 \langle P_{N^\perp _g(\mathcal{H}^{*}_{\omega})}\im  \Sigma _1   R,
\im \beta \alpha _2 \Sigma _3 \Sigma _1     \partial _\omega
R\rangle .
\end{aligned}
\end{equation} Then   $a_1$ is smooth in the
 arguments
 $\omega \in \mathcal{O}$, $z\in \C ^n$ and $f\in H^{-K',-S'} $ (see
\eqref{weighted} for the definition) for any pair
 $(K', S')$  with, for  $( z,f)$ near 0,
 \begin{equation}\label{eq:bounda1}
 \begin{aligned} & |a_1| \le C (K',S') (|z|+\| f\| _{H^{-K',-S'} } )^2
 \end{aligned}
 \end{equation}
 by \eqref{eq:matrixA}. Furthermore $a_1$ is imaginary valued.
 By Lemmas  \ref{lem:grad omega theta} and \ref{lem:gradient
 zf}, summing on repeated indexes  we get
\begin{equation*}\label{eq:OmegaComponent2}
\begin{aligned} & \Omega  =   ( \im q'+a_1) d\vartheta \wedge d\omega
+  \varepsilon _j dz_j\wedge d\overline{z}_j \\&
+     dz_j\wedge \left (
\langle \Sigma _1 \Sigma _3 (C\xi _j)^*  , \partial _\omega R
\rangle \, d\omega + \im \langle \Sigma _1 \Sigma _3 (C\xi _j)^*
,\Sigma _3  R \rangle \, d\vartheta \right ) \\& -
d\overline{z}_j\wedge \left ( \langle \Sigma _3 \xi _j ^* ,
\partial _\omega R \rangle \, d\omega + \im \langle   \Sigma _3 \xi _j ^*
,\Sigma _3  R \rangle \, d\vartheta \right ) +\\& +\langle
P_c(\omega ) P_c(\omega _0)f' \cdot  , \im \beta \alpha _2 \Sigma _3
\Sigma _1 P_c(\omega ) P_c(\omega _0) f'
 \cdot  \rangle +\\& +
 \langle P_c(\omega ) P_c(\omega _0)f'  \cdot  , \im \beta \alpha _2 \Sigma _3
\Sigma _1  P_c(\omega  )
 \partial _\omega R \rangle \wedge  d\omega + \\&
+  \im \langle P_c(\omega ) P_c (\omega _0)f'  \cdot  , \im \beta
\alpha _2 \Sigma _3 \Sigma _1  P_c(\omega )
 \Sigma _3  R \rangle \wedge  d\vartheta
 . \end{aligned}
\end{equation*}
At points $U=e^{\im \Sigma _3 \vartheta} \Phi _\omega$, that is for
$R=0$,  we have
\begin{equation}\label{eq:OmegaComponent3}
\Omega  =\im d\vartheta \wedge dq + \varepsilon _j   dz_j\wedge
d\overline{z}_j+\langle P_c(\omega ) P_c(\omega _0)f' \cdot   , \im
\beta \alpha _2 \Sigma _3 \Sigma _1  P_c(\omega ) P_c(\omega _0) f'
 \cdot  \rangle.
\end{equation}
which at $\omega =\omega _0$ gives $\Omega =\Omega _0$.
 \qed

Since $\Omega_\tau=\Omega_0 +\tau (\Omega -\Omega _0)$ with $\tau \in [0,1]$ and
$\Omega =\Omega_0$ at $e^{\im \Sigma _3\vartheta}\Phi_{\omega_{0}}$,
and since $ \Omega_0$
is a
non degenerate $2$-form, $\Omega_\tau$ is also
non degenerate   in a neighborhood of $e^{\im \Sigma _3\vartheta}
\Phi_{\omega_{0}}$. Thus the map $X\mapsto i_X\Omega_\tau$ from vector fields
to $1$-forms  is bijective at any point in the neighborhood of $e^{\im \Sigma
_3\vartheta}\Phi_{\omega_{0}}$. Notice that Lemma \ref{lem:OmegaOmega0}   is claimed at $\omega _0$ and not at different standing waves, and that   the $e^{\im \Sigma
_3\vartheta}\Phi_{\omega_{0}}$  are the only   stationary solutions preserved by our changes of coordinates.

The next lemma suggests as  candidate for the 1 form $\gamma$
the choice $\gamma
=   \alpha$, for $\alpha$ see below. This is not yet the final choice of $\gamma$.

\begin{lemma}
  \label{lem:1forms} Consider the forms, summing on repeated indexes,
\begin{equation*} \label{eq:1forms}\begin{aligned} &
\varpi (U)Y:=\frac{1}{2}\langle \im \beta \alpha _2 \Sigma _3
\Sigma _1U , Y\rangle \\& \varpi _0(U):=-\im qd\vartheta - \varepsilon
_j\frac{\overline{z}_j dz_j - {z}_j d\overline{z}_j}{2}
+\frac{1}{2}\langle f (U),\im \beta \alpha _2 \Sigma _3 \Sigma _1 f
'(U)\quad \rangle.
\end{aligned}
\end{equation*}
Then $ d\varpi _0=\Omega
_0$, $d\varpi =\Omega   . $ Set
\begin{equation} \label{eq:alpha1} \alpha (U):=\varpi (U)-\varpi _0(U)
+d\psi (U)\text{ where } \psi (U):=\frac{1}{2}\langle \Sigma _3\Phi ^*
  , R\rangle .
\end{equation}
We have  $\alpha  = \alpha ^{\vartheta} d\vartheta +\alpha
^{\omega}
   d\omega
 + \langle \alpha ^f,f'\rangle$ with
\begin{equation} \label{eq:alpha2} \begin{aligned}
\alpha ^{\vartheta} +\frac{\im}{2}\| f\| _2^2   =&
 -\frac{\im}{2}     \| z\cdot \xi +\overline{z}\cdot
 \Sigma _1C\xi \| _2^2-\im \Re \langle z\cdot \xi +\overline{z}\cdot
 \Sigma _1C\xi ,  (P_c(\omega ) f)^*\rangle
\\& -  \im   \Re \langle (P_c(\omega )
 -P_c(\omega _0 )) f,  (P_c(\omega )
 f) ^* \rangle ,
\\    \alpha ^{\omega}  =&   -\frac{1}{2}
 \langle  R^*, \Sigma _3\partial _\omega
R\rangle  ,  \\    \alpha ^{f} = &  \frac{1}{2}\im \beta \alpha _2  \Sigma_1
\Sigma _3P_c(\mathcal{H}_{\omega
_0})\left ( P_c(\mathcal{H}_{\omega
} ) -P_c(\mathcal{H}_{\omega
_0})\right )f.
\end{aligned}
\end{equation}

\end{lemma}
\proof Here the proof is almost the same of \cite[ Lemma 7.2 ]{Cuccagna1}.
We focus on \eqref{eq:alpha2},  the only nontrivial
statement. We will sum over repeated indexes.
 We have
\begin{equation} \label{eq:beta} \begin{aligned} &
\varpi =\frac{1}{2}\langle e^{-\im \Sigma _3\vartheta}\im \beta
\alpha _2 \Sigma _1 \Sigma _3\Phi  , \cdot  \rangle  +
\frac{1}{2}\langle e^{-\im \Sigma _3\vartheta}\im \beta \alpha _2
\Sigma _1 \Sigma _3P_c(\omega )f, \cdot \rangle  \\& + \frac{1}{2} z_j \langle
e^{-\im \Sigma _3\vartheta}\im \beta \alpha _2 \Sigma
_1 \Sigma _3\xi _j, \cdot
 \rangle -\frac{1}{2}\overline{z}_j \langle e^{-\im
\Sigma _3\vartheta} \im \beta \alpha _2   \Sigma _3C\xi _j, \cdot
\rangle   .
\end{aligned}
\end{equation}
By Lemma \ref{lem:dual spec dec} and summing on repeated indexes
    we obtain
\begin{equation} \label{eq:betatilde0} \begin{aligned} &
\frac{1}{2}  \langle  e^{-\im \Sigma _3\vartheta}\im \beta \alpha _2
\Sigma _1 \Sigma _3\Phi  , \cdot  \rangle =
\frac{\langle \frac{1}{2}   \im \beta \alpha _2
\Sigma _1 \Sigma _3\Phi   ,
     \partial _\omega \Phi  \rangle }
  {q'(\omega )}\langle
   e^{  -\im \Sigma _3\vartheta }    \Phi ^*, \cdot  \rangle \\&   +
  \frac{\langle \frac{1}{2}   \im \beta \alpha _2
\Sigma _1 \Sigma _3\Phi  ,
  \Sigma _3 \Phi  \rangle }
  {q'(\omega )}\langle
  e^{ -\im \Sigma _3\vartheta }\Sigma _3 \partial _\omega \Phi ^*, \cdot
\rangle
    +  \varepsilon _j   \langle \frac{1}{2}      \im \beta \alpha _2
\Sigma _1 \Sigma _3\Phi  ,     \xi _j  \rangle \langle
    e^{ - \im \Sigma _3\vartheta }\Sigma _3   \xi _j ^*, \cdot  \rangle \\&-
   \varepsilon _j  \langle \frac{1}{2}   \im \beta \alpha _2
\Sigma _1 \Sigma _3\Phi  ,
   \Sigma _1 C\xi _j  \rangle   \langle  e^{  -\im \Sigma _3\vartheta }
   \Sigma _3\Sigma _1(C\xi _j)^*  , \cdot  \rangle   +\langle
   e^{-  \im \Sigma _3\vartheta }
   (P_c(\mathcal{H}_\omega ^* )   \frac{1}{2}   \im \beta \alpha _2
\Sigma _1 \Sigma _3\Phi ^*)^* , \cdot  \rangle .
 \end{aligned}
\end{equation}
By $\im \beta \alpha _2
  \Sigma _1\Phi  = \im \beta \alpha _2
  C\Phi  =(\im \beta \alpha _2)^2
  \Phi ^*= \Phi ^* $ we have
\begin{equation}   \label{eq:betatilde1} \begin{aligned} & \langle     \im \beta
\alpha _2
\Sigma _3 \Sigma _1\Phi  ,
     \partial _\omega \Phi   \rangle =  \langle
 \phi  ^*,
      \partial _\omega \phi  \rangle - \langle
 \phi  ,
      \partial _\omega \phi ^* \rangle =0,
\end{aligned}
\end{equation}
by $ \displaystyle \langle
 \phi  ,
      \partial _\omega \phi ^* \rangle = \int _{\R^3}\left (   a\partial _\omega
a +b \partial _\omega b  \right ) dx=\langle
 \phi ^* ,
      \partial _\omega \phi  \rangle  ,
$ see {\ref{Assumption:H2}}. Then
\begin{equation}   \label{eq:betatilde2} \begin{aligned} &  \frac{1}{2}  \langle
 e^{-\im \Sigma _3\vartheta}\im \beta \alpha _2
\Sigma _1 \Sigma _3\Phi  , \cdot  \rangle =\\& -
  \frac{q }
  {q' }\langle
  e^{ -\im \Sigma _3\vartheta }\Sigma _3 \partial _\omega \Phi ^*, \cdot
\rangle
  +  \varepsilon _j  \langle \frac{1}{2} \im \beta \alpha _2
\Sigma _1 \Sigma _3\Phi  ,     \xi _j  \rangle \langle
    e^{ - \im \Sigma _3\vartheta }\Sigma _3   \xi _j ^*, \cdot  \rangle \\&-
  \varepsilon _j   \langle \frac{1}{2}   \im \beta \alpha _2
\Sigma _1 \Sigma _3\Phi  ,
   \Sigma _1 C\xi _j  \rangle   \langle  e^{  -\im \Sigma _3\vartheta }
   \Sigma _3\Sigma _1(C\xi _j)^*  , \cdot  \rangle    +\langle
   e^{-  \im \Sigma _3\vartheta }
   (P_c(\mathcal{H}_\omega ^* )   \frac{1}{2}   \im \beta \alpha _2
\Sigma _1 \Sigma _3\Phi ^*)^* , \cdot  \rangle .
\end{aligned}
\end{equation}
with by \eqref{eq:ApplmatrixA}
\begin{equation} \label{eq:bettild}-
  \frac{q }
  {q' }\langle
  e^{ -\im \Sigma _3\vartheta }\Sigma _3 \partial _\omega \Phi ^*, \cdot
\rangle
=\frac{q}{q'}\langle R, \Sigma _3 \partial _\omega ^2 \Phi ^* \rangle
\, d \omega -\im \, \frac{q}{q'}   \, (q'   +\langle R,   \partial
_\omega \Phi ^*\rangle ) \, d \vartheta .
\end{equation}
  Applying Lemma  \ref{lem:gradient zf}, we get (by $\im \beta \alpha _2\Sigma
_1 f=f^* $ which follows from $\Sigma _1U=CU$)
\begin{equation} \label{eq:beta0} \begin{aligned} &
\varpi _0=-\im q\, d\vartheta - \varepsilon _j   \frac{\overline{z}_j \, d z_j -
{z}_j \, d\overline{z}_j  }{2}+\frac{1}{2}\langle f (U),\im \beta \alpha _2
\Sigma
_3\Sigma _1f '(U)\cdot  \rangle
\\& =\im \left ( -q +\frac{1}{2}\| R\| _{L^2}^2 \right
) \, d \vartheta    + \frac{1}{2}\langle \Sigma _3R^*, \partial _\omega R\rangle
\, d \omega +
 +
  \frac{1}{2}\langle \im \beta \alpha _2\Sigma _1\Sigma _3  \left ( 1 -
P_c(\omega
 _0)P_c(\omega )  \right )   f,  f'\, \rangle + \\& + \frac{1}{2}    z_j
\langle e^{-\im \Sigma _3\vartheta}\Sigma _1\Sigma _3(C\xi _j)^*, \cdot
 \rangle -\frac{1}{2}\overline{z}_j \langle e^{-\im
\Sigma _3\vartheta} \Sigma _3\xi _j^*,    \cdot  \rangle   +
+
\frac{1}{2}\langle e^{-\im \Sigma _3\vartheta}\im \beta \alpha _2 \Sigma
_1\Sigma
_3P_c(\omega )f,\cdot  \rangle    .
\end{aligned}
\end{equation}
By \eqref{eq:coordinate} we have
    \begin{equation} \label{eq:dPsi0} \begin{aligned} & d\psi =
\frac{1}{2}\langle \Sigma _3\Phi ^*, \partial _\omega R \rangle
d\omega +\frac{1}{2}  \langle \Sigma _3\Phi ^*, \xi _j \rangle
  dz_j + \frac{1}{2}  \langle \Sigma _3\Phi ^*, \Sigma _1C\xi _j \rangle
d\overline{z}_j  +\frac{1}{2}\langle \Sigma _3\Phi ^*,
P_c(\omega ) f' \cdot  \rangle   .
\end{aligned}
\end{equation}
Applying to \eqref{eq:dPsi0}   Lemma
 \ref{lem:gradient zf}
and the identities \eqref{eq:cancel0} below, we get $d\psi =$
\begin{equation} \label{eq:dPsi} \begin{aligned}d\psi  &   = \frac{1}{2}
\langle
\Sigma
_3\Phi ^* ,   \xi _j \rangle
\langle e^{-\im \Sigma
_3\vartheta} \Sigma _3\xi _j^*, \cdot  \rangle +
\frac{1}{2}  \langle \Sigma
_3\Phi ^* , \Sigma _1C \xi _j \rangle
\langle e^{-\im \Sigma
_3\vartheta}\Sigma _1 \Sigma _3(C\xi _j)^*, \cdot  \rangle
\\& + \frac{1}{2}\langle e^{-\im \Sigma _3\vartheta}\left ( P_c
(\mathcal{H}_\omega ^{*})\Sigma _3 \Phi \right  ) ^*, \cdot  \rangle
\\& +\frac{q}{q'}  \langle \Sigma _3\partial _\omega\Phi  ^*, \partial
_\omega R \rangle d\omega    \\& -\frac{\im }{2} \left \langle
 \underbrace{\langle \Sigma _3 \Phi ^*, \xi _j \rangle \Sigma _3 \xi ^*_j  +
  \langle \Sigma _3 \Phi ^*, \Sigma _1C\xi _j
  \rangle \Sigma _1\Sigma _3 (C\xi _j)^*+
  (P_c(\mathcal{H}_{\omega}^{*})\Sigma _3 \Phi )^*}_{P _{N_g^\perp
(\mathcal{H}_{\omega} )}\Sigma _3 \Phi ^*}
, \Sigma _3 R \right \rangle d\vartheta.
\end{aligned}
\end{equation}
To get the third line of \eqref{eq:dPsi} we have used:

\begin{equation*} \label{eq:cancel0} \begin{aligned} & \frac{1}{2}
 \langle \Sigma _3 \Phi ^*, \partial
_\omega R \rangle - \frac{1}{2}    \langle \Sigma _3 \Phi ^*, \xi
_j \rangle  \langle \Sigma _3\xi _j ^* , \partial _\omega R \rangle
-\\& \frac{1}{2}    \langle \Sigma _3 \Phi ^* , \Sigma _1C\xi _j
\rangle \langle \Sigma _1\Sigma _3(C\xi _j)^* , \partial _\omega R
\rangle - \frac{1}{2}      \langle  \left (
P_c(\mathcal{H}^*_\omega)\Sigma _3\Phi \right ) ^*,  \partial
_\omega R \rangle   =\frac{1}{2} \langle \Sigma _3 \Phi ^*,
\partial _\omega R \rangle ;\\& - \frac{1}{2}\left [ \langle \Sigma
_3 \Phi ^*,
\partial _\omega R \rangle  -\frac{1}{q'} \langle \Sigma _3 \Phi ^* ,
\Sigma _3 \Phi  \rangle  \langle \Sigma _3\partial _\omega\Phi ^* ,
\partial _\omega R \rangle   \right ] =\frac{2q}{2q'}\langle \Sigma
_3\partial _\omega\Phi ^*, \partial _\omega R \rangle .
\end{aligned}
\end{equation*}
Let us consider the sum \eqref{eq:alpha1}. There are various cancelations.
 The first and second   (resp.
the first term of the third) line of \eqref{eq:dPsi} cancel  with
the second and third lines of \eqref{eq:betatilde2} (resp. the
first term of the rhs of
  \eqref{eq:bettild}). The last three terms in
rhs\eqref{eq:beta} cancel with the last two lines of
\eqref{eq:beta0}. The $-\im qd\vartheta $ term in the rhs of
\eqref{eq:beta0}) cancels with the $-\im qd\vartheta $ term in
\eqref{eq:bettild}. Adding the   fourth line of \eqref{eq:dPsi}
with the last term of rhs\eqref{eq:bettild} we get the product of
$\im $ times the following quantities:

\begin{equation} \label{eq:cancel} \begin{aligned} &
 -\frac{1}{2}   \langle
 P_{N^\perp _g(\mathcal{H} _\omega  )} \Sigma _3 \Phi ^*,\Sigma _3 R\rangle -
\frac{q}{q'} \langle R,   \partial _\omega \Phi
^*\rangle =- \frac{1}{2} \langle   \Phi ^* ,  R\rangle \\&
+\frac{1}{2} \langle  P_{N  _g(\mathcal{H} _\omega ^* )} \Sigma _3
\Phi ^*,\Sigma _3 R\rangle - \frac{q}{q'} \langle R,
\partial _\omega \Phi ^*\rangle
\\& = - \frac{1}{2} \langle   \Phi ^*, R \rangle +
 \frac{1}{2q'} \langle \Phi ^*, \Sigma _3R
   \rangle  \langle  \partial _\omega \Phi, \Sigma _3\Phi  ^*
 \rangle   \\& +
 \frac{1}{2q'} \langle  \Sigma _3 \partial _\omega \Phi ^*, \Sigma _3R
  \rangle  \langle  \Sigma _3\Phi ^*,
\Sigma _3 \Phi \rangle - \frac{q}{q'} \langle R,   \partial _\omega
\Phi ^*\rangle =0,
\end{aligned}
\end{equation}
where for the second equality we have used
\begin{equation} P_{N  _g(\mathcal{H} _\omega ^* )}=\frac{1}{q'}
   \Phi  ^* \langle    \partial _\omega \Phi  , \cdot
\rangle +\frac{1}{q'}\Sigma _3 \partial _\omega \Phi  ^* \langle
\Sigma _3 \Phi
 , \cdot  \rangle .\nonumber
\end{equation}
The last equality in \eqref{eq:cancel} can be seen as follows. The
two terms in the third line in \eqref{eq:cancel}  are both equal to
0. Indeed,  $\langle  \Sigma _3\Phi ^*,
\partial _\omega \Phi \rangle =0$ by \eqref{eq:betatilde1} and,
by $ R\in N^{\perp}_g (\mathcal{H}_\omega ^*)$
  and $\Phi ^*\in   N _g (\mathcal{H}_\omega ^*) $,
  $ \langle R,     \Phi ^*\rangle =0$. The two terms in
the fourth line in \eqref{eq:cancel}  cancel each other. Then we
get formulas for $\alpha ^{\omega} $ and $\alpha ^{f}$. We get
$\alpha ^{\vartheta} $   also by $\| P_c(\omega )f\| _2^2=\|  f\|
_2^2 +2\Re \langle (P_c(\omega )
 -P_c(\omega _0 )) f,  (P_c(\omega )
 f) ^* \rangle .$ \qed

\begin{lemma}
  \label{lem:linearAlgebra} We have, summing over repeated indexes (also on  $j$
 and $\overline{j}$):
\begin{equation} \label{eq:linAlg0}\begin{aligned} &
i _{Y  }\Omega _0=\im q'  Y_\vartheta d\omega  -\im q'  Y_\omega
d\vartheta + \varepsilon _j ( Y_j d\overline{z}_{j}-Y_{\overline{j}}dz_j)+
\langle \im \beta \alpha _2\Sigma _1 \Sigma _3 Y_f,f' \cdot  \rangle .
\end{aligned}\end{equation}
For the $a_1$ in \eqref{eq:a1}, and for $  \Gamma  =
i_Y\widetilde{\Omega}$, we have
\begin{equation} \label{eq:linAlg1}  \begin{aligned}     \Gamma
_\omega =& a_1Y_\vartheta +Y_j \langle \Sigma _1
  \Sigma _3 (C\xi _j)^*, \partial _\omega R \rangle
 - Y_{\overline{j}} \langle
  \Sigma _3 \xi _j^*, \partial _\omega R \rangle
   +\langle   Y_f
  , \im \beta \alpha _2\Sigma _3 \Sigma _1  P_c
  \partial _\omega R \rangle ;\\
-  \Gamma _\vartheta   =&
  a_1Y_\omega -\im Y_j \langle  \Sigma _1
  \Sigma _3 (C\xi _j)^*, \Sigma _3 R \rangle
 + \im Y_{\overline{j}} \langle
  \Sigma _3 \xi _j^*, \Sigma _3 R \rangle  -\im \langle   Y_f
  , \im \beta
  \alpha _2\Sigma _3 \Sigma _1   P_c  \Sigma _3 R   \rangle ;
\\    -  \Gamma
_j =& \langle  \Sigma _1
  \Sigma _3 (C\xi _j)^*, \partial _\omega R \rangle
Y_\omega  + \im \langle  \Sigma _1
  \Sigma _3 (C\xi _j)^*, \Sigma _3  R
\rangle Y_ \vartheta
   ; \\      \Gamma_{\overline{j}} =&
\langle   \Sigma _3 \xi _j^*,
\partial _\omega R \rangle Y_\omega  + \im \langle   \Sigma
_3 \xi _j^*, \Sigma _3  R \rangle Y_ \vartheta
  ;
\\    \im \beta\alpha _2 \Sigma _3 \Sigma _1\Gamma _{f}= &
 ( P_c(\omega _0)  P_c(\omega  )-1 ) Y_{f}   +  Y_\omega
 P_c(\omega  _0 )P_c(\omega  )  \partial _\omega R
     +\im \, Y_{\vartheta}  P_c(\omega _0 )
 P_c(\omega  )  \Sigma _3 R
.
\end{aligned}\end{equation}
In particular, for $  \gamma  = i_{Y^{\tau }} {\Omega}_\tau =i_{Y^{\tau }}
{\Omega}_0+\tau \, i_{Y^{\tau }} \widetilde{{\Omega} } $ we have

\begin{equation} \label{eq:linAlg2}  \begin{aligned}    \gamma
_\omega =&   (\im q'+\tau a_1)Y_\vartheta^\tau  +\tau Y_j^\tau  \langle \Sigma
_1
  \Sigma _3 (C\xi _j)^*, \partial _\omega R \rangle
 - \tau Y_{\overline{j}}^\tau  \langle
  \Sigma _3 \xi _j^*, \partial _\omega R \rangle   \\&
+\tau \langle   Y_f^\tau
  , \im \beta \alpha _2\Sigma _3 \Sigma _1  P_c
  \partial _\omega R \rangle  ;\\
  -  \gamma _\vartheta   =&  (\im q'+\tau a_1)Y_\omega^\tau  -\tau \im Y_j^\tau
\langle  \Sigma _1
  \Sigma _3 (C\xi _j)^*, \Sigma _3 R \rangle
 + \tau \im Y_{\overline{j}}^\tau  \langle
  \Sigma _3 \xi _j^*, \Sigma _3 R \rangle  \\&
- \im \tau \langle   Y_f^\tau
  , \im \beta
  \alpha _2\Sigma _3 \Sigma _1   P_c  \Sigma _3 R   \rangle  ;\\
    -  \gamma
_j =&  \varepsilon _j  ({Y}^\tau  ) _{\overline{j}}  +\tau \langle  \Sigma _1
  \Sigma _3 (C\xi _j)^*, \partial _\omega R \rangle
Y_\omega^\tau   + \im \tau \langle  \Sigma _1
  \Sigma _3 (C\xi _j)^*, \Sigma _3  R
\rangle Y_\vartheta^\tau   ; \\
\gamma_{\overline{j}} =& \varepsilon _j
({Y}^\tau  )_{
{j}}+\tau \langle
\Sigma _3 \xi _j^*,
\partial _\omega R \rangle Y_\omega^\tau   + \im \tau \langle   \Sigma
_3 \xi _j^*, \Sigma _3  R \rangle Y_ \vartheta
  ;\\
     \im \beta\alpha _2 \Sigma _3 \Sigma _1\gamma _{f}= &
({Y}^\tau  )_{f}+\tau ( P_c(\omega _0)  P_c(\omega  )-1 ) Y_f^\tau    \\+&
\tau
Y_\omega^\tau
 P_c(\omega  _0 )P_c(\omega  )  \partial _\omega R
     +\im \tau \, Y_{\vartheta}^\tau  P_c(\omega _0 )
 P_c(\omega  )  \Sigma _3 R
    \,
     .
\end{aligned}\end{equation}
\end{lemma}
\proof Identity \eqref{eq:linAlg0} is straightforward.
Identity \eqref{eq:linAlg2}
follows immediately from \eqref{eq:linAlg0}--\eqref{eq:linAlg1}.
Finally, \eqref{eq:linAlg1} is elementary linear algebra, and
basically the same of \cite[ Lemma 7.3]{Cuccagna1}. \qed

\begin{remark}
Choosing $\gamma =  \alpha$
in Lemma \ref{lem:linearAlgebra}  with   $\mathcal{F}_\tau$    the flow of $Y^\tau $,
then   $  ({Y}^\tau )_\vartheta \not \equiv 0$   is an obstruction to
the fact that     $K\circ \mathcal{F}_1$  is  a $\vartheta$
invariant Hamiltonian   yielding a semilinear Dirac equation.
So we want     $
({Y}^\tau )_\vartheta =0 $ or $d\vartheta(Y^\tau)=\im \Omega_\tau(X^\tau_\vartheta,Y^\tau) =0 $,  with $X^\tau_\vartheta$ the
Hamiltonian  fields of $\vartheta$ . To this effect we add a correction to
$\alpha$ and  define $Y^\tau$  from $\alpha + \im dF$ where
$(\alpha+\im dF)(X_\vartheta^\tau)=0$.
\end{remark}

\begin{lemma}
  \label{lem:HamThetaOmega} Consider the vector field
   $X^\tau _\vartheta $ (resp. $X^\tau _\omega $) defined by
   $i_{X^\tau _\vartheta }\Omega_\tau =-\im d\vartheta$
   (resp. $i_{X^\tau _\omega }\Omega_\tau =-\im d\omega$). Then we have  (here
$P_c=P_c(\mathcal{H}_\omega )$
   and $P_c^0=P_c(\mathcal{H}_{\omega _0 })$):
\begin{equation}\label{hamiltonians1}\begin{aligned} X^\tau _\vartheta
=& (X^\tau _\vartheta ) _\omega \big [ \frac{\partial}{\partial\omega}
- \tau
 \langle
\Sigma _3 \xi _j^* , \partial _\omega R\rangle
\frac{\partial}{\partial z_j} - \tau
 \langle \Sigma _1
\Sigma _3(C \xi _j)^* , \partial _\omega R\rangle
\frac{\partial}{\partial \overline{z}_j} \\& -\tau P_c^0(1 +\tau
P_c-\tau P_c^0
 )^{-1} P_c^0 P_c\partial _\omega R \big ] ,
   ,
\end{aligned}
\end{equation}
where $(X^\tau _\vartheta ) _\omega$ is real valued and given by (for the $a_1$ in \eqref{eq:a1})
\begin{equation}\label{hamiltonians2}\begin{aligned}&
 (X^\tau _\vartheta ) _\omega =\frac{ \im}{ \im q'+\tau a_1+ \tau a_2}=-
 (X^\tau _\omega ) _\vartheta
\end{aligned}
\end{equation}

   \begin{equation}\label{a2}\begin{aligned} & a_2:=  \im \tau
   \langle
\Sigma _3 \xi _j^* , \partial _\omega  R\rangle \langle \Sigma _1\Sigma _3(C\xi
_j)^* ,
  \Sigma _3R\rangle -\im \tau\langle \Sigma _1 \Sigma _3(C \xi _j)^* ,
 \partial _\omega  R\rangle \langle    \xi _j^* ,   R\rangle +\\&
+\im \tau\langle  P_c^0(1 +\tau P_c-\tau P_c^0
 )^{-1} P_c^0 P_c\partial _\omega R,
 \im \beta \alpha _2\Sigma _3\Sigma _1P_c \Sigma _3 R\rangle .
\end{aligned}
\end{equation}

\end{lemma}

\proof The proof is almost the same of  \cite[Lemma 7.5]{Cuccagna1}. By
\eqref{eq:linAlg2} for $\gamma =-\im \,d\vartheta$, $X^\tau _\vartheta $
satisfies
\begin{equation}\label{HamTheta1}\begin{aligned}
& (X^\tau _\vartheta ) _\vartheta =0; \\ &  \im  = (\im q' + \tau a_1)
(X^\tau _\vartheta ) _\omega -\im \tau \langle \Sigma _1 \Sigma _3(C \xi _j)^*
, \Sigma _3 R\rangle (X^\tau _\vartheta ) _j+\\& +  \im \tau \langle \Sigma
_3 \xi _j ^*, \Sigma _3 R\rangle (X^\tau _\vartheta ) _{\overline{j}}-\im
\tau \langle (X^\tau _\vartheta ) _{f},    \im \beta \alpha _2\Sigma _3\Sigma
_1P_c
\Sigma _3 R\rangle ;
\\& (X^\tau _\vartheta ) _{f} =\tau (1-P_c^0 P_c )(X^\tau _\vartheta ) _{f}  -
\tau (X^\tau _\vartheta ) _\omega P_c^0P_c
\partial _\omega R ;\\&
   (X^\tau _\vartheta ) _{\overline{j}} = -\tau (X^\tau _\vartheta ) _\omega
 \langle \Sigma _1
\Sigma _3 (C\xi _j)^* , \partial _\omega R\rangle ; \,
  (X^\tau _\vartheta ) _{ {j}}=- \tau (X^\tau _\vartheta ) _\omega
 \langle
\Sigma _3 \xi _j^* , \partial _\omega R\rangle .
\end{aligned}
\end{equation}
This yields \eqref{hamiltonians1} for $X^\tau _\vartheta $ and the
first equality in \eqref{hamiltonians2}.  The fact that $(X^\tau _\vartheta ) _\omega$ is real valued follows from \eqref{hamiltonians2} and the fact that $a_1$ and  $a_2$
are  imaginary valued,  which can be checked by the definitions.   \qed

The following lemma is an immediate consequence of the formulas in Lemma
\ref{lem:HamThetaOmega} and of \eqref{eq:bounda1}.
\begin{lemma}
  \label{lem:HamBounds}
For any $ (K', S' ,K , S)$ we have
\begin{equation}\label{HamTheta2}\begin{aligned} &
|1-   (X^\tau _\vartheta ) _\omega  \, q'|\lesssim \| R \| _{H^{-K',-S'}}^2
\\& |(X^\tau _\vartheta ) _j| +|(X^\tau _\vartheta )
_{\overline{j}}| + \| (X^\tau _\vartheta ) _{f}  \| _{H^{ K , S
}}\lesssim \| R \| _{H^{-K',-S'}}  \end{aligned} \end{equation}
\end{lemma}

\begin{definition}\index{
$H_c^{K,S}$}\index{$\widetilde{\Ph}^{K,S}$}\index{$\Ph^{K,S}$}
Set $H_c^{K,S}(\omega )=P_c(\omega )H ^{K,S}$ and denote
\begin{equation}\label{eq:PhaseSpace} \widetilde{\Ph}^{K,S}=\mathbb{C }^n\times
H_c^{K,S}(\omega
_0)\, , \quad \Ph^{K,S}=\mathbb{R}^2 \times \widetilde{\Ph}^{K,S}
\end{equation}
with elements $(\vartheta , \omega , z, f)\in \Ph^{K,S}$ and $(  z,
f)\in \widetilde{\Ph}^{K,S}$.
\end{definition}

\begin{lemma}
  \label{lem:flow Htheta} We consider  $\forall$ $\tau\in[0,1]$
the hamiltonian field $X^\tau _\vartheta $ and
   the flow
\begin{equation*}\label{FlowTheta1}
\frac{d}{ds}\Phi _s(\tau,U)=X^\tau _\vartheta  (\Phi _s(\tau,U))\, , \, \Phi
_0(\tau,U)=U.\end{equation*} \begin{itemize}
\item[(1)] For any $(K', S')$ there is a $s_0>0$
and a neighborhood $\mathcal{U}$ of $\mathbb{R}\times \{  (\omega
_0,0,0)\}$ in $\Ph^{-K',-S'}$ such that the map $ (s,\tau,U)\to \Phi
_s(\tau,U)$ is smooth

\begin{equation*}\label{FlowTheta2} (-s_0,s_0) \times
 [0,1]\times \left (\mathcal{U}\cap \{ \omega =\omega _0\}
 \right )\to   \Ph^ {-K',-S'} .
\end{equation*}

\item[(2)] $\mathcal{U}$ can be chosen so that for any $\tau\in[0,1]$
there  is  another neighborhood $\mathcal{V}_\tau$ of
$\mathbb{R}\times \{ (\omega _0,0,0)\}$ in $\Ph^{-K',-S'}$ s.t.
the above map establishes a diffeomorphism

\begin{equation}\label{FlowTheta5} (-s_0,s_0) \times
 \left ( \mathcal{U}\cap \{ \omega =\omega _0\}
 \right )\to   \mathcal{V}_t .
\end{equation}

 \item[(3)]
  $ f(\Phi _s(\tau,U))
-f (U )=G (t, s , z, f )  $ is a smooth map for all $(K,S)$

\begin{equation}(-s_0,s_0) \times
 [0,1]\times \left ( \mathcal{U}\cap \{ \omega =\omega _0\}
 \right )\to   H^{K,S} \nonumber
\end{equation}
with $ \| G (t, s , z, f ) \| _{H^{K,S}} \le C |s| (|z| +\| f\|
_{H^{-K',-S'}} ).$

\end{itemize}
\end{lemma}\proof The proof is exactly the same of Lemma 7.7 \cite{Cuccagna1}. We only remark, that the field $X^\tau _\vartheta$, the flow $\Phi _s(\tau,U)$ and the
function $F(\tau ,U)$ in Lemma \ref{lem:correction alpha} are defined intrinsically, and so are periodic in $\vartheta$. This is because $X^\tau _\vartheta$ satisfies
these properties, since $i_{X^\tau _\vartheta }\Omega_\tau =-\im d\vartheta$ with both $\Omega_\tau$ and $d\vartheta$ intrinsically defined and  periodic in $\vartheta$.
   \qed
\begin{lemma}
  \label{lem:correction alpha} We consider a scalar function
  $F(\tau ,U)$ defined as follows:
  \begin{equation*}\label{defCorrect}F(\tau ,\Phi _s(\tau,U))= \im \,
\int _0^s\alpha   _{\Phi _{s'}(t,U)}\left ( X^\tau _\vartheta  (\Phi
_{s'}(t,U))\right ) ds'\, ,  \text{ where $\omega (U)=\omega _0$ .}
\end{equation*}
We have $F
  \in C^{ \infty} ( [0,1]\times \mathcal{U}, \mathbb{R})$ for
   a neighborhood $\mathcal{U}$ of $\mathbb{R}\times \{  (\omega
_0,0,0)\}$ in $\Ph^{-K',-S'}$. We have
\begin{equation} \label{estCorrect}
|F (t,U)| \le C (K',S')  |\omega -\omega _0|\, \left (  |z|+ \| f
\| _{H^{-K',-S'}}\right )^2
   \end{equation}
We have (exterior differentiation only in $U$)
\begin{equation}\label{Vectorfield21}
(\alpha +\im \, d F )(X^\tau _\vartheta ) =0.
\end{equation}

\end{lemma}
\proof    The proof is elementary and is exactly the same of Lemma 7.8 \cite{Cuccagna1}.  \qed

We now have the desired correction for $\alpha$ and below we introduce the
vector field whose   flow yields the wanted change
of coordinates.

\begin{lemma}
  \label{lem:vectorfield}  Denote by $\mathcal{X}^\tau $    the vector field
which solves $
 i_{\mathcal{X}^\tau } \Omega _\tau =-\alpha - \im \, d  F(\tau ) .
  $
Then the following properties hold.

 \begin{itemize}
\item[(1)]  There is a neighborhood $\mathcal{U}$ of
$ \mathbb{{R}}\times \{ (\omega _0, 0,0)  \} $ in $\Ph ^{1,0}$ such
that $ \mathcal{X}^\tau  ( U)
  \in C^{ \infty} ( [0,1]\times \mathcal{U}, \Ph ^{1,0})$.

  \item[(2)]  We have
$
 (\mathcal{X}^\tau )_\vartheta  \equiv 0.
   $

   \item[(3)] For constants $C(K,S,K',S')$

    \begin{equation}\label{Vectorfield3}\begin{aligned} &
\left | (\mathcal{X}^\tau )_\omega + \frac{\|f\| _2^2 }{2q'(\omega )}
\right | \lesssim (|z|+\| f \| _{H^{-K',-S'}})^2;  \\&
|(\mathcal{X}^\tau   )_
 {j} |  +|(\mathcal{X}^\tau   )_
 {\overline{j }} | +\|
  (\mathcal{X}^\tau   )_{f} \| _{H^{ K , S }}
  \lesssim  (|z|+\| f \| _{H^{-K',-S'}}) \times  \\&
  \times
(|\omega -\omega _0| +|z|+\| f \| _{H^{-K',-S'}} +\|f\| _{L^2}^2).
   \end{aligned}\end{equation}

\item[(4)]    We have $ L _{\mathcal{X}^\tau }
\frac{\partial}{\partial \vartheta} :=\left [ \mathcal{X}^\tau ,
\frac{\partial}{\partial \vartheta} \right ] =0.$

 \item[(5)] We have  $(\mathcal{X}^\tau   )_
 {\overline{j }} = \overline{(\mathcal{X}^\tau   )}_
 { {j }}$,   $(\mathcal{X}^\tau   )_
 {f} =C \Sigma _1(\mathcal{X}^\tau   )_
 {f} $.  $(\mathcal{X}^\tau   )_
 {\omega}$ is real valued.

\end{itemize}
\end{lemma}
\proof    The proof  is almost the same of \cite[ Lemma 7.9 ]{Cuccagna1}. Claim
(1)
follows from the regularity properties of $\alpha$,
$F$ and $\Omega_\tau $ and from equations \eqref{Vectorfield5} and
\eqref{Vectorfield7} below. \eqref{Vectorfield21} implies (2) by
\begin{equation}  \im (\mathcal{X}^\tau )_\vartheta =
\im d\vartheta (\mathcal{X}^\tau )=- i _{X^\tau _\vartheta}\Omega_\tau
(\mathcal{X}^\tau  )= i _{\mathcal{X}^\tau }\Omega_\tau  (X^\tau _\vartheta
)=-(\alpha +\im \, dF )(X^\tau _\vartheta ) =0. \nonumber
\end{equation}
 We have $\im (\mathcal{X}^\tau )_\omega =\im d\omega
  (\mathcal{X}^\tau )=- i _{X^\tau _\omega}\Omega_\tau
(\mathcal{X}^\tau  )$, so
\begin{equation}\label{Vectorfield5} \begin{aligned}
 & \im (\mathcal{X}^\tau )_\omega = i _{\mathcal{X}^\tau }\Omega
_\tau (X^\tau _\omega )= -(X^\tau _\omega )_{\vartheta} \big [\alpha ^\vartheta
  +\tau
\partial _j F \,  \langle
 \xi _j ^*,   R\rangle  -\tau
\partial _{\overline{j}}F   \langle
\Sigma _1 (C\xi _j)^* ,   R\rangle \\ & +\tau \langle \nabla _{f}F + \im
\alpha ^{f} , P_c^0(1 +\tau  P_c-\tau P_c^0
 )^{-1} P_c^0
P_c\Sigma _3 R \rangle \big ].\end{aligned}
\end{equation}
Then by \eqref{eq:alpha2}, \eqref{hamiltonians2} and
\eqref{a2}, we get  the first inequality in \eqref{Vectorfield3}:
\begin{equation}\label{Vectorfield51} \begin{aligned}
&\left | (\mathcal{X}^\tau )_\omega +\frac{\| f\| _2^2  }{2q'(\omega )}
    \right |
  \le C  \left ( |z| +\|  f\| _{H^{-K',-S'}}  \right ) ^2.\end{aligned}
\end{equation}
By \eqref{eq:linAlg2} we have the following equations
\begin{equation}\label{Vectorfield7} \begin{aligned}
      \im \, \partial _j F  & =\varepsilon _j (\mathcal{X}^\tau   )_
 {\overline{j }} +
 \tau \langle \Sigma _1 \Sigma _3(C \xi _j)^*, \partial _\omega R \rangle
(\mathcal{X}^\tau )_\omega \, \\  - \im \,
\partial _{\overline{j }}F   & =\varepsilon _j (\mathcal{X}^\tau   )_{j} +
 \tau \langle  \Sigma _3 \xi _j^*, \partial _\omega R \rangle
(\mathcal{X}^\tau )_\omega \, \\ \im \beta \alpha _2  \Sigma _3\Sigma _1(\alpha
^{f}+\im \,
\nabla _{f} F ) & =-(\mathcal{X}^\tau   )_{f} -
 \tau   (P_c ^0  P_c-1)(\mathcal{X}^\tau   )_{f}
   - \tau (\mathcal{X}^\tau )_\omega  P_c^0P_c   \partial _\omega
 R   .\end{aligned}
\end{equation}
Formulas \eqref{Vectorfield7} imply

\begin{equation}\label{Vectorfield8} \begin{aligned}
  &     |(\mathcal{X}^\tau _\omega )_
 {\overline{j }} |\le    |\partial _{j} F | +C
 \left ( |z| +\|  f\| _{H^{-K',-S'}}  \right )
|(\mathcal{X}^\tau )_\omega |\\ & |(\mathcal{X}^\tau _\omega )_
 {j} |\le    |\partial _{\overline{j}}
F | +C\left ( |z| +\|  f\| _{H^{-K',-S'}}  \right )
|(\mathcal{X}^\tau )_\omega | \\& \|
  (\mathcal{X}^\tau _\omega )_{f} \| _{H^{ K , S }} \le
\|
  \alpha ^{f} \| _{H^{ K , S }}+ \|
 \nabla
_{f} F  \| _{H^{ K , S }} +C\left ( |z| +\|  f\| _{H^{-K',-S'}}
\right ) |(\mathcal{X}^\tau )_\omega |   \end{aligned}\nonumber
\end{equation}
which with \eqref{Vectorfield51}, \eqref{eq:alpha2} and Lemma
\eqref{estCorrect} imply \eqref{Vectorfield3}. Claim (4)
 follows by $L_{\frac{\partial}{\partial \vartheta}} \left (\alpha + \im
d F \right ) =0  $ and by the product rule for the Lie
derivative,
\begin{equation}L_{\frac{\partial}{\partial \vartheta}} \left (
i_{\mathcal{X}^\tau } \Omega_\tau \right ) = i_{[\frac{\partial}{\partial
\vartheta},\mathcal{X}^\tau ]} \Omega_\tau +i_{\mathcal{X}^\tau }
L_{\frac{\partial}{\partial \vartheta}}\Omega_\tau
=i_{[\frac{\partial}{\partial \vartheta},\mathcal{X}^\tau ]} \Omega_\tau  .
\nonumber
\end{equation}
It is elementary   to check that  \eqref{Vectorfield5}  and \eqref{Vectorfield7} imply Claim (6), when we use the fact that  $(X^\tau _\omega )_{\vartheta}$ is real valued, we consider \eqref{eq:alpha2}, the fact that $F$ is real valued.

\qed

The following lemma gathers some properties of the change of
coordinates.
\begin{lemma}
  \label{lem:flow1}    Consider the vectorfield $\mathcal{X}^\tau $
in Lemma \ref{lem:correction alpha} and denote by
$\mathcal{F}_\tau(U)$ the corresponding flow. Then the flow
$\mathcal{F}_\tau(U)$ for $U$ near $ e^{\im \Sigma _3 \vartheta}\Phi
_{\omega _0}$ is defined for all $\tau\in [0,1]$. We have $\vartheta
\circ \mathcal{F}_1=\vartheta $. We have
\begin{equation}\label{flow5}\begin{aligned} &
q\left ( \omega   (\mathcal{F}_1 (U))\right ) = q\left (\omega (U)
\right ) -\frac{\| f \|_2^2 }{2}  + \mathcal{E}_{\omega} (U)
\\&
z_j (\mathcal{F}_1 (U))=z_j  (U)+ \mathcal{E}_{j }
(U)\\& f (\mathcal{F}_1 (U))= f (U)   + \mathcal{E}_{f}  (U)
\end{aligned}
\end{equation}

with
\begin{eqnarray}
  \label{flow6}  & |\mathcal{E}_{
\omega  }(U)| \lesssim   (|\omega -\omega _0| +|z|+\| f \|
_{H^{-K',-S'}} )^2,
\\& \label{flow7}
|\mathcal{E}_{  j }(U)| +\| \mathcal{E}_f(U)\| _{H^{ K , S
}}\lesssim (|\omega -\omega _0| +|z|+\| f \| _{H^{-K',-S'}}+\| f \|
^2_{L^{2}} ) \\& \times (|\omega -\omega _0| +|z|+\| f \|
_{H^{-K',-S'}} ) .\nonumber
\end{eqnarray}
For each $\zeta =\omega , z_j , f$ we have
$\mathcal{E}_\zeta(U)
=\mathcal{E}_\zeta ( \| f\| _{L^2}^2, \omega , z, f)
$
with, for a neighborhood $\mathcal{U}^{-K',-S'}$ of
$ \{ (\omega _0,0,0)\}$ in $\Ph^{-K',-S'}\cap \{\vartheta=0\}$ and for
some  fixed $a_0>0$
\begin{equation}\label{flow8}\begin{aligned} & \mathcal{E}_\zeta
( \varrho , \omega , z, f)\in C^\infty ( (-a_0,a_0)\times
\mathcal{U}^{-K',-S'}, \mathbb{C}) \text{ for  $\zeta =\omega , z_j  $}
\end{aligned}
\end{equation}
\begin{equation}\label{flow9}\begin{aligned} & \mathcal{E}_f
( \varrho , \omega , z, f)\in C^\infty ( (-a_0,a_0)\times
\mathcal{U}^{-K',-S'}, H ^{ K , S } \cap \mathbf{X}).
\end{aligned}
\end{equation}

\end{lemma}
\proof The argument is the same of Lemma 7.10 \cite{Cuccagna1}, but we review it for the sake of the reader.
We add a new variable $\varrho$. We define a new field by
\begin{equation}\label{Vectorfield52} \begin{aligned}
 & \im (Y^\tau )_\omega =   -(X^\tau _\omega )_{\vartheta} \big [
 \alpha ^\vartheta +\im \frac{\| f\|_2^2-\rho }{2}
 +\tau
\partial _j F \,  \langle
 \xi _j ^*,   R\rangle  -\tau
\partial _{\overline{j}}F   \langle
\Sigma _1 (C\xi _j)^* ,   R\rangle \\ & +\tau \langle \nabla _{f}F + \im
\alpha ^{f} , P_c^0(1 +\tau  P_c-\tau P_c^0
 )^{-1} P_c^0
P_c\Sigma _3 R \rangle
  \big ], \end{aligned}
\end{equation}
which implies that $(Y^\tau )_\omega$ is real valued,
by \begin{equation*}\label{Vectorfield71} \begin{aligned}
      \im \, \partial _j F  & =\varepsilon _j (Y^\tau  )_
 {\overline{j }}+
 \tau \langle \Sigma _1 \Sigma _3(C \xi _j)^*, \partial _\omega R \rangle
(Y^\tau )_\omega \, \\  - \im \,
\partial _{\overline{j }}F   & =\varepsilon _j (Y^\tau  )_{j} +
 \tau \langle  \Sigma _3 \xi _j^*, \partial _\omega R \rangle
(Y^\tau )_\omega \, \\   \im \beta \alpha _2  \Sigma _3\Sigma _1(\alpha ^{f}+ \im \, \nabla
_{f} F ) & =-(Y^\tau  )_{f} -
 \tau   (P_c ^0  P_c-1)(Y^\tau  )_{f}
   - \tau (Y^\tau )_\omega  P_c^0P_c   \partial _\omega
 R  \ ,\end{aligned}
\end{equation*}
where we see $(Y^\tau  )_
 {\overline{j }} = \overline{(Y^\tau  )}_{j} $,  $C\Sigma _1(Y^\tau  )_{f}=(Y^\tau  )_{f}$ and $(Y^\tau  )_{f} \in X_c (\mathcal{H}_{\omega _0})$,
and by $Y^\tau _\rho =2\langle (Y^\tau )_f, \im \beta \alpha _2 \Sigma _1f\rangle $. Then   $Y^\tau =Y^\tau (
\omega , \rho , z, f)$ defines a new flow $\mathcal{G}_\tau (\rho , U)
$, which reduces to $\mathcal{F}_\tau ( U) $ in the invariant manifold
defined by $\rho = \| f\| _2^2.$ Notice that by $\rho (t)= \rho (0)+\int _0^t
Y_{\rho}^{s}ds$ it is easy to conclude $\rho  (\mathcal{G}_1 (\rho ,U) )= \rho ( U) +O(\text{rhs\eqref{flow6}})$.   Using \eqref{eq:alpha2} , \eqref{hamiltonians2}, \eqref{HamTheta2}  and
\eqref{Vectorfield52} it is then easy to get
\begin{equation} \begin{aligned} &q(\omega (t))= q(\omega (0))+\int _0^t q'(\omega (s))Y_{\omega}^{s}ds =q(\omega (0))-\int _0^t \frac{\rho(  s ) }{2} ds +O(\text{rhs\eqref{flow6}}). \end{aligned}\nonumber
\end{equation}
By standard arguments, see for example the proof of Lemma 4.3 \cite{BambusiCuccagna},  we get
\begin{equation*}\label{flow51}\begin{aligned} &
q\left ( \omega   (\mathcal{G}_1 (\rho ,U))\right ) = q\left (
\omega (U)\right )   -\frac{\rho }{ 2}  + \mathcal{E}_{\omega}
(\rho ,U)
\\& z_\ell (\mathcal{G}_1 (\rho ,U))=z_\ell  (U)+ \mathcal{E}_{\ell
} (\rho ,U)\\& f (\mathcal{G}_1 (\rho ,U))= f (U)   +
\mathcal{E}_{f} (\rho ,U)
\end{aligned}
\end{equation*}
with $\mathcal{E}_{\zeta }  (\rho ,U)$ satisfying \eqref{flow8} for
$\zeta =\omega  , z_\ell$, \eqref{flow9} for $\zeta =f$  and such that $C\Sigma _1\mathcal{E}_{f }  (\rho ,U) = \mathcal{E}_{f }  (\rho ,U)$. We have
$\mathcal{E}_{\zeta }  (   U)=\mathcal{E}_{\zeta }  (\| f\| _2 ,U)$
  satisfying \eqref{flow6} for
$\zeta =\omega   $ and \eqref{flow7} for $\zeta =z_\ell ,f$.
\qed

Eventually we have the desired Darboux type result:
\begin{lemma}{\bf (Darboux Theorem)}
  \label{lem:flow2}
Consider  the flow $\mathcal{F}_\tau$ of Lemma \ref{lem:flow1}. Then
we have $ \mathcal{F}_\tau^*\Omega_\tau =\Omega
_0 .$
  We have $
Q\circ \mathcal{F}_1=q.
$
If $\chi$ is a function with $\partial _\vartheta \chi \equiv 0$,
then $\partial _\vartheta (\chi \circ \mathcal{F}_t) \equiv 0$.
\end{lemma}
\proof  The   proof   is the
same of Lemma 7.11 \cite{Cuccagna1}.
\qed


\section{Reformulation of \eqref{eq:SystK} in the new coordinates}
\label{section:reformulation}
 We set $
H:=K\circ \mathcal{F}_1.$
In the new coordinates \eqref{eq:SystK}  becomes
\begin{equation} \label{eq:SystK1} \begin{aligned} &
q' \dot \omega  =  \frac{\partial H}{\partial \vartheta}\equiv 0 \,
, \quad    q'  \dot \vartheta  =
 -\frac{\partial H}{\partial \omega} \ ,
  \end{aligned}
\end{equation}

\begin{equation} \label{eq:SystK2} \begin{aligned} &
 \im \dot z_j  =  \varepsilon _j  \frac{\partial H}{\partial   \overline{z}_j  }
\, , \quad
    \im \dot f=  \im \beta \alpha _2 \Sigma _3    \Sigma _1 \nabla _f  H.
\end{aligned}
\end{equation}
Recall that we are solving the initial value problem
\eqref{Eq:NLDE} and that we have chosen $\omega _0$ with $q(\omega
_0)=\| u_0\| _{L^2_x}^{2}.$  Correspondingly it is enough to focus
on \eqref{eq:SystK2} with $\omega =\omega _0$.
Consider the notation
of Theorem \ref{th:AsStab}. Let us focus for the moment on the case
 $\varepsilon _j \equiv 1 $ in
   system
\eqref{eq:SystK2}. Then  we prove :
\begin{theorem}\label{theorem-1.2}
Assume {\ref{Assumption:H1}}--{\ref{Assumption:H12}}.  Then  for
any integer $k_0>3$ there exist $\epsilon _0 >0$ and $C>0$ such that for
$ |z(0)|+\| f (0) \| _{H^{k_0} }\le \epsilon <\epsilon _0  $  the
corresponding  solution of \eqref{eq:SystK2} is globally defined
and there are    $f_\pm  \in H^{k_0}$ with $\| f_\pm\| _{H^{k_0}
}\le C \epsilon $   such that
\begin{equation}\label{scattering1}\lim _{t\to \pm \infty } \|e^{\rmi
 {\vartheta}(t)\Sigma_3} f(t) -
  e^{-\im t D_m}  f_\pm \| _{H^{k_0} }=0
\end{equation}
and $\lim _{t\to   \infty } z(t)=0,
$ for   $\vartheta (t)$   the exponent in \eqref{Eq:ansatzU}.
Fix $p_0 >2$ and $\tau _0>1$.
Let $\frac{1}{p}=\frac{1}{2}-\frac{1}{q}$ and $\alpha (q) = \frac{3}{p}$. Then,
 we can choose  $\epsilon _0$ small enough such that
$ f(t,x)=A(t,x)+\widetilde{f}(t,x)$
with
$$\forall n \in \N, \,C_n(t):=\sup_{x\in\R^3 }\langle x \rangle
^{n}|  A(t,x)|\to 0 \mbox{ as } t\to \infty $$
and for some fixed $C$
\begin{equation}\label{Strichartz1}
 \|  \widetilde{f} \| _{L^p_t( [0,\infty ),B^{ k_0-\frac{3}{p}} _{q,2})\cap
L^2_t( [0,\infty ),H^{ k_0 ,-\tau _0}_x)\cap L^2_t( [0,\infty ),L^{
\infty}_x)}\le
  C \epsilon.
\end{equation}
There exist $\omega _+$ such that $|\omega_+-\omega _0|=O(\|f_+\|_2^2)$
such that
$  \lim _{t \to +\infty }
\omega (t) = \omega _+.$
\end{theorem}

\paragraph{ {\it Proof that Theorem \ref{theorem-1.2} implies Theorem
\ref{th:AsStab}}}. If we denote $(\omega , z', f')$ the initial
coordinates, and $(\omega _0 , z , f )$ the coordinates in
\eqref{eq:SystK2}, we have from Lemma \ref{lem:flow1} :
$$|z'-z|=O(|z|+ \| f \| _{L
_x ^{2,-2}} ) \mbox{ and }  \|f'-f\|_{H^{K,S}}=O(|z|+ \| f \| _{L_x ^{2,-2}}
)$$
for any $(K,S)\in(\R^+)^2$. The two
error terms
$O$ converge to 0 as $t\to \infty$. Hence the asymptotic behavior of
$(  z', f')$ and of $(  z , f )$ is the same. We also have, from Lemma
\ref{lem:flow1}, $ q\left
( \omega   (t) \right ) = q\left (\omega _0 \right ) -\frac{\| f(t)
\|_2^2 }{2}    + O( |z(t)|+\| f(t) \| _{L^{2,-2 }_x} ) $  which
implies, say at $+\infty$
\begin{equation}   \begin{aligned}
 &   \lim _{t \to +\infty }
q\left ( \omega   (t) \right ) = \lim _{t \to +\infty  }
  \left ( q\left (\omega _0 \right ) -\frac{\|  e^{-\im t   \mathcal{H} _{\omega
_0,0}} f_+ \|_2^2 }{2}\right ) = q\left (\omega _0 \right )
-\frac{\|   f_+ \|_2^2 }{2}=q(\omega _+)
\end{aligned} \nonumber
\end{equation}
for    $\omega _+$   the unique element near $\omega _0$
for which the last inequality holds. So $ \lim _{t \to +\infty }
\omega (t) = \omega _+.$

\qed

In the case $\varepsilon _j\in \{1, -1\}$ with $\varepsilon _j\not
\equiv 1$, using the same argument of Theorem \ref{theorem-1.2}, we
prove that solutions which remain close to the standing wave,
actually have remainder which scatters. We state this in terms of
the system \eqref{eq:SystK2} and the coordinates after Darboux, but
of course it can be stated also in terms of the original
coordinates, as in Theorems \ref{th:AsStab} and \ref{th:orbital
instability}.

\begin{theorem}\label{theorem-1.3}
  Assume
{\ref{Assumption:H1}}--{\ref{Assumption:H4}},
{\ref{Assumption:H5'}} and
{\ref{Assumption:H6}}--{\ref{Assumption:H12}}. Then there exist
$\epsilon _0 >0$ with the following property.  Suppose that
  $(z(t), f(t))$ is a solution of \eqref{eq:SystK2} such that
   $ |z(t)| + \|
f(0)\| _{H^{k_0} }\le \epsilon  < \epsilon _0  $ for all $t\ge 0$.
Suppose furthermore that  there exists a fixed $C>0$ such that  $   \| f(t)\| _{H^{k_0} }\le C
\epsilon $ for all $t\ge 0$. Then  there exist  $f_+  \in H^{k_0}$
  such that \eqref{scattering1} holds (case $+$)
and we have $\lim _{t\to   +\infty } z(t)=0.$
Furthermore, we can write $f(t,x)=A(t,x)+\widetilde{f}(t,x)$ as in Theorem
\ref{theorem-1.2} in such a way that the same conclusions of Theorem
\ref{theorem-1.2} regarding $A(t,x)$ and $ \widetilde{f} $ hold.
\end{theorem}
\begin{remark}
\label{rem:stab im as stab} Theorem \ref{theorem-1.3} is analogous
to an observation in \cite{MartelMerle} regarding the fact that solutions remaining for all times close to a standing wave, stable or unstable,  converge to it.
Among other references see also
\cite{Beceanu,NakanishiSchlag}.
\end{remark}

Finally,
 Theorem
\ref{th:orbital instability}, that is orbital instability, is a
consequence of the following theorem.

\begin{theorem}\label{theorem-1.4}  Assume
{\ref{Assumption:H1}}--{\ref{Assumption:H4}},
{\ref{Assumption:H5'}} and
{\ref{Assumption:H6}}--{\ref{Assumption:H12}}. Then
there is a $ \epsilon _1>0$ such that for any  $\delta >0$ there is
a solution $(z(t), f(t))$
  of \eqref{eq:SystK2} such that
   $ |z(0)| + \|
f(0)\| _{H^{k_0} }\le \delta   $ but there exists $t\ge 0$ such that  $ |z(t)|
\geq \epsilon _1 $.
\end{theorem}

\subsection{Taylor expansions}
\label{subsec:Taylor expansions} We recall that  $\varepsilon
_j=\langle\xi_j, \Sigma_3\xi_j\rangle\in \{1, -1\}$ is the
signature of the eigenvalues of $\mathcal{H}_\omega$. We set
$d(\omega) :=E(\Phi_\omega)+\omega Q(\Phi_\omega).
$
We recall that $\omega_0$ is the unique element such that
$q(\omega_0)=\|u_0\|_2^2$ and $G$ is the primitive of the
non-linearity $g$ vanishing at $0$.

\begin{lemma}
\label{lem:K} The following statements  hold.
\begin{equation}  \label{eq:ExpK} \begin{aligned} & K =d(\omega )-\omega
\| u_0\| _2^2+K_2+K_P  \text{ with } \\& K_2:=\sum _j\varepsilon _j\lambda _j
(\omega ) |z_j|^2+ \frac{1}{2} \langle \im \beta \alpha _2\Sigma _1
\Sigma _3 \mathcal{H}_{\omega  } f,
  f\rangle \text{ and }\\&
K_P = \langle G_6 (\omega , f(x)),1\rangle+\sum _{|\mu +\nu |= 3}
\langle k_{\mu \nu }(\omega ,z ), 1\rangle  z^\mu \overline{z}^\nu
+\sum _{|\mu +\nu |= 2} z^\mu \overline{z}^\nu \langle  K_{\mu \nu
}(\omega ,z ),\im \beta \alpha _2\Sigma _3 \Sigma _1 P_c(\omega
)f\rangle   \\&   + \sum _{d=2}^4 \langle G_{d } ( \omega , z  ),
(P_c(\omega )f)^{\otimes d} \rangle + \int
_{\mathbb{R}^3} \langle G_5(x,\omega, z , f(x) ),  f^{\otimes
5}(x) \rangle dx,
\end{aligned}
\end{equation}
where  for a small neighborhood $\mathcal{U}$ of $(\omega _0,0)$ in
$\mathcal{O}\times \mathbb{C}^n$, we have what follows.
\begin{enumerate}
\item $ G_6 (x, \omega , f)=G \left ( \frac12(P_c(\omega )f(x))\cdot i
\alpha_2\Sigma_3\Sigma_1 (P_c(\omega )f(x))
\right )$,

\item $k_{\mu \nu
}( \cdot , \omega ,z ) \in C^\infty ( \mathcal{U},
H^{K,S}_x(\mathbb{R}^3,\mathbb{C}^8) $,
\item $K_{\mu \nu
}( \cdot , \omega ,z ) \in C^\infty ( \mathcal{U},
H^{K,S}_x(\mathbb{R}^3,\mathbb{C}^8)\cap \mathbf{X}) $,
\item $G_{d
}( \cdot , \omega ,z ) \in C^\infty ( \mathcal{U},
H^{K,S}_x(\mathbb{R}^3, B   (
 (\mathbb{C}^8)^{\otimes d},\mathbb{C} ))) $, for $2\le d \le 4$
and $G_{2}( \cdot , \omega ,0 ) \equiv 0 $.

\item Let $^t\eta = (\zeta , C \zeta) $ for
$ \zeta \in \C^4$. Then for
  $G_5(\cdot ,\omega  , z , \eta  )$   we have
\begin{equation} \label{5power2}\begin{aligned} &\forall l\in \N \cup \{0\},\,
  \| \nabla _{ \omega ,z,\overline{z} ,\zeta,C \zeta   }
^lG_5( \omega ,z,\eta  ) \| _{H^{K,S}_x (\mathbb{R}^3,   B   (
 (\mathbb{C}^8)^{\otimes 5},\mathbb{C} )} \le C_l .
 \end{aligned}\nonumber \end{equation}
\item We have $k_{\mu \nu }= {k}_{\nu \mu   }^*$, $K_{\mu \nu }
=-C\Sigma_1  {K}_{\nu \mu   } $.
\end{enumerate}
\end{lemma}
\proof

Consider $U=  e^{\im \Sigma _3\vartheta }  (\Phi_\omega + R)$ as in
\eqref{Eq:ansatzU} . Decompose $R$ as in \eqref{eq:decomp2}. Set
$U=\varphi(\omega,z) +  P_c(\omega )f$.
Let $K_p(U)=\int h(U(x))\,dx$, see Lemma \ref{lem:systU}, then after first a
Taylor integral expansion around $f$ at first order and a Taylor integral
expansion around $\phi$ at fourth order, we have
\begin{equation*}\label{ExpEP0}\begin{aligned}
    h (U)&=h \left ( P_c(\omega )f   \right )
+\int _0^1    dh ( t\varphi +  P_c(\omega )f)\varphi\,dt\\
&=h \left (  P_c(\omega )f\right)+\int _0^1 \sum _{i\le 4}  \frac{1}{i!}
 d^{i+1}h( t\varphi ) (P_c(\omega)f)^i \varphi  \,dt  +\\&
+ 5\int _{[0,1]^2}  (1-s)^4\frac{1}{5!}
 d^{6}h (t\varphi +sP_c(\omega)f)  (P_c(\omega)f)^4 \varphi \,dtds
\end{aligned}\end{equation*}
Since $\Phi_\omega$ is a critical point of $K$ as it is in the kernel of
$\mathcal{H}_\omega \Sigma_3$, so in the Taylor expansion of $K$ around
$\Phi_\omega$ there is no first order term. The second derivative of $K$ is the bilinear form
$
 \frac{1}{2} \langle \im \beta \alpha _2\Sigma _1
\Sigma _3 \mathcal{H}_{\omega  } \cdot,
  \cdot\rangle
$.
This gives $K_2$.

The term $K_P$ contains all terms of order higher than $2$ in $f$ and $z$.
Thus coincides with the term of order higher than $2$ in $f$ and $z$ in the
above expansion after integration in $x$.

The Hamiltonian $K$ is a real quantity and considering its conjugate will
exchange $\bar{z}$ and $z$ and lead by a straightforward calculation to the
last assertion.  The fact that $K_{\mu \nu} (\omega , z)\in
\mathbf{X}$ follows
from { $ \langle  \im \beta \alpha _2\Sigma _1
\Sigma _3    \mathbf{A} X, Y\rangle  =\langle  \im \beta \alpha _2\Sigma _1
\Sigma _3    X,  \mathbf{A}Y\rangle  $, $ \langle  \im \beta \alpha _2\Sigma _1
\Sigma _3    \mathbf{B} X, Y\rangle  =\langle  \im \beta \alpha _2\Sigma _1
\Sigma _3    X,  \mathbf{B}Y\rangle  $ and $f\in \mathbf{X}$, see Lemma
\ref{lem:SymmLin1}}.\qed

The following lemma is a reformulation with some rearrangements of the above one
in the canonical coordinates provided by Lemma \ref{lem:flow1}. We set $\delta
_j$ be for $j\in \{ 1,...n \}$ the multi index
$\delta _j=( \delta _{1j}, ..., \delta _{nj}).$ Let $\lambda
_j^0=\lambda _j(\omega _0)$ and $\lambda ^0 = (\lambda _1^0,
\cdots, \lambda _n^0)$.

\begin{lemma}
  \label{lem:ExpH} Let $H:=K\circ \mathcal{F} _1$. Then, around $
 e^{i\Sigma _3 \vartheta}  \Phi _{\omega _0}$
 we have the
expansion
\begin{equation}  \label{eq:ExpH1} H =d(\omega _0)
-\omega _0\| u_0\| _2^2+ \psi (\|f\| _2^2) +H_2 ^{(1)}+\resto
^{(1)}, \text{ where}
\end{equation}

\begin{equation}  \label{eq:ExpH2} H_2 ^{(1)}=
\sum _{\substack{ |\mu +\nu |=2\\
\lambda ^0\cdot (\mu -\nu )=0}}
k_{\mu \nu}^{(1)}( \| f\| _2^2 )  z^\mu
\overline{z}^\nu + \frac{1}{2} \langle   \im \beta \alpha _2\Sigma _1 \Sigma _3
\mathcal{H}_{\omega _0} f,   f\rangle .
\end{equation}
and $\resto ^{(1)}=\widetilde{\resto ^{(1)}} +
 \widetilde{\resto ^{(2)}}  $, with
\begin{equation}  \label{eq:ExpH2resto} \begin{aligned} &
 \widetilde{\resto ^{(1)}}=\sum _{\substack{ |\mu +\nu |=2\\
\lambda ^0\cdot (\mu -\nu )\neq 0  }}k_{\mu \nu}^{(1)}(\| f\|
_2^2 )z^\mu \overline{z}^\nu +\sum _{|\mu +\nu |  = 1} z^\mu
\overline{z}^\nu \langle  H_{\nu \mu }(\| f\|
_2^2 ),\im \beta \alpha _2\Sigma _3 \Sigma _1 f\rangle ,\\&
\widetilde{\resto
^{(2)}}= \int _{\mathbb{R}^3}G (\frac12(P_c(\omega_0 )f(x))\cdot i
\alpha_2\Sigma_3\Sigma_1 (P_c(\omega_0 )f(x)))\,dx + \sum _{|\mu +\nu |
= 3 } z^\mu \overline{z}^\nu \int _{\mathbb{R}^3}k_{\mu \nu}
(x,z,f,f(x),\| f\| _2^2 ) dx\\&  + \sum _{|\mu +\nu | =2  }z^\mu
\overline{z}^\nu \int _{\mathbb{R}^3} \left [ \im \beta \alpha _2\Sigma _1
\Sigma
_3H_{\nu \mu } (x,z,f,f(x), \| f\| _2^2 )\right ]^T f(x)  dx  \\&
+\sum _{j=2}^5 \resto ^{(1)} _j   +
    \widehat{\resto}
    ^{(1)} _2(z,f,  \| f\| _2^2 )  \\& \text{and }
\resto ^{(1)} _j    =\int _{\mathbb{R}^3} F_j(x,z ,f, f(x),\| f\|
_2^2) f^{\otimes j}(x) dx
\end{aligned}
\end{equation}
and where the following holds.

\begin{enumerate}

\item We have  $\psi (s)$ is smooth with $\psi (0)=\psi ' (0)=0$.

\item At $\| f\| _2=0$ we have:
\begin{equation}  \label{eq:ExpHcoeff1} \begin{aligned} &
k_{\mu \nu}^{(1)}( 0 ) =0 \text{ for $|\mu +\nu | = 2$  with $(\mu
, \nu )\neq (\delta _j, \delta _j)$ for all $j$;} \\& k_{\delta _j
\delta _j }^{(1)}( 0 ) =\varepsilon _j\lambda _j (\omega _0)  ,
\text{ where $\delta _j=( \delta _{1j}, ..., \delta _{mj})$ and here
we are not summing in $j$;}
\\& H_{\nu \mu }(  0 ) =0 \text{ for $|\mu +\nu | = 1$. }
\end{aligned}
\end{equation}
These $k_{\mu \nu}^{(1)}( \varrho )$ and $H_{\nu \mu }( x,\varrho
)$ are smooth in all variables with $H_{\nu \mu }( \cdot ,\cdot )
\in C^\infty ( \mathbb{R}_\varrho, H^{K,S} _x(\mathbb{R}^3,\mathbb{C}^8) \cap \mathbf{X})$
for all $(K,S)$.

\item We have for all indexes
\begin{equation}  \label{eq:ExpHcoeff2} \begin{aligned} &k_{\mu \nu}^{(1)} =
 (k_{\mu \nu}^{(1)})^*\, , \quad k_{\mu \nu}
= k_{\mu \nu}^*\, , \quad  H_{\nu \mu } =-C\Sigma _1 {H}_{\mu\nu
 } .
\end{aligned}
\end{equation}
\item   We have  $F_2(x,0 ,0, 0,0)=0 $.

\item For all $(K,S, K',S')$ positives there is a neighborhood $
\mathcal{U}^{-K',-S'}$ of $  \{ (  0, 0) \}$ in
$\widetilde{\Ph}^{-K',-S'}$, see \eqref{eq:PhaseSpace}, such that
\begin{enumerate}
\item
for $^t\eta =
(\zeta , C \zeta) $ where
$ \zeta \in \C^4$.
we have,  for
 $k_{\mu \nu} (x, z,f,\eta , \varrho )$ with
 $(z,f,\zeta , \varrho )\in \mathcal{U}^{-K',-S'}\times \C^4 \times \R$
 \begin{equation}\label{eq:coeff a2} \forall l\in\N^6,\quad\,\| \nabla _{
z,\overline{z},\zeta,C \zeta,f,\varrho} ^l k_{\mu \nu} \|
_{ H^{K,S}_x(\mathbb{R}^3,\mathbb{C})} \le C_l;
\end{equation}
\item  for
 $H_{\nu \mu } (x , z,f,g ,
\varrho ) $,
\begin{equation}\label{eq:coeff G2}
\forall l\in\N^6,\quad\,\| \nabla _{ z,\overline{z},\zeta,C \zeta,f,\varrho} ^l
H_{\nu \mu } \| _{ H^{K,S}_x(\mathbb{R}^3,\mathbb{C}^2)} \le C_l;
\end{equation}

\item for $F_j (x ,z,f,g,\varrho )$,
  \begin{equation} \begin{aligned} &
\forall l\in\N^6,\quad\,\| \nabla _{ z,\overline{z},\zeta,C\zeta,f,\varrho} ^l
F_j \| _{ H^{K,S}_x(\mathbb{R}^3,B   (
 (\mathbb{C}^2)^{\otimes j},\mathbb{C} ))} \le C_l;
 \end{aligned}\nonumber \end{equation}

\item   we have $\widehat{\resto} ^{(1)}
 _2(z,f,  \varrho ) \in C^\infty ( \mathcal{U}^{-K',-S'}\times
  \R , \R)
   $ with  \begin{equation} \begin{aligned} & |\widehat{\resto} ^{(1)}
 _2(z,f,  \varrho ) | \le C (|z|+|\varrho |+ \| f \| _{ H^{-K',-S'}})
\| f \| _{ H^{-K',-S'}}^2.
 \end{aligned}\nonumber \end{equation}
\end{enumerate}
\end{enumerate}

\end{lemma}
\proof The following proof is a continuation of proof of Lemma \ref{lem:flow1}.
We thus consider $H=K\circ \mathcal{G}_1$ as a function of $(\varrho, U)$.
By $\mathcal{G}_1(0,\Phi _{\omega _0})=\mathcal{F}_1(\Phi _{\omega _0})= \Phi
_{\omega _0}$,
  $K'(\Phi _{\omega _0})=0$ and   $\|
\mathcal{F}_1(U) -U\| _{\Ph ^{K,S}}\lesssim \| R \| _{L^2}^2$ we
conclude  $H'(\Phi _{\omega _0})=0$ and  $H''(\Phi _{\omega _0})
=K''(\Phi _{\omega _0})$. In particular, this yields the formula
for $H_2^{(1)}+\widetilde{\resto^{(1)}}$   for $\varrho=\| f\|_2^2=0$.

The other terms are obtained
by substituting in  $K_P$ of \eqref{eq:ExpK} the formula  \eqref{flow5}. The
term  $\psi (\varrho)$ arises from
$d(\omega \circ \mathcal{G}_1)-\omega \circ \mathcal{G}_1\| u_0\|
_2^2$.  There are no monomials $\| f\|
_2^j z^\mu \overline{z}^\nu \langle H, f\rangle ^i$ with $|\mu +\nu
|+i=1$, due to \eqref{flow6} (applied for $\omega =\omega _0$). By
$\langle \rmi \beta \alpha _2\Sigma _1\Sigma _3   f,  f\rangle =\|f\|_2^2$,
we have $\langle \rmi \beta \alpha _2\Sigma _1\Sigma _3
\mathcal{H}_{\omega _0 +\delta \omega } f,   f\rangle =\langle \rmi
\beta \alpha _2\Sigma _1\Sigma _3  \mathcal{H}_{\omega  _0 } f,
f\rangle +\frac{\|f\|^2\varrho}{2} +\widetilde{F}_2$ where $\widetilde{F}_2$ can
be absorbed
in $j=2$ in $\widetilde{\resto^{(2)}}$ and $\frac{\|f\|^2\varrho}{2}$ can be
absorbed in $\psi$ when restricted to $\varrho=\|f\|_2^2$.

Notice that  $\widehat{\resto} ^{(1)}_2$ is a remainder term obtained from
terms in $\mathcal{E}$ of Lemme \ref{lem:flow1}.
\qed


\section{Birkhoff normal forms}
\label{section:Birkhoff}

\subsection{Normal form}
\label{subsec:Normal form}
Here again and in the following sections, we use  the notation $\lambda
_j^0=\lambda _j(\omega _0)$.
Set  $
 \mathcal{H}:=\mathcal{H }_{\omega _{0}}
P_c(\mathcal{H} _{\omega _{0}} )$.

\begin{definition}
\label{def:normal form} A function $Z(z,f)$ is in normal form if it
is of the form
$ Z=Z_0+Z_1$
where we have finite sums of the following types:
\begin{equation}
\label{e.12a}Z_1= \sum _{|\lambda ^0 \cdot(\nu-\mu)|>m-\omega _{0}}
z^\mu \overline{z}^\nu \langle \im \beta \alpha _2 \Sigma _1\Sigma _3 G_{\mu
\nu}(\| f
\| _2^2 ),f\rangle
\end{equation}
with $H_{\mu \nu}( x,\varrho )\in  C^\infty ( \mathbb{
R}_{\varrho},H_x^{K,S})$ for all $K$, $S$;
\begin{equation}
\label{e.12c}Z_0= \sum _{\lambda  ^0\cdot(\mu-\nu)=0} a_{\mu , \nu}
(\| f \| _2^2)z^\mu \overline{z}^\nu
\end{equation}
and $a_{\mu , \nu} (\varrho  )\in  C^\infty ( \mathbb{ R}_\varrho ,
\mathbb{C})$. We will always assume the symmetries
\eqref{eq:ExpHcoeff2}. \qed\end{definition}

We consider the coefficients of the type of \eqref{eq:ExpH2} (below it will be
those of the $H^{(r)}_2$ in Theorem
\ref{th:main}) and thus let, for $\delta _j= (\delta _{1j},...,
\delta _{nj})$,
\begin{equation} \label{eq:lambda}\lambda _j=
\lambda _j ( \| f\| _2^2 ) =\lambda_j^0 +  k
_{\delta _j\delta _j} (\| f\| _2^2 ), \quad \lambda =
(\lambda _1, \cdots, \lambda _m).\end{equation}
Let
\begin{equation}  \label{eq:Diag} D_2=
\sum _{j=1}^n \varepsilon _j
 \lambda  _j  ( \| f\| _2^2 )  |z^j|  + \frac{1}{2} \langle \im \beta
\alpha _2 \Sigma _1\Sigma _3 \mathcal{H}_{\omega
_0} f,   f\rangle .
\end{equation}
We have ($\lambda _j' (\varrho)$ is the
derivative in $\varrho$) for $F$ a scalar valued function
that, summing on repeated indexes,
 \begin{equation} \label{PoissBra1}
\begin{aligned} &\{ D_2, F  \} :=dD_2 (X_F)= \partial _j D_2
(X_F)_j+
\partial _{\overline{j}}
D_2 (X_F)_{\overline{j}} +\langle \nabla _fD_2, (X_F)_f\rangle \\
&= -\im   \partial _j D_2\partial _{\overline{j}} F+\im \ \partial
_{\overline{j}} D_2\partial _jF-    \langle \nabla _fD_2,\beta
\alpha _2 \Sigma _3 \Sigma _1\nabla _fF \rangle =\\& \im  \lambda _j
{z}_j\partial _jF - \im  \lambda _j \overline{z}_j
 \partial _{\overline{j}} F+\im \langle  \mathcal{H}f,
   \nabla _fF \rangle  +2\im \lambda _j'
(\| f \| _2^2) |z_j|^2 \langle   f , \Sigma _3  \nabla _fF \rangle
.
\end{aligned}
\end{equation}
In particular, we have,  for $G=G(x)$, (we use   $\Sigma _1\im
\Sigma _2=\Sigma _3$)
 \begin{equation} \label{PoissBra2} \begin{aligned} & \{
D_2, z^{\mu} \overline{z}^{\nu}   \} =  \im  \lambda  \cdot (\mu
 - \nu  ) z^{\mu} \overline{z}^{\nu} ,\\& \{ D_2, \langle
\im \beta \alpha _2\Sigma _1\Sigma _3 G ,f\rangle \} =  \im \langle
\mathcal{H}f ,\im \beta \alpha _2\Sigma _1\Sigma _3G \rangle -2 \,
\im \sum _{j=1}^{n}\lambda ' _j |z_j|^2 \langle\im \beta \alpha
_2\Sigma   _1 f,G\rangle \\& =  -\im \langle f ,\im \beta \alpha
_2\Sigma _1\Sigma _3\mathcal{H}G \rangle -2 \, \im \sum
_{j=1}^{n}\lambda ' _j |z_j|^2 \langle\im \beta \alpha _2\Sigma   _1
f,G\rangle ,
\\& \{ D_2, \frac{1}{2} \| f\| _2^2 \} =\{ D_2, \frac{1}{2} \langle
f ,\im \beta \alpha _2 \Sigma _1 f \rangle \} =-\im \langle \mathcal{H}f,
 \im  \beta \alpha _2\Sigma _1 f   \rangle = -\im \langle V_{\omega _0}f,
 \im  \beta \alpha _2\Sigma _1 f     \rangle .
\end{aligned}
\end{equation}
In the sequel we will   prove  that $\| f\| _2$ is small.
\begin{remark} \label{rem:finitesums}
We will consider only
$ |\mu  +\nu |\le 2N +3.$
Then, $\lambda
^0\cdot(\mu-\nu)\neq 0$ implies $|\lambda ^0\cdot(\mu-\nu)|\ge c
>0$ for some fixed $c$, and so we can assume also $|\lambda
 \cdot(\mu-\nu)|\ge c/2$. Similarly $|\lambda
^0\cdot(\mu-\nu)|<m-\omega  _{0} $ (resp.   $|\lambda
^0\cdot(\mu-\nu)|>m-\omega  _{0} $) will be assumed equivalent to
$|\lambda
 \cdot(\mu-\nu)|<m-\omega  _{0} $ (resp.   $|\lambda
 \cdot(\mu-\nu)|>m-\omega  _{0} $).
\end{remark}

\begin{lemma}[Homological equation]
\label{sol.homo} Consider

\begin{eqnarray}
\label{eq.stru.1} K =\sum _{|\mu +\nu |=M_0+1} k_{\mu\nu} (\| f \|
_2^2 ) z^{\mu} \overline{z}^{\nu} + \sum _{|\mu +\nu |=M_0} z^{\mu}
\overline{z}^{\nu}
 \langle \im \beta \alpha _2 \Sigma _1\Sigma _3  K_{\mu   \nu
}(\| f \| _2^2)
  , f \rangle .
\end{eqnarray}
 Suppose that all the terms in \eqref{eq.stru.1} are not in normal
 form and that the symmetries
\eqref{eq:ExpHcoeff2} hold.   Consider
\begin{equation}
\label{solhomo} \begin{aligned} &\chi =  \sum_{|\mu +\nu
|=M_0+1}\frac{ k_{\mu\nu}(\| f \| _2^2)}{
  \im \lambda  \cdot (\mu
 - \nu  )} z^\mu \overline{z}^\nu \\&
{ +}   \sum_{|\mu +\nu
|=M_0}
   z^\mu \overline{z}^\nu   \langle  \im \beta \alpha _2 \Sigma _1\Sigma _3
    \frac{1}{   \im  (\lambda  \cdot (\mu
 - \nu  )-\mathcal{H}) }  K_{\mu   \nu
}(\| f \| _2^2)
  , f \rangle .\end{aligned}
\end{equation}
Then we have $ { \left\{ D_2, \chi   \right\}} =K+L$
with,  summing on repeated indexes,
\begin{equation*}
\label{eq:RestohomologicalEq} \begin{aligned} & L=   { -}2
   \frac{k_{\mu\nu}'  }{(\mu -\nu )\cdot \lambda } z^{\mu}
\overline{z}^{\nu}   \langle
 V_{\omega _0} f,\im \beta \alpha _2\Sigma _1
   f   \rangle \\&
{ -}2 \lambda ' _j z^\mu \overline{z}^{\nu} |z_j|^2
   \left \langle \im \beta \alpha _2 \Sigma _1 f , \frac{1}{(\mu -\nu )\cdot
\lambda
   -\mathcal{H}}K_{\mu \nu } \right \rangle \\&
{ +}2 \lambda ' \cdot(\mu -\nu ) z^\mu \overline{z}^{\nu} |z_j|^2
   \left \langle   f ,\im \beta \alpha_2 \frac{1}{\left ( (\mu -\nu )\cdot
\lambda
   -\mathcal{H}\right ) ^2}K_{\mu \nu } \right \rangle \langle
 {  V_{\omega_0}}f,
   i\beta \alpha_2 \Sigma_1 f   \rangle
\\&
{ -}2z^{\mu} \overline{z}^{\nu} \left \langle  f ,\Sigma _3 \Sigma
_1\frac{1}{(\mu -\nu )\cdot \lambda
   -\mathcal{H}}K_{\mu \nu } '\right \rangle
\langle  {  V_{\omega_0}}f,
   i\beta \alpha_2 \Sigma_1 f   \rangle
   . \end{aligned}
\end{equation*}
If the  coefficients in \eqref{eq.stru.1}  satisfy \eqref{eq:ExpHcoeff2}, the same is true for the coefficients in \eqref{solhomo}.
\end{lemma}
\proof The proof follows by the tables \eqref{PoissBra2}, by the
product rule for the derivative and by the symmetry properties of
$\mathcal{H}$.   \qed

 \subsection{Canonical transformations}
\label{sec:Canonical}

First we consider functions
\begin{equation}
\label{chi.1}    \chi  =   \sum_{|\mu +\nu |=M_0 +1}b_{\mu   \nu
}(\| f\| _2^2) z^{\mu} \overline{z}^{\nu}   + \sum_{|\mu +\nu |=M_0
}z^{\mu} \overline{z}^{\nu}
 \langle  \im   \beta  \alpha _2\Sigma _1\Sigma _3B_{\mu   \nu
}(\| f\| _2^2)
  , f \rangle   \end{equation}
where    $ b_{\mu   \nu }(\varrho )\in C^{\infty
}(\mathbb{R}_\varrho , \mathbb{C})$ and $ B_{\mu   \nu }(x,\varrho)
\in C^{\infty }(\mathbb{R} , P_c(\omega _0)H^{k,s}_x(\mathbb{R}^3,
\mathbb{C}^8))$ for all $k$ and $s$. Assume
\begin{equation} \label{chi.11} b_{\mu   \nu } = ({b}_{\nu
\mu   })^*\text{ and }\im \beta \alpha _2\Sigma _1 B_{\mu   \nu }=- ({B}_{ \nu
\mu})^* \text{ for all indexes}. \end{equation}
The canonical transformations used in the proof of Theorem \ref{th:main} are
compositions of the    Lie
transforms  $\phi :=
 \phi^\tau\big|_{\tau=1},$
with   $\phi ^\tau$ the flow
of the Hamiltonian vector field $X_{\chi}$ (with
respect to $\Omega _0$ and only in $(z,f)$).  Let for  $K>0$ and $S>0$
fixed and large
\begin{equation} \label{chi.12} \|  \chi \| =
\sum |b_{\mu   \nu }(\|  f\| _2^2)|+\sum \|B_{\mu   \nu }(\|  f\|
_2^2)\| _{H^{ K , S }}. \end{equation}
Then, the following lemma can be proved like Lemma 9.2 \cite{Cuccagna1}.

\begin{lemma}\label{lie_trans}
Consider the $\chi$ in \eqref{chi.1} and its Lie transform $\phi$.
Set $(z',f')=\phi (z,f)$.   Then there are $ \mathcal{G}(z,f,
\varrho)$, $\Gamma (z,f,\varrho)$, $\Gamma _0 (z,f,\rho)$ and
$\Gamma _1 (z,f,\rho)$ with the following properties.

\begin{itemize}
\item[(1)]  $\Gamma \in C^\infty
(  \U ^{-K',-S'}, \mathbb{C}^{ n}) $, $\Gamma _0, \Gamma _1\in
C^\infty ( \U ^{-K',-S'}, \mathbb{R}) $, with $\U ^{-K',-S'}\subset
\mathbb{C}^{ n} \times H^{-K',-S'}_c(\omega _0)\times \mathbb{R}$
an appropriately small neighborhood of the origin.
\item[(2)] $  \mathcal{G}\in C^\infty
(\U ^{-K',-S'} ,  H^{K,S}_c(\omega _0) )  $  for any $K,S$.
\item[(3)] The transformation $\phi$ is of the following form:
\begin{eqnarray}
\label{lie.11.a}&    z'  =   z  +
 \Gamma (z,f,\|  f\| _2^2) ,\\
\label{lie.11.b} & f' = e^{ \im \Gamma _0 (z,f,\|  f\|
_2^2)P_c(\omega _0) \Sigma _3}f +\mathcal{G}(z,f,\|  f\| _2^2)  .
\end{eqnarray}

\item[(4)] There are constants   $c_{K',S'}$ and $c_{K, S,K',S'}$
such that
\begin{eqnarray}\label{lie.11.c}
 |\Gamma (z,f,\|  f\| _2^2)| &\leq& c_{K',S'}  (\| \chi \|
+\text{\eqref{lie.11.f}})  |z | ^{M_0-1}
 ( |z|+ \norma{f }
_{H^{-K',-S'}} ), \\   \|\mathcal{G}(z,f,\|  f\|
_2^2 )\|_{H^{K,S}}&\leq& c_{K, S,K',S'}   (\| \chi \| +\text{\eqref{lie.11.f}})
 |z | ^{{M_0} } , \\
 \label{lie.11.e} |\Gamma _{0}(z,f,\|  f\| _2^2)| &\leq& c_{K',S'} |z |^{M_0-1
} (
|z | +\| f   \| _{H^{ -K',-S'}} )^2.
\end{eqnarray}

\item[(5)] We have
\begin{eqnarray} \label{lie.11.h}  & \| f'\| _2^2 =
\| f   \| _2^2  +
 \Gamma _1(z,f,\|  f\| _2^2),
\\& \label{lie.11.f}
\left | \Gamma _1(z,f,\|  f\| _2^2) \right | \le
    C |z|^{M_0 -1} ( |z|+\| f \|
_{H^{-K',-S'}} ) ^2     .
\end{eqnarray}

\item[(6)] We have
\begin{equation} \label{eq:prime f1}
\begin{aligned} & e^{ \im \Gamma _0P_c(\omega _0) \Sigma _3    }
=e^{ \im \Gamma _0  \Sigma _3    } +T (\Gamma _0),
\end{aligned}
\end{equation}
where $T (r)\in C^\infty (\mathbb{R}, B ( H ^{-K',-S'},H ^{ K ,  S
}) ) $ for all $(K,S,K',S')$, with norm $$\| T (r) \| _{B ( H
^{-K',-S'},H ^{ K ,  S })} \le C (K,S,K',S') |r|.$$ More
specifically, the range of $T(r)$ is a subspace of $ \mathbf{X}_d
(\mathcal{H})+ \mathbf{X}_d (\mathcal{H}^*).$

\end{itemize}
\end{lemma}

The crux of  this section is   the
following result.

\begin{theorem}
\label{th:main}   For any integer $r\ge 2$   there are a
neighborhood $ \mathcal{U}^{1,0}$ of $  \{ (  0, 0) \}$ in
$\widetilde{\Ph}^{1,0}$, see \eqref{eq:PhaseSpace}, and a smooth
canonical transformation $\Tr_r: \mathcal{U}^{1,0}\to
\widetilde{\Ph}^{1 ,0}$ s.t.
\begin{equation}
\label{eq:bir1} H^{(r)}:=H\circ \Tr_r=d(\omega _0) -\omega _0\|
u_0\| _2^2+ \psi (\|f\| _2^2)+H_2^{(r)}+Z^{(r)}+\resto^{(r)}.
\end{equation}
     where:
\begin{itemize}
\item[(i)] $H_2^{(r)}=H_2^{(2)}$ for $r\ge 2$, is of the
form \eqref{eq:ExpH2} where  $k_{\mu \nu}^{(r)}(\| f\| _2)$
satisfy \eqref{eq:ExpHcoeff1}--\eqref{eq:ExpHcoeff2};

\item[(ii)]$Z^{(r)}$ is in normal form, in the sense of Definition
\ref{def:normal form} above, with monomials of
degree $\le r$  whose coefficients satisfy \eqref{eq:ExpHcoeff2};
\item[(iii)] the transformation   $\Tr _r$ is of the form
 \eqref{lie.11.a}--
  \eqref{lie.11.b} and satisfies \eqref{lie.11.c}--
  \eqref{lie.11.e} for $M_0=1$;
\item[(iv)] we have $\resto^{(r)} = \sum _{d=0}^6\resto^{(r)}_d$
and for all $(K,S, K',S')$ positives there is a neighbourhood
$  \mathcal{U}^{-K',-S'}$ of $   \{ (  0, 0) \}$ in
$\widetilde{\Ph}^{-K',-S'}$
 such that
\begin{itemize}
\item[(iv.0)]
\begin{equation} \resto^{(r)}_0=
\sum _{|\mu +\nu |  = r+1 } z^\mu \overline{z}^\nu \int
_{\mathbb{R}^3}k_{\mu \nu}^{(r)}(x,z,f,f(x),\| f\| _2^2 ) dx
\nonumber
\end{equation}
and for   $k_{\mu \nu}^{(r)}(z,f,\eta , \varrho )$   with $^t\eta
= (\zeta , C\zeta )$, $\zeta \in \C^4$ we have for
$(z,f)\in \mathcal{U}^{-K',-S'}$ and $|\varrho |\le 1$

\begin{equation}\label{eq:coeff a} \|\nabla _{ z,\overline{z},
\zeta,C\zeta,f,\varrho} ^lk_{\mu \nu}^{(r)}(\cdot
,z,f,\eta , \varrho )\| _{H^{K,S}(\R ^3, \C  )} \le C_l \text{ for
all $l$};
\end{equation}

\item[(iv.1)]
\begin{equation} \resto^{(r)}_1=
\sum _{|\mu +\nu |  = r  }z^\mu \overline{z}^\nu \int
_{\mathbb{R}^3} \left [\im \beta \alpha _2 \Sigma _1 \Sigma _3H_{\mu \nu
}^{(r)}(x,z,f,f(x), \| f\| _2^2 )\right ]^T f(x)  dx  \nonumber
\end{equation}

\begin{equation}\label{eq:coeff G}\text{with }  \|\nabla _{ z,\overline{z},
\zeta,C\zeta,f,\varrho} ^l H_{\nu \mu }^{(r)}(\cdot
,z,f,\eta , \varrho )\| _{H^{K,S}(\R ^3, \C ^8)} \le C_l \text{ for
all $l$};
\end{equation}

\item[(iv.2--5)] for $2\le d \le 5$,

\begin{equation} \resto^{(r)}_d=
 \int
_{\mathbb{R}^3} F_d^{(r)}(x, z ,f,f(x),\| f\| _2^2)  f^{\otimes
d}(x) dx +  \widehat{\resto}^{(r)}_d,\nonumber
\end{equation} with   for
any $l$

\begin{equation}\label{eq:coeff F}   \|\nabla _{ z,\overline{z},
\zeta,C\zeta,f,\varrho} ^l F_d^{(r)} (\cdot ,z,f,\eta ,
\varrho )\| _{H^{K,S} (\mathbb{R}^3,   B   (
 (\mathbb{C}^8)^{\otimes d},\mathbb{C} )} \le C_l,
\end{equation}
  with $F_2^{(r)}(x, 0 ,0,0,0)=0$  and
  with $\widetilde{\resto}^{(r)}_d (z ,f, \| f\| _2^2)$ s.t.

\begin{equation}\label{eq:Rhat}\begin{aligned} &
\widehat{\resto}^{(r)}_d (z ,f, \varrho )\in C^\infty
(\mathcal{U}^{-K',-S'}\times \R ,\R ), \\& |
\widehat{\resto}^{(r)}_d (z ,f, \varrho )|\le C \| f \|
_{H^{-K',-S'}}^d , \\&    | \widehat{\resto}^{(r)}_2 (z ,f, \varrho
)|  \le C (|z|+|\varrho |+ \| f \| _{ H^{-K',-S'}}) \| f \| _{
H^{-K',-S'}}^2;
\end{aligned}\end{equation}

\item[(iv.6)] $ \resto^{(r)}_6=
\int _{\mathbb{R}^3}G (\frac12(P_c(\omega )f(x))\cdot i \alpha_2\Sigma_3\Sigma_1
(P_c(\omega )f(x))\,dx.$
\end{itemize}

\end{itemize}

\end{theorem}

The proof of Theorem \ref{th:main} is the same of Theorem 9.1 in
\cite{Cuccagna1}
and we skip it. The ingredients needed in the proof
(in particular the notion of normal form) are described above.

\section{Non linear dynamics}\label{Sec:NonLinDyn}

 \subsection{Dispersion}
 \label{sec:dispersion}\index{r}\index{$H$}\index{$H_2$}\index{$\lambda_j$}
\index { $\lambda$}\index{$Z_a $}\index{$\resto $}\index{$H_{\mu \nu}$}
We apply Theorem \ref{th:main} for $r=2N_1+1 $  (recall
$N_j\lambda _j<m-\omega _0 <(N_j+1)\lambda _j).$ In the rest of the
article we work with the Hamiltonian $H^{(r)}$. We will drop the upper
index. So we will set $H=H^{(r)}$, $H_2=H_2^{(r)}$, $\lambda
_j=\lambda _j^{(r)}$,    $\lambda  =\lambda  ^{(r)}$,
$Z_a=Z_a^{(r)}$ for $a=0,1$ and $\resto =\resto ^{(r)}$. In
particular we will denote by $H_{\mu \nu}$ the coefficients $G_{\mu
\nu}^{(r)}$ of $Z_1^{(r)}$. We will show:
\begin{theorem}\label{proposition:mainbounds} Fix $p_0 >2$ and $\tau _0>1$.
Let $\frac{2}{p}=\frac{3}{2}(1-\frac{2}{ q})$ and $\alpha (q) =
\frac{2}{p}$, i.e. $(1+ \frac{\theta}{2})(1-\frac{2}{q})
=\frac{2}{p}$ with $\theta  =1$ in Theorem \ref{Thm:Strichartz flat}.
Consider {  $k_0\ge 4$}, $k_0\in \Z$, $\epsilon \in (0, \varepsilon _0)$ and $\varepsilon _0>0$  as in Theorem \ref{th:AsStab}).  Then there is a
fixed
$C >0$ such that for $\varepsilon _0>0$ sufficiently small
   and for $p\ge p_0$ we have the following inequalities:
\begin{eqnarray}
&   \|  f \| _{L^p_t( [0,\infty ),B^{ k_0 -\frac{2}{p}} _{q,2})}\le
  C \epsilon ;
  \label{Strichartzradiation} \\& \|  f \| _{L^2_t( [0,\infty ),H^{ k_0 ,-\tau
_0}_x)}\le
  C \epsilon
  \label{Smootingradiation}\\& \|  f \| _{L^2_t( [0,\infty ),L^{ \infty}_x)}\le
  C \epsilon
  \label{endStrich}
\\& \| z ^\mu \| _{L^2_t([0,\infty ))}\le
  C \epsilon \text{ for all multi indexes $\mu$
  with  $\lambda\cdot \mu >m-\omega _0 $} \label{L^2discrete}\\& \| z _j  \|
  _{W ^{1,\infty} _t  ([0,\infty )  )}\le
  C \epsilon \text{ for all   $j\in \{ 1, \dots , n\}$ }
  \label{L^inftydiscrete} .
\end{eqnarray}
\end{theorem}
  Due to time reversibility, it is easy to conclude that
\eqref{Strichartzradiation}--\eqref{L^inftydiscrete} are true over
the whole real line.

The proof of
 Theorem
\ref{proposition:mainbounds}  involves a  standard continuation
argument following \cite[End of proof of Theorem II.2.1]{Sogge}.
We assume
\begin{eqnarray}
&   \|  f \| _{L^p_t( [0,T ],B^{  k_0-\frac{2}{p}} _{q,2})}+ \|  f
\| _{L^2_t([0,T],H^{  k_0 ,-\tau _0}_x)}+ \|  f \| _{L^2_t(
[0,T],L^{ \infty}_x)}  \le
  C _1\epsilon   \label{4.4a}
\\& \| z ^\mu \| _{L^2_t([0,T])}\le
 C_2 \epsilon \text{ for all multi indexes $\mu$
  with  $\omega \cdot \mu >m-\omega _0 $} \label{4.4}\\& \| z _j  \|
  _{W ^{1,\infty} _t  ([0,T]  )}\le
  C_3 \epsilon \text{ for all   $j\in \{ 1, \dots , n\}$ }
  \label{4.4bis}
\end{eqnarray}
for fixed sufficiently large  constants $C_1$--$C_3$. Notice that there is an
$ \varepsilon _1>0$ such that this assumption is true for all $ |z(0)|+\| f (0) \| _{H^{k_0} }
<\varepsilon _1  $ if say $T\in (0, 1]$.   We  then
prove that  there exists a fixed $ \varepsilon _0\in (0, \varepsilon _1)$, with  $ \varepsilon _0=\varepsilon _0(C_1,C_2,C_3) $,  such that
 for $\epsilon \in (0, \varepsilon _0)$, \eqref{4.4a}--\eqref{4.4bis} imply the same estimate but with $C_1$--$C_3$ replaced
by $C_1/2$--$C_3/2$. This implies that the set of $T$ such that \eqref{4.4a}--\eqref{4.4bis} is  open
in $\R^+$. Since it is also closed, it is all $\R^+$.
Then \eqref{4.4a}--\eqref{4.4bis}  hold   with
$[0,T]$ replaced by $[0,\infty )$ for all $ |z(0)|+\| f (0) \| _{H^{k_0} }
<\epsilon <  \varepsilon _0  $.

The proof  of  Theorem
\ref{proposition:mainbounds}  consists in three main steps.
\begin{itemize}
\item[(i)] Estimate $f$ in terms of $z$.
\item[(ii)] Substitute the variable $f$  with a
new "smaller" variable $g$ and find smoothing estimates for $g$.
\item[(iii)] Reduce the system for $z$ to a closed system involving
only the $z$ variables, by insulating the  part of $f$  which
interacts with $z$, and by decoupling the rest (this reminder is
$g$). Then clarify the nonlinear Fermi golden rule.
\end{itemize}

\paragraph{Step (i)}  Using the Proposition \ref{Lemma:conditional4.2} below, we will choose $C_1>2K_1(C_2)$. This tells us that if we get upper bounds on  $C_2$ and $C_3$,
and this is done in Sect. \ref{sec:FGR}, then we will have proved Theorem
\ref{proposition:mainbounds}.

\begin{proposition}\label{Lemma:conditional4.2}  Assume \eqref{4.4a}--\eqref{4.4bis}. Then there exist constants $C=C(C_1,C_2,C_3), K_1(C_2)$,
    such that, if
  $C(C_1,C_2,C_3) \epsilon  $  is sufficiently small,   then we have
\begin{eqnarray}
&  \|  f \| _{L^p_t( [0,T ],B^{  k_0-\frac{2}{p}} _{q,2})}+ \|  f \|
_{L^2_t([0,T],H^{  k_0 ,-\tau _0}_x)}+\|  f \|_{ L^2_t( [0,T],L^{
\infty}_x)}\le  K_1(C_2) \epsilon  \ .
  \label{4.5}
\end{eqnarray}
\end{proposition}
\begin{proof}
Consider $Z_1$ of the form \eqref{e.12a}. Set:
 \begin{equation}\label{eq:H^0}\index{$ H_{\mu \nu}^0$}\index{$\lambda
^0_j$} H_{\mu
\nu}^0=H_{\mu \nu}(\| f \| _2^2 ) \mbox{ for } \| f \| _2^2=0;     \lambda
^0_j=\lambda _j(\omega _0).\end{equation}

Then we have (with
finite sums)
\begin{equation}\label{eq:f variable} \begin{aligned}  &\im \dot f -
\mathcal{H}f - 2    (\partial _{ \| f \| _2^2} H) P_c(\omega _0)
\Sigma _3  f = \sum _{\substack{|\lambda  ^0 \cdot(\nu-\mu)|>m-\omega _0,\\
|\mu+\nu|\leq 2N_1+1}}
z^\mu \overline{z}^\nu
  H_{\mu \nu}^0  \\&  +
  \sum _{\substack{|\lambda ^0 \cdot(\nu-\mu)|>m-\omega _0,\\
|\mu+\nu|\leq 2N_1+1}} z^\mu
\overline{z}^\nu  (H_{\mu \nu} -  H_{\mu \nu}^0)    +\im  \beta \alpha _2 \Sigma
_3 \Sigma _1\nabla _f \resto - 2    (\partial _{ \| f \| _2^2} \resto )
P_c(\omega _0)
\Sigma _3  f.
\end{aligned}\end{equation}
In order to obtain bounds on $f$, we need bounds on the right hand term of the
equation especially the last two terms. They are provided by the following
lemma.
\begin{lemma}\label{lem:bound remainder}
Assume \eqref{4.4a}--\eqref{4.4bis}  and consider a fixed  $\tau_0>1$. Then
there is a   constant $C=C(C_1,C_2,C_3)$ independent of $\epsilon$
such that the following is true: we have $$\beta \alpha _2 \Sigma _3
\Sigma _1 \nabla _{  f} \resto  - 2    (\partial _{ \| f \| _2^2} \resto )
P_c(\omega _0)
\Sigma _3  f= R_1+ R_2$$ with
\begin{equation} \label{bound:R1R2}\begin{aligned}
 &
 \| R_1 \| _{ H^{ k_0 }_x } \le   C( C_1,C_2,C_3)  (|z  |
^{2N_1+2}+    \| f\|  _{L^{\infty}}^2 \|    f  \| _{ H^{k_0}_x}) \\& \|
   R_2 \| _{ H^{ k_0,\tau _0}_x }\le  C( C_1,C_2,C_3)  ( |z|+    \| f \| _{L ^{2 }
_x}^2+ \| f \| _{H^{k_0,-\tau _0 }  _x})  \| f \| _{H^{k_0,-\tau _0 }  _x}
.\end{aligned}
\end{equation}
In particular we have for some other fixed constant $C=C(C_1,C_2,C_3)$,
\begin{equation}
\label{bound1:z1} \| R_1 \| _{L^1_t([0,T],H^{ k_0 }_x)}+\|
   R_2 \| _{L^{2
 }_t([0,T],H^{ k_0,\tau _0}_x)}\le C( C_1,C_2,C_3) \epsilon^2.
\end{equation}

\end{lemma}
\proof  \eqref{bound1:z1} is a consequence of \eqref{bound:R1R2} and
\eqref{4.4a}--\eqref{4.4bis}.  We focus on \eqref{bound:R1R2}.
For $d\le 1 $ and arbitrary fixed $(S,K)$  we have $\nabla
_{  f}\resto _d \in H^{S,K} $.
By   (iv0--iv1)   Theorem \ref{th:main}
  \begin{equation*} \label{estimate01}\| \nabla _{  f}\resto _0  \|_{ H^{S,K} } +
\|
\nabla _{  f}\resto _1  \|_{ H^{S,K} }\le C|z  |
^{2N_1+2}.\end{equation*} These terms can be
absorbed in $R_1$.
For $2\le d \le 5$ we have
\begin{equation*}
\label{schematic-1}\begin{aligned} &  \Sigma _3 \Sigma _1\nabla _f
\widehat{\resto} _d - 2    (\partial _{ \| f \| _2^2} \widehat{\resto} _d )
P_c(\omega _0) \Sigma_3f=  \Sigma _3 \Sigma _1\nabla _f \widehat{\resto} _d  (z
,f, \rho)
,
\end{aligned}
\end{equation*}
computed at $\rho =\| f\| _2^2$. By \eqref{eq:Rhat} we obtain
\begin{equation*}
\label{schematic-2}\begin{aligned} &  \|   \nabla _f \widehat{\resto} _d  (z ,f,
\rho)\| _{H^{K',S'}} \le C \| f\|  _{H^{-K',-S'}}^{d-1} \text{ for $3\le d\le 5$
and }\\&  \|   \nabla _f \widehat{\resto} _2  (z ,f, \rho)\| _{H^{K',S'}} \le C
\| f\|  _{H^{-K',-S'}}^{2} +C |z|  \,  \| f\|  _{H^{-K',-S'}}  .
\end{aligned}
\end{equation*}
Since $K'$ and $S'$ are arbitrarily large, we have $\| f\|  _{H^{-K',-S'}}\le \|
f\|  _{H^{k_0,-\tau _0}}$. So
these terms can be absorbed in $R_2$.  Other terms are treated as in
\cite[Lemma 7.5]{BambusiCuccagna} : For $d=2,3,4,5$ we have schematically

\begin{equation}
\label{schematic11}\begin{aligned} &    F_d(x,z,f,  f(t,\cdot ) , \rho )
f^{\otimes (d-1)}(t,\cdot )+   \partial _w F_d(x,z,f, w, \rho )_{ w= f(t,\cdot
)}  f^{\otimes d }(t,\cdot )  \\& + \nabla _{g}\left (  \int_{\R^3}  F_d  (x,z ,
g , f(t,x), \| f (t) \| ^2 _{L^2_x})[f(t,x)]^{\otimes d }dx\right ) _{g=f} .
\end{aligned}
\end{equation}
The first line of
\eqref{schematic11} has $H^{k_0,\tau _0}_x$ norm bounded, for some fixed
sufficiently large $ \textbf{N}$,
 by
\begin{equation}
\label{schematic-3}\begin{aligned} &  \widetilde{C}\|  \langle x \rangle
^{\textbf{N}}   F_d(x,z,f,  f(t,x ) , \rho ) \| _{W^{k_0,\infty }_x}  \| f \|
_{H^{k_0,-\tau _0 }  _x}^{d-1} \\& + \widetilde{C}  \|  \langle x \rangle
^{\textbf{N}}   \partial _w F_d(x,z,f, w, \rho )_{ w= f(t,x )} \|
_{W^{k_0,\infty }_x} \| f \| _{H^{k_0,-\tau _0 }  _x}^{d } \le C   \| f \|
_{H^{k_0,-\tau _0 }  _x}^{d -1} +C   \| f \| _{H^{k_0,-\tau _0 }  _x}^{d  }.
\end{aligned}
\end{equation}
When these terms are bounded by  $\| f \| _{H^{k_0,-\tau _0 }  _x}^{d _1 } $
for $d_1\ge 2$, we can absorb them in  $R_2$. Cases $d_1=1$ come from terms
  in the first line of  \eqref{schematic-3} with $d=2$. By
$F_2(x,0,0,0,0)=0$   these   are less than
\begin{equation*}
\label{schematic-4}\begin{aligned} &  ( |z|+ \| f \| _{H ^{-K', -S'} _x}+  \| f
\| _{L ^{2 } _x}^2)  \| f \| _{H^{k_0,-\tau _0 }  _x}
\end{aligned}
\end{equation*}
and   can be absorbed in $R_2$.
 Looking at   the second line of \eqref{schematic11}  and for $\textbf{N}$
sufficiently large, we have
\begin{equation*} \label{estimate23}\begin{aligned} & \| \nabla _{g}\left (
\int_{\R^3}  F_d  (x,z , g , f(t,x), \| f (t) \| ^2 _{L^2_x})[f(t,x)]^{\otimes d
}dx\right ) _{g=f} \|  _{ H^{ k_0}_x}=\\& \left | \sup _{\| \psi  \|  _{ H^{
-k_0}_x} =1} \int_{\R^3}   D _{ {g}} F_d  (x,z , g , f(t,x), \| f (t) \| ^2
_{L^2_x})_{g=f}[\psi ] [f(t,x)]^{\otimes d }dx \right |  \\&
   \le C \sup _{  \| \psi  \|  _{ H^{ -k_0}_x} =1}  \| D _{ {g}} F_d  (x,z , g ,
f(t,x), \| f (t) \| ^2 _{L^2_x})_{g=f}[\psi ] \| _{L^{\infty,  \textbf{N} } _x}
  \|  f \| _{ H^{k_0,-\tau _0}_x}^{d }  \le  C  \|  f \| _{ H^{k_0,-\tau
_0}_x}^{d }. \end{aligned}
\end{equation*}
So the second line of \eqref{schematic11} can be absorbed in $R_2$.
Finally we consider  $ \nabla _f \resto _6= \Sigma _1g (|f (t,x)|^2/2)f
(t,x) $.
 Then for a fixed   $C   $     we have $
\|  \nabla _f \resto _6  \| _{ H^{k_0}_x}\le C  \|   f   \| _{
L^{\infty }_x} ^2
   \|    f  \| _{ H^{k_0}_x}  .$ \qed

 Denote by $F$ the rhs of
 \eqref{eq:f variable} and set $\varphi  = 2\partial _{ \| f \| _2^2} H$.

 \begin{lemma}\label{lem:conditional4.21} Consider $\im \dot \psi -
\mathcal{H}\psi - \varphi  (t)
 \Sigma_3P_c \psi=F$  where $P_c=P_c(\omega _0)$  and $\psi= P_c\psi$. Let $k\in
\Z$  with $k\ge 0$ and $\tau _0>1$. Then there exist
 $c _0>0$ and
   $C>0$ such that if  $\| \varphi   \| _{L^\infty _t [0,T]}<c _0$   then
   for $p\ge p_0>2$ and for $(p,q) $ as in Theorem \ref{proposition:mainbounds}   we have
\begin{equation} \label{eq:421}  \begin{aligned}
 &     \|  \psi \| _{L^p_t( [0,T ],B^{  k -\frac{2}{p}} _{q,2})\cap
L^2_t([0,T],H^{  k  ,-\tau _0}_x) }\le C\| \psi(0) \| _{H^{  k }}+ C \|  F \|
_{L^1_t([0,T],H^{ k }_x) + L^2_t([0,T],H^{ k  , \tau _0}_x)}  \end{aligned}
\end{equation}
\end{lemma}
\begin{proof}
We apply the argument for the NLS in    Lemma B.2
\cite{NakanishiSchlag},
see also Theorem 1.5 \cite{Beceanu}. A   more precise statement than
Lemma B.2 \cite{NakanishiSchlag} is  in
\cite{BuslaevPerelman2,Cuccagna3},  but the
proof
    does not seem
 easy to reproduce for Dirac. We fix  any $\delta >0$. Let $P_d=P_d(\omega _0)$
 and $\mathcal{H}_0 =\mathcal{H} _{\omega _0,0}$. Consider
\begin{equation} \label{eq:4220}\im \dot Z -
\mathcal{H}P_cZ+\im \delta P_d Z-\varphi \Sigma _3P_cZ=F .
 \end{equation}
Then notice that for $Z(0)=\psi (0)$  the solution of
\eqref{eq:4220} satisfies $Z(t)\equiv  \psi (t)$. We rewrite
\eqref{eq:4220} as
\begin{equation*} \label{eq:4221}\im \dot Z -
\mathcal{H}_0 Z -\varphi \Sigma _3 Z=F +(V-\mathcal{H}P_d -\im
\delta P_d )Z -\varphi \Sigma _3P_dZ.
 \end{equation*}
Let $(V-\mathcal{H}P_d -\im \delta P_d )=V_1V_2$ with $V_2(x)$ a
smooth exponentially decaying and invertible matrix, and with $V_1$
bounded from $H^{k,s'}\to H^{k,s}$ for all
$k$, $s$ and $s'$. For $\mathcal{U}(t)=e^{-\im \Sigma _3\int
_0^t \varphi (t')dt'}$ we have
 \begin{equation} \label{eq:423}
 Z (t)= \mathcal{U}(t)e^{-\im \mathcal{H}_0 t}Z (0) -\im \int _0^t
 e^{ \im \mathcal{H}_0 (t'-t)}\mathcal{U}(t)\mathcal{U}^{-1}(t')
 \left [ F(t')+ V_1 V_2 Z (t ')  - \varphi  (t')
 \Sigma_3P_d Z (t ')  \right ] dt'  .
 \end{equation}
 $ c_0P_dV_2^{-1  }$
 maps $H^{-K',-S'}\to H^{ K , S }$ for arbitrarily  fixed pairs $(K,S)$ and
$(K',S')$. By picking  $c_0$ small
enough,  we can  assume that
  the related operator norms are small.
By Theorems \ref{Thm:Smoothness flat} and \ref{Thm:Strichartz flat}
 \begin{equation*} \label{eq:4231}  \begin{aligned}
 &   \|  Z \| _{L^p_t B^{  k -\frac{2}{p}}
 _{q,2} \cap L^2_t H^{  k  ,-\tau _0}_x  }\le
  C\| Z(0) \| _{H^{  k }}+C \| F \| _{L^1_t H^{  k    }_x
  +L^2_t H^{  k  , \tau _0}_x} \\& +
  \|  V_1  - \varphi  (t )
  \Sigma_3P_d V_2^{-1  }  \| _{
  L^\infty_t B( H^{  k   }_x, H^{  k  , \tau _0}_x)  }  \|    V_2 Z (t  )
    \| _{  L^2_t H^{  k   }_x  }. \end{aligned}
 \end{equation*}
For $\widetilde{T}_0 f (t)=   V_2\int _0^t
 e^{ \im \mathcal{H}_0 (t'-t)}\mathcal{U}(t)\mathcal{U}^{-1}(t') V_1 f(t')
dt',$  by \eqref{eq:423} we obtain

\begin{equation}\begin{aligned}
 &  (I+\im \widetilde{T}_0) V_2Z (t)=
 V_2\mathcal{U}(t)e^{-\im \mathcal{H}_0 t}Z (0)   -\im V_2\int _0^t
 e^{ \im \mathcal{H}_0 (t'-t)}\mathcal{U}(t)\mathcal{U}^{-1}(t')
 \left [ F(t')   - \varphi  (t')
  \Sigma_3P_d Z (t ')  \right ] dt'    \end{aligned}\nonumber \end{equation}
We then obtain
\eqref{eq:421}
if we can show that \begin{equation} \label{ns1}\| (I+\im
 \widetilde{{T}}_0)^{-1} :L^2_t([0,T), H^k( \R ^3))  \to L^2_t([0,T),
 H^k( \R ^3))\| <C_1,
\end{equation} for $c_0C_1 $ smaller than a fixed number.  It is enough to prove \eqref{ns1} with
$\widetilde{T}_0$ replaced by

\begin{equation}\begin{aligned}
 &    {T}_0 f (t)=   V_2\int _0^t
 e^{ \im \mathcal{H}_0 (t'-t)}  V_1 f(t') dt'  .  \end{aligned}\nonumber
\end{equation}
Indeed by Theorem \ref{Thm:wdec} we have
\begin{equation}\begin{aligned}
 & \| (\widetilde{{T}}_0-{T}_0) f\| _{L^2 _{t }H^k_x} \le \|  \int _0^t
 \| V_2 e^{ \im \mathcal{H}_0 (t'-t)}  (e^{\im \Sigma _3\int _t ^{t'}  \varphi
(t'')
 dt''} -1)  V_1 f(t')\| _{H^k _{x }} dt' \| _{L^2 _{t }}\\& \le \widetilde{C}
 c_0^{\frac{1}{4}}
\|  \int _0^t \langle t'-t \rangle ^{-\frac{5}{4}}    \|  f(t')\|
_{H^k_x} dt' \| _{L^2 _{t }} \le C  c_0^{\frac{1}{4}}   \| f(t')\|
_{L^2 _{t }H^k_x} .
\end{aligned}\nonumber
\end{equation}
Set
\begin{equation}\begin{aligned}
 &    {T}_1 f (t)=   V_2\int _0^t
 e^{ (\im \mathcal{H}P_c  +\delta P_d )(t'-t)}  V_1 f(t') dt' = V_2\int _0^t
 ( e^{ (\im \mathcal{H} (t'-t)}P_c
 + e^{  -\delta    |t'-t|}P_d)  V_1 f(t') dt'  .
 \end{aligned}\nonumber \end{equation}
By Lemma \ref{lem:smooth11} we have $\|  T_ 1  :L^2_t([0,T), H^k(
\R ^3))  \to L^2_t([0,T),
 H^k( \R ^3))\| <C _2$ for a fixed $C_2$. For exactly
  the same reasons of \cite{NakanishiSchlag} we have

\begin{equation}\begin{aligned}
 &      (I+\im  {T}_0) (I-\im  {T}_1)= (I-\im  {T}_1)
 (I+\im  {T}_0)=I .
 \end{aligned}\nonumber \end{equation}
This yields \eqref{ns1} with $ \widetilde{{T}}_0$ replaced by  $
{T}_0$ and with $C_1=1+C_2$.
 \end{proof}

 \begin{lemma}\label{lem:conditional4.22} Using the notation of Lemma \ref{lem:conditional4.21},  but this time picking $\tau _0>3/2$,     we have
\begin{equation} \label{eq:endpoint1}  \begin{aligned}
 &     \|  \psi \| _{L^2_t( [0,T ],L^{ \infty} ) }\le C\| \psi(0) \| _{H^{  k _0 }}+ C \|  F \|
_{L^1_t([0,T],H^{ k_0 }_x) + L^2_t([0,T],H^{ k_0  , \tau _0}_x)}  \end{aligned}
\end{equation}
\end{lemma}
\begin{proof} We proceed as above until \eqref{eq:423}. We claim we have
 \begin{equation} \label{eq:endpoint2}  \begin{aligned}
 &   \|  Z \| _{L^2_t L^{ \infty}_x  }\le
  C\| Z(0) \| _{H^{  k _0}}+C \| F \| _{L^1_t H^{  k_0    }_x
  +L^2_t H^{  k_0  , \tau _0}_x}  \\& +
  \|  V_1  - \varphi  (t )
  \Sigma_3P_d V_2^{-1  }  \| _{
  L^\infty_t B( H^{  k _0  }_x, H^{  k _0 , \tau _0}_x)  }  \|    V_2 Z (t  )
    \| _{  L^2_t H^{  k _0  }_x  }. \end{aligned}
 \end{equation}
\eqref{eq:endpoint2} will yield  \eqref{eq:endpoint1} by the argument in Lemma \ref{lem:conditional4.21}. So now we prove \eqref{eq:endpoint2}. We have for $k>1/2$
\begin{equation}    \begin{aligned}
 &  \| e^{-\im \mathcal{H}_0 t}Z (0)\| _{L^2_t L^{ \infty}_x  }\le C \| e^{-\im \mathcal{H}_0 t}Z (0)\| _{L^2_t  B _{6,2} ^{k} }\le C '\| Z (0)\| _{H ^{k+1} } \le C '\| Z (0)\| _{H ^{k_0} } \end{aligned} \nonumber
 \end{equation}
by Theorem \ref{Thm:Strichartz flat}.  Similarly, splitting $F=F_1+F_2$, we have

\begin{equation}    \begin{aligned}
 &  \| \int _0^t e^{ \im \mathcal{H}_0 (t'-t)} \mathcal{U}^{-1}(t') F_1 (t')dt'\| _{L^2_t L^{ \infty}_x  }\le C \| \int _0^t e^{ \im \mathcal{H}_0 (t'-t)} \mathcal{U}^{-1}(t') F_1 (t')dt'\| _{L^2_t  B _{6,2} ^{k} }\\& \le C '\| F_1\| _{L^1_t H ^{k+1} } \le C '  \| F_1\| _{L^1_t H ^{k_0} } .\end{aligned} \nonumber
 \end{equation}
Using $B ^{k}_{\infty ,2} \subset L^{ \infty} $  for $k>0$,   by   Theorem 3.1 \cite{Boussaid} we have for  $k_0>3 $

\begin{equation}    \begin{aligned}
 &  \| \int _0^t e^{ \im \mathcal{H}_0 (t'-t)} \mathcal{U}^{-1}(t') F_2 (t')dt'\| _{L^2_t L^{ \infty}_x  }\le C \left \| \int_0^t  \min \{ |t-t'|   ^{-\frac{1}{2}},
 |t-t'|   ^{-\frac{3}{2}} \}  \|    F_2(t') \|_{B^{k_0}_{1, 2  }} dt' \right \|
_{L_t^{2} } \\& \le C '\| F_2\| _{L^2_t B^{k_0}_{1, 2  } } \le C ''  \| \langle x \rangle ^{ \tau _0}F_2\| _{L^2_t B^{k_0}_{2, 2  } } =C ''  \|  F_2\| _{L^2_t H^{k_0, \tau _0}  }  ,\end{aligned} \nonumber
 \end{equation}
 where we have used $\|   \varphi _j* F_2 \| _{L^1_x}\le \| \langle x \rangle ^{- \tau _0}   \| _{L^2_x}  \| \langle x \rangle ^{ \tau _0} \varphi _j* F_2 \| _{L^2_x}
 \le C''' \| \varphi _j* ( \langle \cdot  \rangle ^{ \tau _0}   F_2 )\| _{L^2_x}$
 for fixed $C'''>0$  and fixed $\tau _0>3/2$.  With $F_2$ replaced by
 $(V_1V_2-\varphi \Sigma _3P_d)Z$ we get a similar estimate. This yields inequality \eqref{eq:endpoint2}.

\end{proof}

{\it Continuation of the proof of Proposition \ref{proposition:mainbounds}.}
By \eqref{eq:f variable}
we can apply to $f$ Lemmas \ref{lem:conditional4.21} and \ref{lem:conditional4.22} by taking $\varphi (t)
=2    (\partial _{ \| f \| _2^2} H)$ and  $F=\text{rhs\eqref{eq:f
variable}}-\varphi (t) [\Sigma _3,P_d]f$. Then
\begin{equation*} \label{eq:4212}  \begin{aligned}
 &     \|  f \| _{L^p_t( [0,T ],B^{  k_0-\frac{2}{p}} _{q,2})\cap
L^2_t([0,T],H^{  k_0 ,-\tau _0}_x)\cap L^2_t( [0,T],L^{ \infty}_x)}\le C\|
f(0)
\| _{H^{  k_0}}+ C \|  F \| _{L^1_t([0,T],H^{ k_0}_x) +  L^2_t([0,T],H^{ k_0
,
s}_x)}. \end{aligned}\nonumber
\end{equation*}
We have
 \begin{equation*} \label{eq:4213}  \begin{aligned}
 &      \|  F \| _{L^1_t H^{ k_0}_x  +  L^2_t H^{ k_0 , \tau _0}_x }
 \lesssim  \sum _{\lambda\cdot \mu >m-\omega _0 }\| z ^\mu \| _{L^2_t }^2 +\|
R_1
\| _{L^1_t H^{ k_0}_x   }+\| R_2 \| _{   L^2_t H^{ k_0 , \tau _0}_x } +
\epsilon
  \|  f \| _{  L^2_t  H^{ - k_0 ,-\tau _0}_x }. \end{aligned}
\end{equation*}
 For $\epsilon $ small this yields  Proposition \ref{proposition:mainbounds}
 by Lemma   \ref{lem:conditional4.21} and by \eqref{4.4}.

\end{proof}

\begin{lemma}
\label{lem:asymptotically f} Assume the conclusions of Theorem
\ref{proposition:mainbounds}. Then there exists a fixed $C>0$ and $f_+' \in
H^{k_0 } $ with $\| f_+'\| _{H^{k_0}}<C \epsilon $ such that for
 for $\vartheta (t) $ the phase in the ansatz \eqref{Eq:ansatzU}
  we have
\begin{equation}\label{scattering111}  \lim_{t\to +\infty}
\left \|  e^{\im \vartheta (t) \Sigma _3}f (t) -
 e^{- \im t D_m    }  {f}_+'  \right
\|_{H^{k_0}}=0.
 \end{equation}
\end{lemma} \begin{proof} For $\psi (t)=f(t)$, for $F=\text{rhs\eqref{eq:f
variable}}-\varphi (t) [\Sigma _3,P_d]f$ and for $t_1<t_2$, we have
 \begin{equation} \label{eq:423scat} \begin{aligned} & \|
\mathcal{U}^{-1}(t_2)e^{ \im \mathcal{H}_0 t_2} f (t _2)-
\mathcal{U}^{-1}(t_1)e^{ \im \mathcal{H}_0 t_1} f (t _1) \| _{H^{k_0}}\\& \le
\| \int   _{t_1}^{t_2}
 e^{ \im \mathcal{H}_0  t' } \mathcal{U}^{-1}(t')\left [ F(t')+ Vf (t ')  -
\varphi  (t')
 \mathcal{U}^{-1}\Sigma_3P_d f (t ')  \right ] dt' \| _{H^{k_0}}  \le \\&
 C  (\sum _{|\lambda  ^0 \cdot \mu |>m-\omega _0}
\| z^\mu \| _{L^2(t_1,t_2)} +\| R_1 \| _{L^1_t([t_1,t_2],H^{ k_0 }_x)}+\|
   R_2 \| _{L^{2
 }_t([t_1,t_2],H^{ k_0,s}_x)}+ \|
   f \| _{L^{2
 }_t([t_1,t_2],H^{ k_0,-\tau _0}_x)}) .\end{aligned}\nonumber
 \end{equation}
Since the latter has limit 0 as $t_1\to +\infty$,
there exists $ {f}_+'\in H^{ k_0 } $ such that
 \begin{equation}   \lim_{t\to +\infty}
\left \| \mathcal{U}^{-1}(t )f (t) -
 e^{ -\im \mathcal{H}_0 t }{f}_+'  \right \|_{H^{k_0}}=0 . \nonumber
 \end{equation}
 From $\mathcal{H}_0=D_m-\omega _0 \Sigma _3$ and $\mathcal{U}^{-1}(t )= e^{\rmi \Sigma _3 \int _0^t \varphi (t') dt'}$
 we have for $ \theta(t) =  - t\omega _0+\int _0^t  \varphi (t')dt'  $ \begin{equation} \label{scattering21}   \lim_{t\to +\infty}
\left \|  e^{\im \theta (t) \Sigma _3}f (t) -
 e^{- \im t D_m    }{f}_+  \right \|_{H^{k_0}}=0 .
 \end{equation}
\eqref{scattering111}  follows from \eqref{scattering21} if we can prove
  $\theta(t)=\vartheta (t)-\vartheta (0)+o(1)$  with $o(1)\to 0$
as $t\to +\infty$. To prove this claim
  we substitute $R$ in \eqref{Eq:syst2} using \eqref{eq:coordinate} and then
  replace $(z,f)$ with the last coordinate system obtained from Theorem \ref{th:main}. Then we get
 \begin{equation}\label{scattering12}   \im \dot f -\mathcal{H}f-(\dot \vartheta +\omega _0 -\sum _{j=2}^{2N+1} \frac{d}{dt} \Gamma _0^{ (j)}) P_c(\omega _0)\Sigma _3 f =G
 \end{equation}
where $G$  is a functional with values in $  L^{\infty}(\R,  L^1_ {x}) $; $\Gamma _0^{ (j)}$ are the functions in the exponent of \eqref{lie.11.b}  for each of the transformations in Theorem \ref{th:main}.
Set now
\begin{equation*} \begin{aligned}  & \chi  (t)=(\dot \vartheta -\omega _0-\sum _{j=2}^{2N+1} \frac{d}{dt} \Gamma _0^{ (j)})  - 2    (\partial _{ \| f \| _2^2} H)  \ , \\& \textbf{G}=-G+\text{rhs \eqref{eq:f variable} }  .
   \end{aligned}
\end{equation*}
Then taking the difference of  the two equations \eqref{eq:f variable}  and  \eqref{scattering12}  we have
\begin{equation*}   \chi  (t) f=\chi  (t) \Sigma _3P_d(\omega _0) \Sigma _3 f  +  \Sigma _3 \textbf{G}  .
\end{equation*}
 $\textbf{G} $  (resp .  $\chi $ ) is a functional from a neighborhood of  the origin in  $L^\infty (\R , H^{k_0} (\R ^3))$ to
$  L^{\infty}(\R,  L^1(\R ^3)) $ (resp .  $  L^{\infty}(\R ) $ ) .
 If  $\chi  (t_0)\neq 0$ for a given   solution, we can find solutions for which
$f_n (t,x ) $  such that  $ f_n (t_0,\cdot )\to f (t_0,\cdot )$ in $H^{k_0}(\R ^3)$,
$\|  f_n (t_0  ) \|  _{L^1(\R ^3)}\nearrow \infty$,
$\textbf{G}  _{ n}(t_0)\to \textbf{G}  (t_0)$
and $\chi  _{ n}(t_0)\to \chi  _{0}(t_0)$.   This yields a contradiction. So    $\chi  \equiv 0$ and   $\textbf{G} =0.$
 This implies       $\dot \vartheta -\omega _0-\sum _{j=2}^{2N+1} \frac{d}{dt} \Gamma _0^{ (j)}= 2     \partial _{ \| f \| _2^2}  H. $
This and the last inequality in \eqref{lie.11.c} yield the claim  $\theta(t)=\vartheta (t)-\vartheta (0)+o(1)$.
  \end{proof}

\paragraph{Step (ii)} In the proof of Theorem
\ref{proposition:mainbounds}
consists in introducing the variable

\begin{equation*}
  \label{eq:g 1}
g=f+ Y \, , \quad Y:=\sum _{|\lambda ^0 \cdot(\mu-\nu)|>m-\omega
_0} z^\mu \overline{z}^\nu
   R ^{+}_{\mathcal{H}}  (\lambda  ^0 \cdot(\mu-\nu) )
    H_{\mu \nu}^0 .
\end{equation*}
Substituting the new variable  $g$ in \eqref{eq:f variable}, the
first line on the rhs of  \eqref{eq:f variable} cancels out. We
have
\begin{equation}\label{eq:g 2} \begin{aligned}  &\im \dot g -
\mathcal{H}g - 2    \partial _{ \| f \| _2^2} H
P_c(\omega _0)  \Sigma _3 g   =\text{second line of \eqref{eq:f variable}} +\\&
2    \partial _{ \| f \| _2^2} H
P_c(\omega _0)  \Sigma _3Y   + \sum _{k=1}^{n}
\left [
\partial _{ z _k}   Y
\partial _{    \overline{z} _k}\left( Z+\resto \right )
-\partial _{  \overline{z}_k}   Y \partial _{
 {z}_k}\left( Z+\resto \right ) \right]
   .
\end{aligned}\end{equation}

\begin{lemma}\label{lemma:bound g} For $\epsilon$ sufficiently
small, $\tau
_1>1$
 and   $C_0=C_0(\mathcal{H})$  a fixed constant, we have
$$\| g
\| _{L^2_t([0,T],L^{2,-\tau _1}_x)}\le C_0 \epsilon + O(\epsilon
^2).$$
\end{lemma} \begin{proof} Set  $F=(\text{second line of \eqref{eq:f variable}}
-\varphi (t) [\Sigma _3,P_d]g ) $. Then, proceeding as in   \eqref{eq:423}, we
have
\begin{equation} \label{eq:g1}  \begin{aligned}
 &   \| g \| _{  L^2_t L^{2,-\tau _1}_x  }\le  \|   e^{-\im t\mathcal{H}_0}Y
(0)\| _{L^2_t L^{2,-\tau _1}_x} +\|   e^{-\im t\mathcal{H}_0}f (0)\| _{L^2_t
L^{2,-\tau _1}_x}
 +C \| F \| _{L^1_t H^{k }_x+ L^2_t H^{k,   \tau _1}_x}\\& +
 \| \int _0^t e^{ \im (t'-t)  \mathcal{H}_0}\text{second line of \eqref{eq:g 2}}
(t') dt'\| _{  L^2_t L^{2,-\tau _1}_x}
  \\& +
  \|  V_1  - \varphi  (t )
  \Sigma_3P_d V_2^{-1  }
   \| _{  L^\infty_t B( L^2_{ x},L^{2,-\tau _1}_x)  }
    \|    V_2g (t  )   \| _{  L^2_{tx}   }. \end{aligned}
 \end{equation}
 We have $\|   e^{-\im t\mathcal{H}_0}f (0)\| _{L^2_t L^{2,-\tau _1}_x}\lesssim
\|   f (0)\| _{L^2_ { tx } } \lesssim \epsilon $. We have by Lemma
\ref{lem:w-est}
\begin{equation*} \label{eq:g2}  \begin{aligned}
 &\|   e^{-\im t\mathcal{H}_0}Y (0)\| _{L^2_t L^{2,-\tau _1}_x} \le   C \sum
_{|\lambda ^{0}\cdot (\mu -\nu )|> m -\omega _0} \epsilon ^{|\mu + \nu |}.
 \end{aligned}
 \end{equation*}
  We have $\| \text{second line of \eqref{eq:f variable}}\| _{L^1_ t L^2  _x+
L^2_t L^{2, \tau _1}_x}\le O(\epsilon ^2) $. Similarly
$\| \varphi (t) [\Sigma _3,P_d]g  \| _{  L^2_t L^{2,-\tau _1}_x}\le C \epsilon
\| g \| _{   L^2_tL^{2, \tau _1}_x}. $ Hence
 $\| F\| _{L^1_t L^{2 }_x\cap L^2_t L^{2, \tau _1}_x}  \le C \epsilon \| g \|
_{L^2_t L^{2,-\tau _1}_x}+O(\epsilon ^2) $. Now we sketch a bound for
 the second line of \eqref{eq:g1}.
 \begin{equation*} \label{eq:g4}  \begin{aligned}
 &\sum _{|\lambda ^0 \cdot(\mu-\nu)|>m-\omega
_0}\| \int _0^t e^{ \im (t'-t)  \mathcal{H}_0}\partial _{ \| f \| _2^2} H
(t')z^\mu (t')\overline{z}^\nu (t')
P_c(\omega _0)  R ^{+}_{\mathcal{H}^{*}}  (\lambda  ^0 \cdot(\mu-\nu) )
    \Sigma _3H_{\mu \nu}^0 dt'\| _{  L^2_t L^{2,-\tau _1}_x}  \\& \le \sum
_{|\lambda ^0 \cdot(\mu-\nu)|>m-\omega
_0} \| \int _0^t \langle t-t' \rangle ^{-\frac{3}{2}} |\partial _{ \| f \| _2^2}
H (t')z^\mu (t')\overline{z}^\nu (t')| dt'\| _{  L^2_t   }  \lesssim C_2
\epsilon ^2 ,
 \end{aligned}
 \end{equation*}
where we used Lemma \ref{lem:w-est} with $\mathcal{H}$ replaced by
$\mathcal{H}^*$. Of the other contributions to the second line of  \eqref{eq:g1}
we focus on the main ones. Specifically we consider   for $\mu
_j\neq 0$

\begin{equation}\label{countingfactors3}
\| \int _0^te^{\im (t'-t) \mathcal{H}_0}P_c(\omega _0) \frac{{z} ^{ {\mu}  }
\overline{ z} ^ {
\nu }}{z _j}
      \partial _{ \overline{{z}} _j}  Z _0
  R ^{+}_{\mathcal{H} }  (\lambda  ^0 \cdot(\mu-\nu) )
  H^0 _{\mu \nu }  dt' \| _{L^2_tL^{2,-\tau _1}_x}\le C
   \|  \frac{{z} ^{ {\mu}  } \overline{ z} ^ {
\nu }}{z _j}
      \partial _{ \overline{{z}} _j}  Z _0 \| _{L^2 _t}
     \end{equation}
  for
$\lambda (\omega _0) \cdot (  {\mu} -\nu  ) >m -\omega _0  .$
We need to show
\begin{eqnarray} \label{countingfactors1} &
\| \frac{{z} ^{ {\mu}  } \overline{ z} ^ {
\nu }}{z _j}
      \partial _{ \overline{{z}} _j}  Z _0  \| _{L^2 _t}
      =O(\epsilon ^2).
  \end{eqnarray}
Let $z ^\alpha
  \overline{{z}}^\beta $ be a generic monomial of $Z_0$.  Then $\partial
_{\overline{z} _j}(z ^\alpha   \overline{{z}}^\beta ) =\beta
_j\frac{z ^\alpha  \overline{z}^{ {\beta} }}{\overline{z}_j} $,
with the nontrivial case  for $\beta _j\neq 0$. By Definition
\ref{def:normal form} we have $\lambda (\omega _0) \cdot (\alpha
-\beta ) =0 $.  { \ref{Assumption:H11}} can be applied and implies
$|\alpha|=|\beta|\geq 2$. Thus in particular one has
\begin{equation*}
\label{ciao} \lambda (\omega _0)\cdot \alpha\geq \lambda _j(\omega
_0)\Rightarrow
\lambda (\omega _0) \cdot(\mu+\alpha)-\lambda _j(\omega _0) >m -\omega _0\ .
\end{equation*}
So the following holds:
\begin{equation*}\label{countingfactors4}
\| \frac{{z} ^{ {\mu}  }  \overline{z} ^ { \nu }}{z _j}
       \frac{z ^\alpha \overline{z}^{
{\beta} }}{\overline{z}_j}  \| _{L^2 _t} \le \| \frac{  z ^ {
\nu }  {z}^{ {\beta} }}{\overline{z}_j}  \| _{L^\infty  _t}\|
\frac{{z} ^{  \mu   }
        z ^\alpha  }{ {z}_j}  \| _{L^2 _t}\leq
CC_2C_3\epsilon^{|\nu|+|\beta|}\leq CC_2C_3\epsilon^2.
  \end{equation*}
  We conclude that the second line
  in \eqref{eq:g1} is $O(\epsilon ^2)$. The estimates
  omitted are easier than  \eqref{countingfactors3}
  and \eqref{countingfactors1}.  $\| V_2g \| _{L^2_ {tx } }$ can
  be bounded as in Lemma \ref{lem:conditional4.21}.
\end{proof}

\section{The Fermi golden rule}
\label{sec:FGR}
\paragraph{Step (iii)} We proceed as in
  \cite{Cuccagna1}.
We recall Remark \ref{rem:finitesums}. In particular we will only consider finite sums
$
 | \mu +\nu|<2N+3.
$
We will have $\lambda _j^0=\lambda _j (\omega _0)$ and
$\lambda _j=\lambda _j (\| f\| _2^2)$ as in Section
\ref{subsec:Normal form}.
  $|\lambda _j^0- \lambda _j|\lesssim  C_1^2\epsilon ^2$
by \eqref{4.4a}, so  in the sequel we can assume that   $\lambda
^0$ satisfies the same inequalities of $\lambda  .$
Set
$ \index{$R_{\mu \nu }^+$}R_{\mu \nu }^+=R_{ \mathcal{H} }^+ (\lambda ^0\cdot
(\mu
-\nu ) ).$
We substitute
\eqref{eq:f variable} in $\im \dot z_j = \frac{\partial}{\partial
\overline{z}_j} H^{(r)}$ obtaining
\begin{equation}\label{eq:FGR0} \begin{aligned} & \im \dot z _j
=
\partial _{\overline{z}_j}(H_2+Z_0) +   \sum  _{
 |\lambda  \cdot (\mu   -\nu )| >m- \omega
_0    }  \nu _j\frac{z ^\mu
 \overline{ {z }}^ { {\nu} } }{\overline{z}_j}  \langle g ,
\im \beta \alpha _2 \Sigma _1\Sigma _3 H
_{\mu \nu }\rangle       +
\partial _{  \overline{z} _j}  \resto
 \\ &  - \sum  _{ \substack{| \lambda  \cdot
(\alpha    -\beta )|> m- \omega _0
\\
 |\lambda  \cdot (\mu -\nu )|> m- \omega
_0   }}  \nu _j\frac{z ^{\mu +\alpha }  \overline{{z }}^ { {\nu}
+\beta}}{\overline{z}_j} \langle   R_{   \alpha \beta}^+H^0 _{
\alpha \beta },\im \beta \alpha _2 \Sigma _1\Sigma _3H  _{\mu \nu
}\rangle     .
\end{aligned}  \end{equation}
We rewrite this as
\begin{eqnarray} \label{equation:FGR1}& \im \dot z _j= \partial
_{\overline{z}_j}(H_2+Z_0) +
     \mathcal{E}_j
\\ &  \label{equation:FGR12} -\sum  _{ \substack{ \lambda
\cdot  \beta    > m- \omega _0
\\
 \lambda   \cdot   \nu >  m- \omega
_0    \\ \lambda \cdot  \beta -\lambda _k    < m- \omega _0     \,
\forall \, k \, \text{ s.t. } \beta _k\neq 0\\ \lambda  \cdot  \nu
-\lambda _k < m- \omega _0    \, \forall \, k \, \text{ s.t. } \nu
_k\neq 0}} \nu _j\frac{ \overline{{z }}^ {\nu +\beta }
}{\overline{z}_j}\langle R_{ 0 \beta}^+
 { H} _{     0\beta }^0, \im \beta \alpha _2 \Sigma _1\Sigma _3H ^0_{0 \nu
}\rangle
 \\ &  \label{equation:FGR13} -\sum  _{ \substack{ \lambda
\cdot  \alpha    >  m- \omega _0
\\
 \lambda  \cdot   \nu >  m- \omega
_0   \\ \lambda \cdot  \alpha -\lambda _k   <  m- \omega _0     \,
\forall \, k \, \text{ s.t. } \alpha _k\neq 0\\ \lambda \cdot  \nu
-\lambda _k   < m- \omega _0   \, \forall \, k \, \text{ s.t. } \nu
_k\neq 0}} \nu _j\frac{z ^{ \alpha } \overline{{z }}^ {\nu }
}{\overline{z}_j}\langle R_{ \alpha 0 }^+
 H _{   \alpha 0}^0,\im \beta \alpha _2 \Sigma _1\Sigma _3H^0 _{0 \nu }\rangle
.
\end{eqnarray}
Here the elements in \eqref{equation:FGR12} will be eliminated
through a new change of variables. $\mathcal{E}_j$ is a reminder
term defined by

\begin{equation}   \begin{aligned} & \mathcal{E}_j:=
\text{rhs\eqref{eq:FGR0}} -\text{\eqref{equation:FGR12}}-
\text{\eqref{equation:FGR13}}   .
\end{aligned} \nonumber
\end{equation}
Set
\begin{equation}\label{equation:FGR2}  \begin{aligned}   &
\zeta _j =z _j -\sum  _{ (\beta , \nu ) \text{ as in \eqref{equation:FGR12}}
   } \frac{ \nu _j}{\lambda ^0 \cdot (\beta  +\nu) }
\frac{ \overline{{z }}^ {\nu +\beta } }{\overline{z}_j}\langle R_{
0 \beta}^+
 { H} _{  0 \beta  }^0, \im \beta \alpha _2 \Sigma _1\Sigma _3H _{0 \nu
}^0\rangle \\&
+ \sum  _{ (\alpha , \nu ) \text{ as in \eqref{equation:FGR13}}}\frac{ \nu
_j}{\lambda ^0 \cdot (\alpha  - \nu) } \frac{z ^{ \alpha }
\overline{ z}^ { \nu }}{\overline{z}_j} \langle R_{ \alpha 0 }^+
H^0 _{ \alpha 0},\im \beta \alpha _2 \Sigma _1\Sigma _3 H _{0 \nu
}^0\rangle \end{aligned}
\end{equation}
Notice that  in \eqref{equation:FGR2}, by $\lambda \cdot \nu >
 \omega _0-m$, we have $| {\nu} |>1$. Then by \eqref{4.4}

\begin{equation}  \label{equation:FGR3} \begin{aligned}   & \| \zeta  -
 z  \| _{L^2_t} \le C \epsilon \sum _{\substack{ \lambda
  \cdot \alpha    >  m- \omega
_0     \\ \lambda \cdot  \alpha -\lambda _k   < m- \omega _0     \,
\forall \, k \, \text{ s.t. } \alpha _k\neq 0 }} \| z ^{\alpha }\|
_{L^2_t} \le CC_2M\epsilon ^2\\&  \| \zeta  -
 z \| _{L^\infty _t} \le C ^3\epsilon ^3
\end{aligned}
\end{equation}
with $C$ the constant in \eqref{L^inftydiscrete} and $M$ the number
of terms in the rhs. In the new variables \eqref{equation:FGR1} is
of the form

\begin{equation} \label{equation:FGR4} \begin{aligned} &
   \im \dot \zeta
 _j=
\partial _{\overline{\zeta}_j}H_2 (\zeta , f ) +
\partial _{\overline{\zeta}_j}Z_0 (\zeta ,  f )+  { \mathcal{D}_j}
 \\ &   -\sum  _{ \substack{ \lambda ^0
\cdot  \alpha =\lambda ^0\cdot   \nu   >  m- \omega _0
   \\ \lambda
\cdot  \alpha -\lambda _k   <  m- \omega _0     \, \forall \, k \,
\text{ s.t. } \alpha _k\neq 0\\ \lambda \cdot  \nu -\lambda _k  <
m- \omega _0   \, \forall \, k \, \text{ s.t. } \nu _k\neq 0}} \nu
_j \frac{\zeta ^{ \alpha } \overline{ \zeta}^ { \nu
}}{\overline{\zeta}_j} \langle R_{ \alpha 0 }^+ H ^0_{ \alpha
0},\im \beta \alpha _2 \Sigma _1\Sigma _3 H ^0_{0 \nu }\rangle .
\end{aligned}
\end{equation}
From these equations by $\sum _j \lambda _j ^0 ( \overline{\zeta}_j
\partial _{\overline{\zeta}_j}(H_2+Z_0) - {\zeta}_j
\partial _{ {\zeta}_j}(H_2+Z_0)    ) =0$  we get

\begin{equation} \label{eq:FGR5} \begin{aligned}
 &\partial _t \sum _{j=1}^n \lambda  _j ^0
 | \zeta _j|^2  =  2  \sum _{j=1}^n \lambda _j^0\Im \left (
 { \mathcal{D}_j}\overline{\zeta} _j \right ) -\\&    -2
\sum _{  (\alpha , \nu ) \text{ as in \eqref{equation:FGR4}}}
\lambda ^0\cdot \nu \Im \left ( \zeta ^{ \alpha } \overline{\zeta
}^ { \nu } \langle R_{ \alpha 0}^+  H_{ \alpha 0}^0,  \im \beta
\alpha _2 \Sigma _1\Sigma _3  H ^0 _{0\nu }\rangle \right ) .
\end{aligned}
\end{equation}
  We have the following lemma,   whose proof (we skip) is similar to Lemma 4.7
\cite{Cuccagna4}:
\begin{lemma}
\label{lemma:FGR1} Assume  inequalities \eqref{4.4}. Then for a
fixed constant $c_0$ we have
\begin{eqnarray}\label{eq:FGR7} \sum _j\| \mathcal{D}_j
\overline{\zeta} _j\|_{ L^1[0,T]}\le (1+C_2)c_0 \epsilon ^{2}
 . \end{eqnarray}
\end{lemma}
   For the sum in the second line of \eqref{eq:FGR5}
we get

\begin{equation} \label{eq:FGR8} \begin{aligned} & 2\sum _{r>m- \omega
_0  } r
    \Im \left   \langle R_{
\mathcal{H}}^+ (r )\sum _{  \lambda ^0\cdot \alpha =r }\zeta ^{
\alpha } H _{  \alpha 0}^0, \im \beta \alpha _2 \Sigma _1\Sigma _3\sum _{
\lambda
^0\cdot \nu =r}({\zeta} ^{ \nu })^* H ^0_{0\nu
 }  \right \rangle     = \\&   2\sum _{r>m- \omega
_0   } r     \Im \left    \langle R_{ \mathcal{H}}^+ (r )\sum _{
\lambda ^0\cdot \alpha =r }\zeta ^{ \alpha } H _{  \alpha 0}^0,
\Sigma _3\left [\sum _{  \lambda ^0\cdot \alpha =r}\zeta ^{ \alpha
} H ^0_{  \alpha 0} \right ] ^*\right \rangle  = 2\sum _{r>m- \omega
_0   } r     \Im \left    \langle R_{ \mathcal{H}}^+ (r ) \mathbf{ H} _{  r} ,
\Sigma _3 \mathbf{ H} _{  r} ^*\right \rangle  ,
\end{aligned}
\end{equation}
where $\mathbf{H}_r:=\sum _{ \lambda ^0\cdot \alpha =r}\zeta ^{ \alpha } H^0
_{ \alpha 0}$ and where
we have used $\im \beta \alpha _2 \Sigma _1\Sigma _3H_{\mu
\nu }^0= -\Sigma _3 \im \beta \alpha _2 \Sigma _1H_{\nu \mu }^0=
 \Sigma _3 \im \beta \alpha _2 CH_{\nu \mu }^0= \Sigma _3 ({H}^0_{\nu
\mu})^* $ by \eqref{chi.11}.

\begin{lemma}
\label{lemma:FGR81} Consider  $ \mathbf{H}_r $ in \eqref{eq:FGR8}. Assume
$m-\omega _{0}<r < m +\omega _{0}$. Then
 \begin{equation}\label{eq:FGR84} \Im \left    \langle R_{ \mathcal{H}}^+ (r )
\mathbf{ H} _{  r} , \Sigma _3 \mathbf{ H} _{  r} ^*\right \rangle
\ge 0. \end{equation} If we assume {\ref{Assumption:H3}}, in
particular if $m/3 < \omega _{0}<m$, then \eqref{eq:FGR84} holds
for all $ \mathbf{H}_r $ in \eqref{eq:FGR8}.
\end{lemma}
\begin{proof} We proceed as in Lemma 10.5 \cite{Cuccagna1}.
Set  $\mathbf{ F} _{  r}=\mathcal{Z}_+\mathbf{ H}
_{ r}$, where for $\mathcal{Z}_+$   with $\omega =\omega _0$, see  Theorem \ref{cor:wave} in the Appendix.
Set $\mathbf{F}_r= \begin{pmatrix}a \\b \end{pmatrix} $.
Then
\begin{equation*} \label{eq:FGR81} \begin{aligned} &
\Im \left    \langle R_{ \mathcal{H}}^+ (r )  \mathbf{ H}
_{ r},
\Sigma _3\mathbf{ H}
_{ r} ^*\right \rangle =\lim _{\varepsilon \searrow 0} \Im  \left    \langle R_{ \mathcal{H}}  (r +\im \varepsilon ) \mathbf{ H}
_{ r},
\Sigma _3\mathbf{ H}
_{ r} ^*\right \rangle   =\lim _{\varepsilon \searrow 0} \Im \left  \langle  R_{ \mathcal{H}_{ \omega _0,0}}  (r +\im \varepsilon )  \mathbf{F}_r ,
\Sigma _3 {  \mathbf{F}}_r^*\right \rangle \\& =\lim _{\varepsilon \searrow 0}  \Im
\left \langle R_{ D_m}  (r +\omega +\im \varepsilon ) a , a ^*\right \rangle  - \lim _{\varepsilon \searrow 0}  \Im
\left \langle R_{ D_m}  (r -\omega +\im \varepsilon ) b , b ^*\right \rangle \\& =\frac{1}{2} \lim _{\varepsilon \searrow 0}  \varepsilon \|  R_{ D_m}  (r +\omega +\im \varepsilon ) a \| _{L^2}  ^2 - \Im
\left \langle R_{ D_m}  (r -\omega   ) b , b ^*\right \rangle = \frac{1}{2} \lim _{\varepsilon \searrow 0}  \varepsilon \|  R_{ D_m}  (r +\omega +\im \varepsilon ) a \| _{L^2}  ^2 \ge 0.
\end{aligned}
\end{equation*}
Here we exploited that $a,b\in L^2(\R ^3)$, that  $ r -\omega < m$  and so  $R_{ D_m}  (r -\omega  )$ is a well defined selfadjoint operator in $L^2(\R ^3)$,    that
$R_{ D_m} (z ) - R_{ D_m} ( {z}^*  ) =2\im R_{ D_m} (z )   R_{ D_m} ( {z}^* )\Im z$ and that $R_{ D_m} ( {z}^*  )=(R_{ D_m} (z ))^*$.

  Let us consider
   $r=\lambda   \cdot  \mu $ with $\mu \in \mathbb{N}_0^n$, $\lambda   \cdot
\mu >m-\omega  _{0}$ and $\lambda
\cdot  \mu -\lambda   _k   <  m- \omega  _{0} $ for all $k$
 { s.t. } $\mu  _k\neq 0 $.  Suppose $\lambda   \cdot  \mu> m+\omega  _{0}.$
Then  we get
 $m- \omega  _{0}+ \lambda   _k>    m+\omega  _{0} \Rightarrow \lambda   _k>
2\omega  _{0}.$
 Let $N_k\in \mathbb{N}$ such that $N_k\lambda _k
 < m- \omega  _{0} <( N_k+1)\lambda _k$ as in {\ref{Assumption:H9}}.
 Then $ (2N_k +1) \omega _{0} < m$. So, if we assume as in
 {\ref{Assumption:H3}}
that $\omega _{0}>m/3$, we obtain  $\lambda   \cdot  \mu  < m +\omega _{0}$.
This shows that the assumption $\lambda   \cdot  \mu  > m +\omega _{0}$ is
absurd.
  \end{proof}

\begin{remark}
\label{rem:H3bis}
 Notice that to get the
 conclusions of Lemma \ref{lemma:FGR81}
 we can ease  the constraint $3\omega >m $ to $ (2N_k +1) \omega > m$ for all
$k=1,...,n$.
\end{remark}

\bigskip
Now we will assume the following hypothesis.

\begin{enumerate}[label={\bf(H:\arabic{*}')},ref={\bf(H:\arabic{*}')}]
\setcounter{enumi}{11}
\item\label{AssumptionH:12b}  We assume
that for some fixed constant  $C>0$, for any vector $\zeta \in
\mathbb{C}^n$ we have:
\begin{equation} \label{eq:FGR} \begin{aligned} & \sum _{  (\alpha , \nu ) \text{ as in \eqref{equation:FGR4}}}
\lambda ^0\cdot \nu \Im \left ( \zeta ^{ \alpha } \overline{\zeta
}^ { \nu } \langle R_{ \alpha 0}^+  H_{ \alpha 0}^0,  \im \beta
\alpha _2 \Sigma _1\Sigma _3  H ^0 _{0\nu }\rangle \right ) \\&
 \geq C \sum _{ \substack{ \lambda ^0\cdot  \alpha
> m- \omega
_0
\\
   \lambda ^0
\cdot  \alpha -\lambda _k  ^0  <m- \omega _0   \, \forall \, k \,
\text{ s.t. } \alpha _k\neq 0}}  | \zeta ^\alpha  | ^2 .
\end{aligned}
\end{equation}
\end{enumerate}

\begin{remark}
\label{rem:FGR} By Lemma
\ref{lemma:FGR81} we have lhs\eqref{eq:FGR}$\ge 0$. It is likely
then that {\ref{AssumptionH:12b}} is true generically   in the
class of non linearities we consider. But we do not try to prove this
point.
\end{remark}

By {\ref{AssumptionH:12b}} we have

\begin{equation} \label{eq:FGR10} \begin{aligned} &
2\sum  _{j=1}^{n} \lambda _j^0\Im \left ( \mathcal{D}_j   \overline{\zeta}
_j \right )\gtrsim \partial _t \sum  _{j=1}^{n} \lambda _j^0| \zeta _j|^2  +
      \sum _{ \substack{ \lambda ^0\cdot  \alpha
> m- \omega
_0
\\
   \lambda ^0
\cdot  \alpha -\lambda _k  ^0  < m- \omega _0   \, \forall \, k \,
\text{ s.t. } \alpha _k\neq 0}}  | \zeta ^\alpha  | ^2.
\end{aligned}
\end{equation}
Then, for $t\in [0,T]$ and assuming Lemma \ref{lemma:FGR1} we have

\begin{equation}  \sum  _{j=1}^{n} \lambda _j ^0 | \zeta
_j(t)|^2 +\sum _{   \alpha   \text{ as in \eqref{eq:FGR10}}}  \| \zeta ^\alpha \|
_{L^2(0,t)}^2\lesssim \epsilon ^2+ C_2\epsilon ^2.\nonumber
\end{equation}
By \eqref{equation:FGR3} this implies $\| z ^\alpha \|
_{L^2(0,t)}^2\lesssim \epsilon ^2+ C_2\epsilon ^2$ for all the
above multi indexes. So, from  $\| z ^\alpha \|
_{L^2(0,t)}^2\lesssim
  C_2^2\epsilon ^2$ we conclude $\| z ^\alpha \|
_{L^2(0,t)}^2\lesssim  C_2\epsilon ^2$.

Note that as the condition $|\lambda  \cdot (\mu   -\nu )| >m- \omega$
implies that $|\mu   +\nu | \geq 2$, \eqref{eq:FGR0} implies that $\dot{z}$ is
integrable so that it has a limit at infinity which is necessarily $0$.This
yields Theorem \ref{proposition:mainbounds} and completes
the proof of Theorem \ref{th:AsStab}.


\subsection{Proof of Theorem \ref{theorem-1.3}}
\label{subsec:1.3} We only sketch the proof, which is similar to
that of Theorem \ref{theorem-1.2}. For a particular solution
satisfying the hypotheses of Theorem \ref{theorem-1.3} we need to
prove the conclusions of Theorem \ref{proposition:mainbounds}. The
argument is exactly the same of Section \ref{sec:dispersion} until
we reach subsection \ref{sec:FGR}, that is the task of estimating
  $z$. Instead of \eqref{equation:FGR4} we have

  \begin{equation*} \label{equation:FGR41} \begin{aligned} &
   \im \dot \zeta
 _j=\varepsilon _j
\partial _{\overline{\zeta}_j}H_2 (\zeta , f ) +\varepsilon _j
\partial _{\overline{\zeta}_j}Z_0 (\zeta ,  f )+  \varepsilon _j\mathcal{D}_j
 \\ &   -\varepsilon _j\sum _{  (\alpha , \nu ) \text{ as in \eqref{equation:FGR4}}} \nu _j
\frac{\zeta ^{ \alpha } \overline{ \zeta}^ { \nu
}}{\overline{\zeta}_j} \langle R_{ \alpha 0 }^+ H ^0_{ \alpha 0},\im
\beta \alpha _2 \Sigma _1\Sigma _3 H ^0_{0 \nu }\rangle .
\end{aligned}
\end{equation*}
From these equations by $\sum _j \lambda _j ^0 ( \overline{\zeta}_j
\partial _{\overline{\zeta}_j}(H_2+Z_0) - {\zeta}_j
\partial _{ {\zeta}_j}(H_2+Z_0)    ) =0$   we get

\begin{equation} \label{eq:FGR51} \begin{aligned}
 &\partial _t \sum _{j=1}^n \varepsilon _j\lambda  _j ^0
 | \zeta _j|^2  =  2  \sum _{j=1}^n \lambda _j^0\Im \left (
\mathcal{D}_j \overline{\zeta} _j \right ) -\\&    -2 \sum _{  (\alpha , \nu ) \text{ as in \eqref{equation:FGR4}}} \lambda
^0\cdot \nu \Im \left ( \zeta ^{ \alpha } \overline{\zeta }^ { \nu }
\langle R_{ \alpha 0}^+  H_{ \alpha 0}^0,  \im \beta \alpha _2
\Sigma _1\Sigma _3  H ^0 _{0\nu }\rangle \right ) .
\end{aligned}
\end{equation}
The estimate of the reminder term in Lemma \ref{lemma:FGR1}
continues to hold. The   last line of \eqref{eq:FGR51} is negative
by  \eqref{eq:FGR8}.  We assume it is strictly negative and
that in particular \eqref{eq:FGR} holds.  Then we get

\begin{equation} \label{eq:FGR101} \begin{aligned} &
\sum _{  \alpha   \text{ as in \eqref{eq:FGR10}}}  | \zeta ^\alpha  | ^2\lesssim
-\partial _t \sum  _{j=1}^{n}\varepsilon _j \lambda _j^0| \zeta
_j|^2 + 2\sum  _{j=1}^{n} \lambda _j^0\Im \left ( \mathcal{D}_j
\overline{\zeta} _j \right )
      .
\end{aligned}
\end{equation}
When we integrate in $(0,t)$ for $t\le T$  we get

\begin{equation}  \sum _{ \alpha \text{ as in \eqref{eq:FGR10}}}  \| \zeta ^\alpha \|
_{L^2(0,t)}^2\lesssim \epsilon ^2+ C_2\epsilon ^2.\nonumber
\end{equation}
In the rhs we have used  
the hypothesis $|z(t)|\le \epsilon $ for all $t\ge
0$ to bound the first summation in the rhs of \eqref{eq:FGR101}.
This yields Theorem \ref{theorem-1.3}.

\subsection{Proof of Theorem \ref{theorem-1.4}}
\label{subsec:1.4} Also here we just sketch the proof, which is
similar to \cite{Cuccagna2}. The proof is by contradiction. If the
statement of Theorem \ref{theorem-1.4} is wrong, then  for $ |z(0)|
+ \| f(0)\| _{H^{k_0} }\le \delta   $   with $\delta >0$
sufficiently small, we can assume   $ |z(t)| \le \epsilon     $ for
all $t\ge 0$ for any preassigned $ \epsilon  >0$. This implies that
we can apply Theorem \ref{theorem-1.3}. When get

\begin{equation}  \label{eq:4310} \sum _{ \alpha \text{ as in \eqref{eq:FGR10}}}  \| \zeta ^\alpha \|
_{L^2(0,t)}^2\lesssim \sum  _{j=1}^{n}\varepsilon _j \lambda _j^0 (
| \zeta _j(0) |^2-| \zeta _j(t) |^2) +2\int _0^t \sum  _{j=1}^{n}
\lambda _j^0\Im \left ( \mathcal{D}_j \overline{\zeta} _j \right
).
\end{equation}
Suppose $\varepsilon _{j_0}=-1$. Then take initial datum
$z_j(0) =0$ for $j\neq j_0  $, $z_ { j_0}=\delta $ and $f(0) =0$.
 By $f(0) =0$ and Lemma \ref{lem:conditional4.21} for $\psi (0)=0$ we get
 for $t\in \R ^+$

 \begin{equation*} \label{eq:4312}  \begin{aligned}
 &     \|  f \| _{L^p_t B^{  k_0-\frac{3}{p}} _{q,2} \cap L^2_t H^{  k_0 ,-\tau
_0}_x \cap L^2_t L^{ \infty}_x }\lesssim \mathcal{Y}^2   +\| R_1 \| _{L^1_t H^{
k_0}_x   }+\| R_2 \| _{   L^2_t H^{ k_0 , \tau _0}_x }  \end{aligned}
\end{equation*}
    \begin{equation*}\label{eq:4313}\mathcal{Y}^2 :=\sum _{ \alpha \text{ as in \eqref{eq:FGR10}}}\| z ^\alpha \| _{L^2_t }.\end{equation*}
Similarly  \begin{equation*} \label{eq:43120}  \begin{aligned}
 &     \|  g \| _{L^2_t L^{2,- \tau _1}_x }\lesssim \delta ^2+\epsilon
\mathcal{Y}^2 +\| R_1 \| _{L^1_t H^{ k_0}_x   }+\| R_2 \| _{   L^2_t H^{ k_0 ,
\tau _0}_x }. \end{aligned}
\end{equation*}
   Then,
proceeding as in \cite{Cuccagna2,Cuccagna4} one improves the rhs in
\eqref{eq:FGR7}. Indeed, see Lemma 4.9 \cite{Cuccagna4}, we have
\begin{equation}\sum _j\|\mathcal{D}_j \overline{\zeta} _j\|_{
L^1(\R ^+)}\le C  \mathcal{Y} \| g \| _{L^2_t H^{-4,-s}_x } + C \epsilon
 \mathcal{Y}^{2}  + C  \| R_1 \| _{L^1_t H^{ k_0}_x   }+C\| R_2 \| _{   L^2_t
H^{ k_0 , \tau _0}_x } .\nonumber
\end{equation}
Then, one can see that $ \| R_1 \| _{L^1_t H^{ k_0}_x   }+\| R_2 \| _{   L^2_t
H^{ k_0 , \tau _1}_x }\lesssim o(1) \delta $, going through Lemma \ref{lem:bound
remainder}, where $o(1)\to 0$ as $\delta \to 0$. Then from \eqref{eq:4310} we
get $   \mathcal{Y}^2\lesssim -   \delta  +o(1)  \delta
,$
which is absurd.

\appendix

\section{Resolvent estimates and wave operators}

\begin{lemma}  \label{lem:smooth1} We assume {\ref{Assumption:H1}} and
{\ref{Assumption:H6}}--{\ref{Assumption:H8}}. Then for any $\tau >1$  there exists a
constant $ C_1=C _1(\tau , \omega)$  upper semicontinuous in  $\omega$
s.t.  for any $u _0(x) \in L^2(\R ^3,\mathbb{C}^8)$ and any
$\varepsilon
>0$ we have

   \begin{equation} \label{eq:smooth1}
   \| \langle x \rangle ^{-\tau}  R_{  \mathcal{H}_{\omega}}(\lambda \pm \rmi
\varepsilon )
    P_c( \mathcal{H}_{\omega}) u _0\| _{L_{\lambda ,x } ^2(\R ^4)}\le C _{1} \|
P_c( \mathcal{H}_{\omega}) u _0\| _{  L^2(\R ^3) }.
   \end{equation}

   \end{lemma}
\begin{proof}   Notice that by Lemma \ref{lem:absflat}  for any     $\tau >1$,
any $u _0(x) \in   L^2(\R  ^3,\mathbb{C}^8)$  and any $\varepsilon
>0$  we have

   \begin{equation} \label{eq:smooth1unp}   \|
   \langle x \rangle ^{-\tau}  R_{  \mathcal{H}_{\omega ,0}}(\lambda \pm \rmi
\varepsilon )
      u _0\| _{L_{\lambda ,x } ^2(\R ^4)}\le C
      (\tau ) \|   u _0\| _{  L^2(\R ^3) }.
   \end{equation}
  Let
  $u_0= P_c( \mathcal{H}_{\omega}) u _0$,
    $A (x) =\langle x \rangle ^{-\tau}$
  and    $
B(x)\in \mathcal{S}( \mathbb{R}^3,B(\mathbb{C}^8, \mathbb{C}^8))  $
  s.t. $B^*A=V_\omega $. Then
 \begin{equation} \label{eq:smooth11}
 A R_{\mathcal{H}_{\omega }} (z  ) u _0=
 ( 1+ AR_{\mathcal{H}_{\omega ,0}} (z  )B^*   )^{-1}
  AR_{\mathcal{H}_{\omega ,0}} (z  )u _0  .
 \end{equation}
 The following operators preserve $  \mathbf{X}$:  $A  $,
 $B^*$, $  R_{\mathcal{H}_{\omega }} (z  )  $ and
 $  R_{\mathcal{H}_{\omega 0 }} (z  )  $.
 Pick $\delta _0 >0 $   sufficiently small so that by
 {\ref{Assumption:H6}} for
any
  $\lambda _j(\omega )\in \sigma _d(\mathcal{H}_{\omega })$
   we have $|\lambda _j(\omega )|<m-\omega -\delta _0$.  Then by
\eqref{eq:smooth1unp}
 and  \eqref{eq:smooth11},  Lemma \ref{lem:smooth1}  is a consequence
 of the Lemma \ref{lem:smooth2} below. \end{proof}

\begin{lemma}  \label{lem:smooth2} Let $A (x), B(x)
$ be as above in \eqref{eq:smooth11}. Then, if we assume
{\ref{Assumption:H3}}, {\ref{Assumption:H6}} and
{\ref{Assumption:H7}}, there exists a constant $ C_2=C _2(\tau
,\omega)$ upper semicontinuous in  $\omega$ such that for any
$\varepsilon
>0$ we have

   \begin{equation} \label{eq:smooth2}
   \sup _{\lambda \in (\R \backslash
    [-m+\omega +\delta _0,m-\omega -\delta _0])}\|
    ( 1+  AR_{\mathcal{H}_{\omega ,0}}
    (\lambda \pm \rmi \varepsilon  )B^*   )^{-1}
     \| _{B(\mathbf{X}, \mathbf{X})}  \le C _{2}  .
   \end{equation}
For any   $\tau >1$   the  limit
   $\displaystyle
 R_{\mathcal{H}_{\omega }}^+(\lambda)=\lim _{\varepsilon \searrow 0}
R_{\mathcal{H}_{\omega }}
    (\lambda \pm \rmi \varepsilon  )
  $
   exist  in $B(H^{1, \tau   }_x\cap \mathbf{X}, L^{2, -\tau   }_x)$ and the
   convergence is uniform for $\lambda $ in compact sets.

\end{lemma}

\begin{proof}  First of all we prove
  \eqref{eq:smooth2} in low energies.
We want to prove
 \begin{equation} \label{eq:smooth2low}
   \sup_{ \substack{ \lambda \in ([-\mu _1, \mu _1 ] \backslash
    [-m+\omega +\delta _0,m-\omega -\delta _0]\\  0<\varepsilon <1}}
     \| ( 1+  AR_{\mathcal{H}_{\omega ,0}}
     (\lambda \pm \rmi \varepsilon  )B^*   )^{-1}
       \| _{B(\mathbf{X},\mathbf{X})}   <\infty
        \text{ $\forall$ fixed $ \mu _1>0$.}
   \end{equation}
   We know:   $z\to AR_{\mathcal{H}_{\omega ,0}} (z  )B^*$ is
    a holomorphic map with domain $\C \backslash \R$ and
    values in $B(\mathbf{X}, \mathbf{X})$; for all $z\in \C \backslash \R$,
      $( 1+  AR_{\mathcal{H}_{\omega ,0}}
      (\lambda \pm \rmi \varepsilon  )B^*   )^{-1}$
      is defined . Furthermore,
   $   \lim _{\varepsilon \searrow 0}
   AR_{\mathcal{H}_{\omega ,0}} (\lambda \pm \rmi \varepsilon  )B^*$,
    by (ii) Lemma \ref{lem:absflat},
   exists in $B(\mathbf{X}, \mathbf{X})$ and the
   convergence is uniform for $\lambda $ in compact sets.
 Then we apply
   Lemma 7.5 \cite{Berthier} and conclude that, outside
     closed sets
   $\Gamma ^{\pm}\subset \R$   with 0 Lebesgue
   measure in  $\R$, the map
   $z\to ( 1+  AR_{\mathcal{H}_{\omega ,0}} (z  )B^*   )^{-1}$
   extends  in a continuous map  defined
   in $\{ z: \Im z>0  \} \cup (\R \backslash \Gamma ^+)$
   (resp. $\{ z: \Im z<0  \} \cup (\R \backslash \Gamma ^-)$)
    with values in  $B(\mathbf{X}, \mathbf{X})$.
    Given $\lambda \in \Gamma ^+$
    there exists $\psi \in \mathbf{X}\backslash \{ 0 \}$
    with $\psi =-  AR_{\mathcal{H}_{\omega ,0}}
     ^{+ }(\lambda   )B^*\psi .  $ But then, by standard arguments
     $u:=R_{\mathcal{H}_{\omega ,0}}
     ^{+ }(\lambda   )B^*\psi  \in L^{2, -\tau}(\R ^3, \C ^8)$ is a nonzero solution
     of \eqref{Eq:resonance}. By (H7)--(H8) we have
       $u\in  L^{2}(\R ^3, \C ^8)$.  Furthermore $\psi\in \mathbf{X}$    implies $u\in \mathbf{X}$.
     But by   (H6)  no such $u\in \mathbf{X}$ can exist.  So
     the intersection  of  $\Gamma ^+  $
      with $\R \backslash (-m +\omega +\delta _{0},
      m -\omega  -\delta _{0})$ is empty.  A similar argument shows that
       the intersection of  $\Gamma ^-  $
      with $\R \backslash (-m +\omega +\delta _{0},
      m -\omega  -\delta _{0})$ is empty.

Having considered
   the low energy case \eqref{eq:smooth2low}, we consider
for $\mu _1$ any fixed large real number:

 \begin{equation} \label{eq:smooth21}   \sup _{|\lambda |\ge \mu _1}\| ( 1+
AR_{\mathcal{H}_{\omega ,0}}^{\pm } (\lambda    )B^*   )^{-1}   \| _{B(L^{ 2
}_x, L^{ 2 }_x)}  \le C _{3}  .
   \end{equation}
For definiteness we will consider  $\lambda  \ge \mu _1$.
We consider the expansion
 $    \sum _{\ell =0} ^{\infty}  \left (
AR_{\mathcal{H}_{\omega , 0}}  ^{\pm }(\lambda    )B^*   \right )^{\ell}.
   $  We start now the implementation of the high energy argument
in
   \cite{ErdoganGoldbergSchlag}. We have
\begin{equation} \label{eq:smooth23}   R_{\mathcal{H}_{\omega , 0}}  ^{\pm
}(\lambda    ) = \begin{pmatrix}  R_{D_m }  ^{\pm }(\lambda +\omega    )
  &
 0\\
  0&   R_{D_m }  ^{\pm }(\lambda -\omega    )
\end{pmatrix}  =R_0 ^{\pm }(\lambda )   \mathcal{A}(\lambda , \nabla )
   \end{equation}
  \begin{equation} \label{eq:smooth24} \begin{aligned}&  R_0 ^{\pm }(\lambda
) :=\begin{pmatrix}  R_{-\Delta +m^2 }  ^{\pm }((\lambda +\omega  )^2  ) I_2
  &
 0\\
  0&   R_{-\Delta +m^2 }  ^{\pm }((\lambda -\omega )^2   )I_2
\end{pmatrix}\\&    \mathcal{A}(\lambda , \nabla ):=    \begin{pmatrix}
\mathcal{A}_1(\lambda , \nabla )
  & 0   \\
    0    &    \mathcal{A}_2(\lambda , \nabla )
\end{pmatrix}\, , \quad  \mathcal{A}_j(\lambda , \nabla ) :=\begin{pmatrix}
{\lambda - (-1)^j\omega +m}
  &
   -\rmi{\sigma \cdot \nabla }   \\
    -\rmi {\sigma \cdot \nabla }    &    {\lambda - (-1)^j\omega -m}
\end{pmatrix}      . \end{aligned}
   \end{equation}
 For definiteness let us consider
  $ R_{\mathcal{H}_{\omega ,0}}  ^{+ }$. Let now
   $\chi _{0}, \psi_0\in C^\infty _0(\R )$ by  cutoffs
 supported near 0 and let
 $\chi _1:=1-\chi _{0}  $ and  $\psi _1:=1-\psi _{0}. $
We can choose them so that
\begin{equation}
 \label{eq:smooth251} \begin{aligned} & \chi _1 \left ( {|x-y|}
 \right )   =\left (\psi _0 \left ( {|x |}
 \right ) \psi _1 \left ( {| y|}
 \right ) +  \psi _1 \left ( {| x|}
 \right ) \psi _0 \left ( {|y |}
 \right ) +  \psi _1 \left ( {| x|}
 \right ) \psi _1 \left ( {|y |}
 \right )
 \right ) \chi _1 \left ( {|x-y|}
 \right )
\end{aligned}  \end{equation}
We split for a fixed large number $M_0>0$ \begin{equation}
 \label{eq:smooth252} \begin{aligned} & R_{-\Delta +m^2 }  ^{+ }
 ((\lambda -(-1)^j\omega )^2,x,y   ) = \sum _{\ell =0}^{1} R_{\ell j}
 (\lambda , x, y)\ ,\\& R_{\ell j}
 (\lambda , x, y) :=
 \frac{e^{{\rmi}
 \sqrt{(\lambda -(-1)^j\omega )^{2}+m^2}|x-y|}}
 {4\pi |x-y|}    \chi _{\ell}\left (\frac{|x-y|}{ M_0 }
 \right )   . \end{aligned}  \end{equation}
 We have a  decomposition $R_{\mathcal{H}_{\omega ,0}}  ^{+ }= R_{\mathcal{H}_{\omega ,0}}  ^{0+ } +R_{\mathcal{H}_{\omega ,0}}  ^{1+ } $
 with kernels $R_{\mathcal{H}_{\omega ,0}}  ^{j+ } =\chi _j \left ( \frac{|\cdot |}{M_0} \right ) R_{\mathcal{H}_{\omega ,0}}  ^{ + }$.
 By
\eqref{eq:smooth251}--\eqref{eq:smooth252} and by \cite{Agmon} there exists $c _{M_0} $ with
$\lim _{M_0\to +\infty}c _{M_0} =0 $ s.t.
\begin{equation} \label{eq:smooth253}\sup _{\lambda \in \R} \| AR_{\mathcal{H}_{\omega ,0}}  ^{1+ } (\lambda )B^*
\| _{B(L^{ 2 }_x, L^{ 2 }_x)}  \le c _{M_0}
    .
\end{equation}
By $\| AR_{\mathcal{H}_{\omega ,0}}  ^{\pm } (\lambda )B^* \| _{B(L^{ 2 }_x, L^{
2 }_x)} \le C  $,   for fixed $C'$ we have
\begin{equation} \label{eq:smooth254} \| AR_{\mathcal{H}_{\omega ,0}}  ^{0\pm } (\lambda )B^*
\| _{B(L^{ 2 }_x, L^{ 2 }_x)}  \le C'
    .
\end{equation}
We have \begin{equation} \label{eq:smooth26}    R_{0 j}(\lambda
,x,y ) =   \lambda R_{-\Delta } ^{+ } \left (\sqrt{\left ( 1
-(-1)^j\frac{\omega }{\lambda} \right )^2 +\frac{m^2}{\lambda
^2}},\lambda x, \lambda y  \right  ) \chi _{0}\left
(\frac{|x-y|}{M_0 }
 \right ) .
\end{equation}
 Key to showing that \eqref{eq:smooth21} follows   directly from
\cite{ErdoganGoldbergSchlag} is the observation that we can write \begin{equation}
\label{eq:smooth27}
 R_{-\Delta } ^{+ } \left (\sqrt{\left ( 1
-(-1)^j\frac{\omega }{\lambda} \right )^2 +\frac{m^2}{\lambda ^2}},
x,  y  \right  ) \chi _{0}\left (\frac{|x-y|}{ \lambda M_0}
 \right )  =
  \frac{e^{{\rmi} |x-y|}}
  {|x-y|}a_{\lambda ,j}(|x-y|) +
  \frac{b_{\lambda ,j}(|x-y|)}{|x-y|} ,  \end{equation}
with
\begin{equation} \label{eq:smooth28}\begin{aligned} &
   \left |  a_{\lambda ,j}^{(k)}(r) \right | \le C(M_0,k) r^{-k} \quad
   \forall \, k\ge 0, \quad a_{\lambda ,j}^{(k)}(r) =0  \quad
   \forall \, 0<r<1\\&  \left |  b_{\lambda ,j}^{(k)}(r)
    \right | \le C(M_0,k)   \quad
   \forall \, k\ge 0, \quad b_{\lambda ,j}^{(k)}(r) =0  \quad
   \forall \,  r>2.
 \end{aligned}
\end{equation}
Notice that \eqref{eq:smooth27}--\eqref{eq:smooth28} are formulas
of the same type of (3.2)--(3.4) \cite{ErdoganGoldbergSchlag}. As a consequence
for
any fixed  small $\delta _0>0$ there are $\ell_0 =\ell (\delta _0)$
and $ \mu _1=\mu _1(\delta _0)$  such that for $\lambda \ge \mu _1
$ we have
\begin{equation} \label{eq:smooth29} \left \|
\left (    A\chi _{0} R_{\mathcal{H}_{\omega ,0}}
  ^{0+}(\lambda    )B^* \right ) ^{\ell _0} \right
   \| _{B(L^{ 2 }_x, L^{ 2 }_x)}  \le \delta _0.
\end{equation}
For $\ell $   large and $\delta _0\le c _{M_0}$, by \eqref{eq:smooth253}, \eqref{eq:smooth254}   and \eqref{eq:smooth29}  we get \begin{equation}
\label{eq:smooth30}\begin{aligned} & \left \| \left (    A\ R_{\mathcal{H}_{\omega ,0}}
  ^{0+}(\lambda    )
  B^*+  A R_{\mathcal{H}_{\omega ,0}}
  ^{1+}(\lambda    )B^* \right ) ^\ell \right
   \| _{B(L^{ 2 }_x, L^{ 2  }_x)}  \le 2^\ell  (2C')^\ell
   c _{M_0}^{\frac{\ell}{\ell _0}}. \end{aligned}
\end{equation}
For $c _{M_0} $ sufficiently small,   \eqref{eq:smooth30} implies
\eqref{eq:smooth21}.
\end{proof}

 We finish with the following  corollary   of Lemma   \ref{lem:smooth1}.
\begin{theorem}  \label{cor:wave} Assume the hypotheses of
Lemma \ref{lem:smooth1}. Pick the  $A, B^\ast  $ of \eqref{eq:smooth11}.
Then there  are isomorphisms
 $\mathcal{W}_{\pm}\colon \mathbf{X}
 \to\mathbf{X}_c(\mathcal{H}_{\omega }) $ and $\mathcal{Z}_{\pm}\colon
\mathbf{X}_c(\mathcal{H}_{\omega })  \to   \mathbf{X}  $,
inverses of each other, defined as follows: for $u\in   \mathbf{X}$,  $v\in  \mathbf{X}_c(\mathcal{H}_{\omega })  $,
 \begin{equation} \label{eq:Weinstein}  \begin{aligned} & \langle
 \mathcal{W} _{\pm}u,v^*\rangle =
\langle  u,v^*\rangle  \mp \lim _{\epsilon \to 0^+ } \frac 1{2\pi
\rmi } \int _{\R} \langle A R _{\mathcal{H}_{\omega ,0}}  (\lambda
\pm \rmi \epsilon  )  u,(B R _{\mathcal{H}_{\omega  }^*}  (\lambda
\pm \rmi \epsilon  )  v)^*\rangle d\lambda ;\\ & \langle
\mathcal{Z}_{\pm} v,u^*\rangle = \langle  v,u^*\rangle  \pm \lim
_{\epsilon \to 0^+ } \frac 1{2\pi \rmi } \int _{\R} \langle A R
_{\mathcal{H}_{\omega }} (\lambda \pm \rmi \epsilon  )  v,(B R
_{\mathcal{H}_{\omega ,0 }^*} (\lambda \pm \rmi \epsilon  )
u)^*\rangle d\lambda .
\end{aligned}
\end{equation}
    $\mathcal{W}_{\pm}$ (resp.$\mathcal{Z}_{\pm}$)
define isomorphisms $H^k(\R ^3, \C ^8)\cap \mathbf{X}\to P_c  (\mathcal{H}_{\omega
})  H^k(\R ^3, \C ^8)  $ (resp. and viceversa) for all $k$. We also
have
  \begin{equation} \label{eq:W2}  \begin{aligned} &
   \mathcal{W}_{\pm}u=\lim _{t\to \pm \infty}
  e^{\rmi t \mathcal{H}_{\omega  }} e^{-\rmi t \mathcal{H}_{\omega ,0 }}u
  \text{ for all $u\in   \mathbf{X}$;}\\& \mathcal{Z}_{\pm}v
  =\lim _{t\to \pm \infty}
  e^{\rmi t \mathcal{H}_{\omega , 0 }} e^{-\rmi t \mathcal{H}_{\omega   }}v
  \text{ for all $v\in  \mathbf{X}_c(\mathcal{H}_{\omega })$.}
  \end{aligned}
\end{equation}

\end{theorem}
\begin{proof} The proof  follows by Lemma \ref{lem:smooth1} by means of the
argument for Theorem 1.5 \cite{Kato}. \eqref{eq:W2} follows by Theorem
3.9
\cite{Kato}.
\end{proof}


\begin{thebibliography}{BCDM88}
\bibitem {Agmon}
S.Agmon.
\newblock Spectral properties of Schr\"odinger operators and scattering theory.
\newblock {\em An. Sc. N. Pisa}, 2(2):151--218, 1975.





\bibitem {BambusiCuccagna}
D.Bambusi and S.Cuccagna.
\newblock On dispersion of small energy solutions of the nonlinear klein gordon
  equation with a potential,
		\newblock {\em Amer. Math. Jour.}, 133(1):1421--1468, 2011.
		



\bibitem {BalabaneCazenaveDouadyMerle}
M.Balabane, T.Cazenave, A.Douady, and F.Merle.
\newblock Existence of excited states for a nonlinear {D}irac field.
\newblock {\em Com. Math. Phys.}, 119(1):153--176, 1988.

\bibitem {BalabaneCazenaveVazquez}
M.Balabane, T.Cazenave, and L.V{\'a}zquez.
\newblock Existence of standing waves for {D}irac fields with singular
  nonlinearities.
\newblock {\em Com. Math. Phys.}, 133(1):53--74, 1990.


\bibitem {Beceanu}
M.Beceanu.
\newblock A centre-stable manifold for the focussing cubic {NLS} in {$\mathbb
  R^{1+3}$}.
\newblock {\em Com. Math. Phys.}, 280(1):145--205, 2008.

\bibitem {BerkolaikoComech}
G.~Berkolaiko and A.~Comech.
\newblock On spectral stability of solitary waves of nonlinear dirac equation
  on a line, arXiv:0910.0917.



\bibitem {Berthier}
A.-M.Berthier.
\newblock {\em Spectral theory and wave operators for the {S}chr{\"o}dinger
  equation}, vol.~71 of {\em Res. N. in Math}.
\newblock Pitman , Boston, Mass., 1982.

\bibitem {BerthierGeorgescu}
A.-M.Berthier and V.Georgescu.
\newblock On the point spectrum of {D}irac operators.
\newblock {\em J. Fun. Anal.}, 71(2):309--338, 1987.




\bibitem {Boussaid}
N.Boussaid.
\newblock Stable directions for small nonlinear {D}irac standing waves.
\newblock {\em Com. Math. Phys.}, 268(3):757--817, 2006.

\bibitem {Boussaid2}
N.Boussaid.
\newblock On the asymptotic stability of small nonlinear {D}irac standing waves
  in a resonant case.
\newblock {\em SIAM J. Math. Anal.}, 40(4):1621--1670, 2008.


\bibitem {Bournaveas}
N.Bournaveas.
\newblock Local well-posedness for a nonlinear {D}irac equation in spaces of
  almost critical dimension.
\newblock {\em Discr. Cont. Dyn. Syst.}, 20(3):605--616, 2008.



\bibitem {BuslaevPerelman}
V.Buslaev and G.Perel{\cprime}man.
\newblock Scattering for the nonlinear {S}chr{\"o}dinger equation: states that
  are close to a soliton.
\newblock {\em Algebra i Analiz}, 4(6):63--102, 1992.

\bibitem {BuslaevPerelman2}
V.Buslaev and G.Perel{\cprime}man.
\newblock On the stability of solitary waves for nonlinear {S}chr{\"o}dinger
  equations.
\newblock In {\em Nonlinear evolution equations}, vol. 164  {\em Am.
  Math. Soc. Tran. Ser. 2}, pages 75--98, AMS , Providence, RI,
  1995.

\bibitem {Brenner}
P.Brenner.
\newblock {On space-time means and everywhere defined scattering operators for
  nonlinear Klein-Gordon equations}.
\newblock {\em Math. Zeitschrift}, 186(3):383--391, 1984.

\bibitem {CazenaveLions}
T.Cazenave and P.-L.Lions.
\newblock Orbital stability of standing waves for some nonlinear
  {S}chr{\"o}dinger equations.
\newblock {\em Com. Math. Phys.}, 85(4):549--561, 1982.


\bibitem {CazenaveVazquez}
T.Cazenave and L.V{\'a}zquez.
\newblock Existence of localized solutions for a classical nonlinear {D}irac
  field.
\newblock {\em Com. Math. Phys.}, 105(1):35--47, 1986.


\bibitem {ChangGustafsonNakanishiTsai}
Shu-Ming Chang, Stephen Gustafson, Kenji Nakanishi, and Tai-Peng Tsai.
\newblock Spectra of linearized operators for {NLS} solitary waves.
\newblock {\em SIAM J. Math. Anal.}, 39(4):1070--1111, 2007/08.




\bibitem {ChristKiselev}
M.Christ and A.Kiselev.
\newblock Maximal functions associated to filtrations.
\newblock {\em J. Fun. Anal.}, 179(2):409--425, 2001.




\bibitem {Chugunova}
M.Chugunova and D.Pelinovsky.
\newblock Block-diagonalization of the symmetric first-order coupled-mode system.
\newblock {\em SIAM J. Appl. Dyn. Syst.}, 5(1):66--83, 2006.


\bibitem {Comech}
  A.Comech.
\newblock On the meaning of the Vakhitov-Kolokolov stability criterion for the nonlinear Dirac equation, arXiv:1107.1763,  2011.

\bibitem {ChugunovaPelinovsky}
M.Chugunova and D.Pelinovsky.
\newblock Block-diagonalization of the symmetric first-order coupled-mode
  system.
\newblock {\em SIAM J. Appl. Dyn. Syst.}, 5(1):66--83 (electronic), 2006.

\bibitem {Cuccagna5}
S.Cuccagna.
\newblock Stabilization of solutions to nonlinear {S}chr{\"o}dinger equations.
\newblock {\em Com. Pure Appl. Math.}, 54(9):1110--1145, 2001.

\bibitem {Cuccagna3}
S.Cuccagna.
\newblock On asymptotic stability in energy space of ground states of {NLS} in
  1{D}.
\newblock {\em J. Diff. Eq.}, 245(3):653--691, 2008.

\bibitem {Cuccagna2}
S.Cuccagna.
\newblock On instability of excited states of the nonlinear {S}chr{\"o}dinger
  equation.
\newblock {\em Phys. D}, 238(1):38--54, 2009.



\bibitem {Cuccagna4}
S.Cuccagna.
\newblock On scattering of small energy solutions of non autonomous hamiltonian
  nonlinear schr{\"o}dinger equations,  \newblock {\em J. Diff. Eq.}, 250(5):2347-2371, 2011.

\bibitem {Cuccagna1}
S.Cuccagna.
\newblock The hamiltonian structure of the nonlinear Schr{\"o}dinger equation
  and the asymptotic stability of its ground states,  \newblock {\em Com. Math. Phys.},
  305:279-331, 2011.

\bibitem {Cuccagna6}
S.Cuccagna.
\newblock On asymptotic stability of moving ground states of the nonlinear Schrodinger equation,  	arXiv:1107.4954v3, 2011, to appear Tran.AMS.



\bibitem {cuccagnamizumachi}
S.Cuccagna and T.Mizumachi.
\newblock On asymptotic stability in energy space of ground states for
  nonlinear {S}chr{\"o}dinger equations.
\newblock {\em Com. Math. Phys.}, 284(1):51--77, 2008.

\bibitem {CuccagnaPelinovskyVougalter}
S.Cuccagna, D.Pelinovsky, and V. Vougalter.
\newblock Spectra of positive and negative energies in the linearized {NLS}
  problem.
\newblock {\em Com. Pure Appl. Math.}, 58(1):1--29, 2005.







\bibitem {ErdoganGoldbergSchlag}
M.B.Erdo\u{g}an, M.Goldberg, and W.Schlag.
\newblock Strichartz and smoothing estimates for {S}chr{\"o}dinger operators
  with almost critical magnetic potentials in three and higher dimensions.
\newblock {\em Forum Math.}, 21(4):687--722, 2009.

\bibitem {EstebanLewinSere}
M.J.Esteban, M.Lewin, and E.S{\'e}r{\'e}.
\newblock Variational methods in relativistic quantum mechanics.
\newblock {\em Bull. Amer. Math. Soc. (N.S.)}, 45(4):535--593, 2008.

\bibitem {EstebanSere}
M.J. Esteban and {\'E}. S{\'e}r{\'e}.
\newblock Stationary states of the nonlinear {D}irac equation: a variational
  approach.
\newblock {\em Com. Math. Phys.}, 171(2):323--350, 1995.


\bibitem {zhousigal}
Zhou Gang, I.M.Sigal, \newblock Relaxation of Solitons in Nonlinear
  Schr\"odinger Equations with Potential.
\newblock {\em dvances in
  Math.}, 216: 443-490, 2007.



\bibitem {GeorgescuMantoiu}
V.Georgescu and M.M\u{a}ntoiu.
\newblock On the spectral theory of singular {D}irac type {H}amiltonians.
\newblock {\em J. Oper. Th.}, 46(2):289--321, 2001.






\bibitem {GrillakisShatahStrauss}
M.Grillakis, J.Shatah  and W.Strauss.
\newblock Stability theory of solitary waves in the presence of symmetry. {I}.
\newblock {\em J. Fun. Anal.}, 74(1):160--197, 1987.

\bibitem {GrillakisShatahStrauss2}
M. Grillakis, J.Shatah  and W. Strauss.
\newblock Stability theory of solitary waves in the presence of symmetry. {II}.
\newblock {\em J. Fun. Anal.}, 94(2):308--348, 1990.


\bibitem {Guan}
Meijiao Guan.
\newblock Solitary wave solutions for the nonlinear {D}irac equations, arXiv:0812.2273,
  2008.

\bibitem {Hislop}
P.D.Hislop.
\newblock Exponential decay of two-body eigenfunctions: a review.
\newblock In {\em Proceedings of the {S}ymposium on {M}athematical {P}hysics
  and {Q}uantum {F}ield {T}heory ({B}erkeley, {CA}, 1999)}, Conf.~4   {\em
  Electron. J. Differ. Equ. Conf.}, pages 265--288 (electronic), San Marcos,
  TX, 2000. Southwest Texas State Univ.

\bibitem {IftimoviciMantoiu}
A.Iftimovici and M.M\u{a}ntoiu.
\newblock Limiting absorption principle at critical values for the {D}irac
  operator.
\newblock {\em Lett. Math. Phys.}, 49(3):235--243, 1999.

\bibitem {Kato}
T.Kato.
\newblock Wave operators and similarity for some non-selfadjoint operators.
\newblock {\em Math. Ann.}, 162:258--279, 1965/1966.

\bibitem {KomechKomech}
A.Komech and A.Komech.
\newblock On global attraction to quantum stationary states. {D}iracequation
  with mean field interaction, February 2010.

\bibitem {Merle}
F.Merle.
\newblock Sur la non-existence de solutions positives d'{\'e}quations
  elliptiques surlin{\'e}aires.
\newblock {\em C. R. Acad. Sci. Paris S{\'e}r. I Math.}, 306(6):313--316, 1988.

\bibitem {MartelMerle}
Y.Martel and F.Merle.
\newblock Asymptotic stability of solitons of the g{K}d{V} equations with
  general nonlinearity.
\newblock {\em Math. Ann.}, 341(2):391--427, 2008.




\bibitem {NakanishiSchlag}
K.Nakanishi and W.Schlag.
\newblock Global dynamics above the ground state energy for the cubic nls
  equation in 3d, July 2010.



\bibitem {Ounaies}
H.Ounaies.
\newblock Perturbation method for a class of nonlinear {D}irac equations.
\newblock {\em Diff. Int. Eq.}, 13(4-6):707--720, 2000.





\bibitem {PelinovskyStefanov}
D.Pelinovsky and A.Stefanov.
\newblock Asymptotic stability of small gap solitons in the nonlinear dirac
  equations, arXiv:1008.4514, 2010.

\bibitem {Ranada}
A.F.Ranada.
\newblock {Classical nonlinear Dirac field models of extended particles}.
\newblock {\em Quantum theory, groups, fields and particles},  271--291, 1982.



\bibitem {ReedSimon4}
M.Reed and B.Simon.
\newblock {\em Methods of modern mathematical physics. {IV}. {A}nalysis of
  operators}.
\newblock Academic Press [Harcourt Brace Jovanovich Publishers], New York,
  1978.


\bibitem {shatah} J.Shatah.
\newblock  Stable standing waves of nonlinear Klein-Gordon equations.
 \newblock {\em Com. Math. Phys.}, 91:313--327, 1983.

\bibitem {shatahstrauss} J.Shatah and W.Strauss.
 \newblock {Instability of  nonlinear bound states}. \newblock {\em Com. Math.
Phys.},   100:173--190,  1985.



\bibitem {Sogge}
C.D.Sogge.
\newblock {\em {Lectures on nonlinear wave equations}}.
\newblock International Press Boston, 1995.

\bibitem  {Soler}
M.Soler.
\newblock {Classical, stable, nonlinear spinor field with positive rest
  energy}.
\newblock {\em Phys. Rev. D}, 1(10):2766--2769, 1970.

\bibitem {SmithSogge}
H.Smith and C.Sogge.
\newblock Global {S}trichartz estimates for nontrapping perturbations of the
  {L}aplacian.
\newblock {\em Com. Part. Diff. Eq.}, 25(11-12):2171--2183,
  2000.



\bibitem {SaitoUmeda}
Y.Sait{\=o} and T.Umeda.
\newblock The zero modes and zero resonances of massless {D}irac operators.
\newblock {\em Hokkaido Math. J.}, 37(2):363--388, 2008.

\bibitem {StraussVazquez}
W.Strauss and L. V{\'a}zquez.
\newblock Stability under dilations of nonlinear spinor fields.
\newblock {\em Phys. Rev. D (3)}, 34(2):641--643, 1986.

\bibitem {SofferWeinstein}
A.Soffer and M.I.Weinstein.
\newblock Multichannel nonlinear scattering theory for nonintegrable equations.
\newblock {\em Com. Math. Phys.}, 133(1):116--146, 1990.





\bibitem {SofferWeinstein2}
A.Soffer and M.I.Weinstein.
\newblock Multichannel nonlinear scattering for nonintegrable equations. {II}.
  {T}he case of anisotropic potentials and data.
\newblock {\em J. Diff. Eq.}, 98(2):376--390, 1992.

\bibitem {Taylor}
M.Taylor.
\newblock {\em Partial differential equations.}, volumes 115--117 of {\em
  Applied Mathematical Sciences}.
\newblock Springer-Verlag, New York, 1996.


\bibitem {Thaller}
B.Thaller.
\newblock {\em The {D}irac equation}.
\newblock Texts and Monographs in Physics. Springer-Verlag, Berlin, 1992.\

\bibitem  {tsaiyau}
T.P.Tsai, H.T.Yau, \newblock   Asymptotic dynamics of nonlinear
Schr\"odinger equations: resonance dominated and radiation dominated
solutions. \newblock {\em Com. Pure Appl. Math.}  {55}153--216, 2002.


\bibitem {Weinstein}
M.I.Weinstein.
\newblock Modulational stability of ground states of nonlinear
  {S}chr{\"o}dinger equations.
\newblock {\em SIAM J. Math. Anal.}, 16(3):472--491, 1985.

\bibitem {Weinstein2}
M.I.Weinstein.
\newblock Lyapunov stability of ground states of nonlinear dispersive evolution
  equations.
\newblock {\em Com. Pure Appl. Math.}, 39(1):51--67, 1986.

\end{thebibliography}
\def\cprime{$'$}

Laboratoire de math\'ematiques,  UFR Sciences et Technicques, Universit\'e de
Franche-Comt\'e, 16, route de Gray, 25030 Besan\c{c}on, France

{\it E-mail Address}: {\tt nabile.boussaid@univ-fcomte.fr}

Department of Mathematics and Geosciences,  University of Trieste, Via Valerio  12/1  Trieste,
34127  Italy.

{\it E-mail Address}: {\tt scuccagna@units.it}

\end{document}